\documentclass[a4paper,12pt,final]{article}

\usepackage[margin=25.4mm]{geometry}
\usepackage[T1]{fontenc}

\usepackage{bm,bef_alex}
\def\FG{\bm}
\def\mafo{\mathrm}
\usepackage{enumitem}
\usepackage{epic,eepic,textcomp,bbm}
\usepackage{xcolor,scrtime,graphicx,mdframed}
\usepackage{array,amsmath,amsfonts,amssymb,eucal,mathrsfs,wasysym,upgreek}
\usepackage{accents,bbold,scalerel}

\usepackage{comment}  
\excludecomment{GSproof}
\includecomment{GSreplace}
\usepackage{mathtools}

\newcommand{\Ld}{\!\;\calL^d}

\newcommand{\TRAGRO}{transport-growth}

\newcommand{\COLORW}{\begin{mdframed}[backgroundcolor=magenta!20,%
                   linewidth=0mm, outermargin=-10mm, innermargin=-5mm]\footnotesize}
\newcommand{\COLORWend}{\end{mdframed}}

\usepackage{tikz}
   \usetikzlibrary{calc}

\newcommand{\overbow}[1]{
   \tikz [baseline = (N.base), every node/.style={}] {
      \node [inner sep = 0pt] (N) {$#1$};
      \draw [line width = 0.6pt] plot [smooth, tension=1.2] coordinates {
         ($(N.north west) + (0.5ex,0.2ex)$)
         ($(N.north)      + (-0.03,0.6ex)$)
         ($(N.north east) + (-0.5ex,0.3ex)$)};
      \filldraw ($(N.north east) + (-0.4ex,0.4ex)$) --  ($(N.north east)
        + (-0.2ex,0.1ex)$) -- ($(N.north east) + (-0.6ex,0.15ex)$) -- cycle;}}

\usepackage{todonotes}

\newcommand{\ul}[1]{\underline{#1}}

\newcommand{\EEE}{\color{black}}

\numberwithin{equation}{section}

\usepackage[color,notref,notcite]{showkeys} 
\definecolor{refkey}{gray}{.8}   
\definecolor{labelkey}{rgb}{0.9,0.2,0.2} %

\makeatletter
\renewcommand*\env@cases[1][1.2]{%
  \let\@ifnextchar\new@ifnextchar
  \left\lbrace
  \def\arraystretch{#1}%
  \array{@{}c@{\quad}l@{}}%
}
\makeatother

\newcommand{\GeoCov}[1]{{#1}^{\Geod}}
\newcommand{\YX}{X}

\newcommand{\LET}{\mathsf{L\!E\!T}}
\renewcommand{\bfH}{H}
\renewcommand{\d}{\mathrm{d}}
\renewcommand{\bfLambda}{\bm{\Lambda}}
\newcommand{\mathbcal}[1]{{\bm#1}}
\renewcommand{\dd}{\mathbcal{d}}
\newcommand{\xx}{\mathbcal{x}}
\newcommand{\zz}{\mathbcal{z}}

\newcommand{\dom}{{\mathrm{dom}}}
\newcommand{\ddd}{\,\mathrm{d}}
\newcommand{\mm}{\boldsymbol{\upmu}}
\newcommand{\lala}{\boldsymbol{\uplambda}}
\def\mfC{\mathfrak C}
\newcommand{\sfd}{\mathsf d}  
\newcommand{\sfF}{\mathsf F}  
\newcommand{\sfG}{\mathsf G}  
\newcommand{\dist}{\operatorname{dist}}
\newcommand{\supp}{\operatorname{supp}}

\newcommand{\A}{A}
\newcommand{\B}{B}
\newcommand{\M}{\mathcal{M}}
\newcommand{\E}{\mathsf{E}}
\newcommand{\Er}{\Delta}
\newcommand{\ob}{x_{\mathsf{ob}}}  
\newcommand\geo[1]{\!\;\overbow{{#1}}\;\!}
\newcommand{\invd}{\mathbb{\Delta}}
\newcommand{\Geod}{\operatorname {Geod}}
\newcommand{\sfW}{\mathrm W}
\newcommand{\HK}{\mathsf{H \hspace{-0.27em} K}}
\newcommand{\SHK}{\mathsf{S\hspace{-0.18em} H\hspace{-0.27em} K}}
\newcommand{\He}{\mathsf{He}}
\newcommand{\SHe}{\mathsf{S\hspace{-0.18em}He}}

\newcommand{\last}{\lambda^{\hspace{-2pt}\ast}}

\makeatother

\newtheorem{mainresult}{Main Result}

\newtheorem{assumption}{Assumption}

\newtheorem{openquestion}[theorem]{Open Question}
\begin{document}

\title{Evolutionary Variational Inequalities on\\  the 
       Hellinger-Kantorovich and \\
       Spherical Hellinger-Kantorovich spaces}

\author{%
Vaios Laschos\thanks{WIAS Berlin.\ Partially supported by DFG  under Germany's Excellence Strategy – The Berlin Mathematics
Research Center MATH+ (EXC-2046/1, project ID: 390685689). } 
  \and 
Alexander Mielke\thanks{WIAS Berlin and Institut f\"ur Mathematik, Humboldt
  Unversit\"at zu Berlin. Partially supported by DFG under Priority Program
  \emph{Variational Methods for Predicting Complex Phenomena in Engineering
    Structures and Materials} (SPP-2256, project ID: 441470105, subproject
  Mi 459/9-1) }   
}

\date{22 December 2023}

\maketitle


{\footnotesize\def\contentsname{}\vspace{-4em} \tableofcontents}

\vspace{2cm}
\begin{abstract}
  We study the minimizing movement scheme for families of geodesically
  semiconvex functionals defined on either the Hellinger-Kantorovich space or
  on the Spherical Hellinger-Kantorovich space. By exploiting some of the
  finer geometric properties of the spaces (namely the local-angle condition
  and the semiconcavity of the squared distance, when
  restricted to suitable subsets), we prove that the sequence of curves,
  which are produced by geodesically interpolating the points generated by the
  minimizing movement scheme, converges to a curve that satisfy the Evolutionary
  Variational Inequality (EVI), when the time step goes to $0$. Under suitable
  conditions, we obtain a global EVI flow on the whole space.
\end{abstract}

Hellinger-Kantorovich space, Spherical Hellinger-Kantorovich space, \\
density estimates, local-angle condition, semiconcavity of the squared distance.

\section{Introduction}

Let $\YX$ be a geodesic metric space and $\calM(\YX)$ the space of all
nonnegative and finite Borel measures on $\YX.$ Independently, in
\cite{CPSV15IDOT}, \cite{KMV16b}, and 
\cite{LiMiSa16OTCR, LiMiSa18OETP}, the space
$(\calM(\YX),\HK)$ was introduced and studied, where $\HK$ denotes the
Hellinger-Kantorovich or Wasserstein-Fisher-Rao distance. In
\cite{LiMiSa16OTCR,LiMiSa18OETP}, it was proved that $(\calM(\YX),\HK)$ is a
geodesic space itself and all geodesic curves were characterized.  In
\cite{LasMie19GPCA}, the spherical Hellinger Kantorovich distance $\SHK$ was
introduced and it was proved that the set of all probability measures
$ \calP(\YX) = \bigset{\mu\in \calM(\YX)}{ \mu(\YX)=1}$ endowed with $\SHK$ is
also a geodesic metric space.

For the rest of the paper, we assume that $\YX\subset\mathbb{R}^{d}$ is a 
compact, convex set with nonempty interior, which allows us to exploit the
characterization of geodesic semiconvexity for $\HK$ provided recently in
\cite{LiMiSa23FPGG}. We 
introduce a family of entropy functionals, i.e.
\begin{equation}
\label{EntropyFunctional}
\E(\mu)=\int_{\YX}E(\rho(x))\calL^d (\d x) + E'_{\infty}\rmd\mu^{s},
\hspace{16pt} \mu =\rho\calL^d +\mu^{s}\hspace{8pt}\text{and}\hspace{8pt}
\mu^{s}\bot\calL^{d}, 
\end{equation}
where
$E'_{\infty}=\lim_{t\rightarrow \infty}\frac{E(t)}{t}=\lim_{t\rightarrow
  \infty}E'(t),$ and $\mu^{s}$ the singular part of $\mu$ with respect to
$\calL^d,$ i.e. the Lebesgue measure restricted at $\YX.$ In this paper we are
going to study De Giorgi's \emph{minimizing movement (MM) scheme}, also known
as JKO scheme (after \cite{JoKiOt98VFFP}) in the case of the Wasserstein space,
\begin{equation}  
 \label{scheme}
 \mu_{1}=\inf_{\mu \in\calM(\YX)}\left\{\frac{\HK^{2}(\mu_{0},\mu)}{2\tau}+\E(\mu)\right\}
\qquad 
\mu_{1} =\inf_{\mu \in \calP(\YX) } \left\{\frac{\SHK^{2}(\mu_{0},\mu)}{2\tau}+\E(\mu)\right\},
\end{equation}
for the Hellinger-Kantorovich and Spherical Hellinger-Kantorovich space
respectively. We are going to limit our exploration to cases where the
functionals $\E$ are of the form \eqref{EntropyFunctional}, and satisfy the
following basic convexity assumptions.  \renewcommand\theassumption{A}
\begin{assumption}
 \label{BasicAssumpt}
\leavevmode\vspace{-0.5em}
\begin{enumerate}\itemsep-0.3em
\item $E:\mathbb{R}^{+}\rightarrow\mathbb{R}$ is a convex function. 
\item $\E$ is geodesically $\lambda$-convex for some $\lambda \in \mathbb{R}$.
\end{enumerate}
\end{assumption}
We note that by \cite[Theorem 5.2]{Ambro}, the functional $\E$ is lower
semicontinuous, and it is the relaxation of itself when defined only on
$(\calM_\text{ac}(X),\HK)$ or on $(\calP_\text{ac}(X),\SHK)$. 

A metric space
along with a lower semicontinuous functional define a metric gradient
system,  cf.\ e.g.\ \cite{Miel23IAGS}: 

\begin{definition}
Let $(\calX,\sfd_{\calX})$ be a metric space and $\phi:\calX \to (-\infty,\infty]$ a
lsc functional, then $(\calX,\phi,\sfd_{\calX})$ is called a \emph{metric gradient
  system}.
\end{definition}

The main goal is to show that geodesic interpolation of points that are
generated by the MM scheme give rise to sequences of curves with good limiting
properties. More specifically, we show that such sequences of curves converge,
when $\tau$ converges to $0$, to curves that satisfy the Evolutionary
Variational Inequalities (EVI) for the metric gradient system
$(\calM(\YX),\E,\HK)$ or $(\calP(\YX),\E,\SHK)$, respectively.

Before we proceed, we briefly remind the reader of the definition of EVI. It
involves the upper right Dini derivative 
\[
\frac{\rmd^+}{\rmd t } \, \zeta(t) := \limsup_{h\to 0^+} \frac1h\left(
\zeta(t{+}h) - \zeta(t)\right). 
\]

\begin{definition}[EVI solutions for metric gradient systems]
\label{EVI-diff}
Let  $(\calX,\phi,\sfd_{\calX})$ be a \emph{metric gradient system}.  For $T\in (0,\infty)$ and $\lambda\in\R$ we say that a continuous
curve $\xx:[0,T) \to \calX$ is an \emph{EVI$_\lambda$ solution} for the metric
gradient system $(\calX,\phi,\sfd_{\calX})$, if $\phi(\xx(t))<\infty$ for all
$t\in (0,T)$ and for every ``observer'' $\ob \in\calX,$ we have
\begin{equation}
\label{eq:EVI.diff}
  \frac{\rmd^{+}}{\rmd t}\frac{1}{2}\sfd_{\calX}^{2}(\xx(t),\ob )
   + \frac{\lambda}{2}\sfd_{\calX}^{2}(\xx(t),\ob )
 \leq \phi(\ob )-\phi(\xx(t)) \quad \text{for all }t\in (0,T).
\end{equation} 
If furthermore $T=\infty,$ we call $\xx$ a \emph{complete EVI$_\lambda$
  solution.}
\end{definition}

EVIs are used to provide a generalization of the definition of gradient flows
in the more abstract setting of geodesic metric spaces, see
\cite{AmGiSa05GFMS}. For a nice exposition on EVIs the reader is advised to
follow the trilogy of papers \cite{MurSav20GFEV, MurSavII+III}. Our main
result (cf.\ Theorem \ref{mainmain}) relies on some of these results and reads as
follows: 

\renewcommand\themainresult{}
\begin{mainresult}\label{Main Result}
  Let $X \subset \R^d$ be a compact, convex set with nonempty interior.
  Furthermore, let $\E$ be of the form \eqref{EntropyFunctional} and let it
  satisfy Assumption \ref{BasicAssumpt}. \smallskip
  
  Then, for all $\mu_0=\rho_0\calL^d$ with
  $0<\ul\rho_0 \leq \rho_0(x) \leq \ol\rho_0< \infty$ a.e.\ in $\YX$, the
  geodesically interpolated solutions of the MM schemes corresponding to the
  gradient system $(\calM(X),\E,\HK)$ (or $(\calP(X),\E,\SHK)$), as in
  $\eqref{scheme}_{\HK}$ (or $\eqref{scheme}_{\SHK}$), converge to a complete
  solution $\bm{\mu}:[0,\infty)\to \calP(\YX)$ of EVI$_\lambda$.  Moreover, for
  all $\mu_0 \in \ol{\dom(\E)}^{\HK} \subset \calM(X)$ (or
  $\mu_{0}\in \ol{\dom(\E)}^{\SHK}\subset \calP(X)$) there exists a unique EVI
  solution, emanating from $\mu_{0}$. \smallskip
\end{mainresult} 

We emphasize that the generation of EVI solutions presented in
\cite[Thm.\,4.0.4]{AmGiSa05GFMS} is not applicable in our case, because it
strongly relies on the uniform generalized convexity assumption (cf.\
\cite[Ass.\,4.0.1]{AmGiSa05GFMS}) for the functional
$v\mapsto \frac1{2\tau} \sfd(w,v)^2 + \E(v)$ for $\sfd =\HK$ or $\SHK$,
respectively. So far, the existence of suitable ``generalizations of geodesic
curves'' as used in the Kantorovich-Wasserstein setting is not known, and there
are no ideas how to approach this question in the case of unbalanced transport.

Thus, we follows the completely different route that is devised in
\cite{Sava07GFDS,Sava11?GFDS,MurSav20GFEV,MurSavII+III}. This methods relies on
the two independent assumptions of (i) semiconvexity of the functional and (ii)
additional properties of the geodesic space $(\calX,\sfd)$, namely the
Local-Angle Condition (LAC) and the semiconcavity of the squared distance,
i.e.\ $\frac12\sfd^2(x_*,\cdot)$. As these two fundamental properties are not
so well-known in the theory of gradient systems, we discuss a few examples in
Section \ref{su:Exa.LAC.k-concave}. The semiconcavity for $\HK$ and $\SHK$ is
only true when working on subsets with suitably upper and lower density bounds,
see \cite{LasMie19GPCA}. Our work adapts the results in \cite{MurSavII+III} by
suitably accounting for the interplay of localization and semiconcavity with
varying parameter $\kappa$, where we can use the absolutely non-trivial result
from \cite{MurSavII+III} (see \cite[Thm.\,7]{Sava07GFDS} for the first
announcement) that the resulting EVI$_\lambda$ is independent of the
semiconcavity parameter $\kappa$.

A family of functionals $\E$ satisfying Assumption \ref{BasicAssumpt} on the
Hellinger-Kantorovich space $(\calM(X),\E)$ is the following.

\begin{example}[The case $(\calM(X),\E,\HK)$]
\label{example}
Consider $\E_{\alpha,m}^{\gamma}$ generated by
$E_{\alpha,m}^{\gamma}(c)=\alpha c^m+\gamma c$ with $ \alpha>0,\ m>1$, and
$\gamma\in\mathbb{R}$. Then, according to \cite[Sec.\,7]{LiMiSa23FPGG} we know
that $\E$ is geodesically $\lambda$-convex on $(\calM(X),\HK)$ with 
$\lambda =2\gamma$. 
\end{example}

Another example of functionals satisfying Assumption \ref{BasicAssumpt}, on the Hellinger-Kantorovich space but also on the Spherical
Hellinger Kantorovich space follows.

\begin{example}[Both cases $(\calM(X),\E,\HK)$ and $(\calP(X),\E,\SHK)$]
\label{example2} 
Let $E(c)=-c^{q}$ and assume that either $d=1$ and $q \in [1/3, 1/2]$ or
$d=2$ and $q=1/2$, then \cite[Sec.\,7]{LiMiSa23FPGG} ensures that $\E$ is
geodesically $0$-convex on  $(\calM(X),\HK)$, and our Proposition \ref{pr:q.leq.1}
gives the same for $(\calP(X),\SHK)$. 
\end{example}

\begin{openquestion}\label{OQ.1}
  In \cite{LiMiSa23FPGG} the functionals $\E$ of the form
  \eqref{EntropyFunctional} that are semiconvex in $(\calM(X),\HK)$ 
  were fully characterized. However, for the case of the Spherical
  Hellinger-Kantorovich space, only very few semiconvexity results are known,
  and these results are corollaries of general theorems that connect metric
  spaces with their spherical counterparts, see Proposition \ref{pr:q.leq.1}. A
  general characterization of geodesically semiconvex functionals on the
  Spherical Hellinger-Kantorovich space is still elusive. We leave this as an
  open question, and we welcome any suggestion for collaboration in this
  direction.
\end{openquestion}

Unlike with other definitions of gradient flows, the EVI approach guarantees
some useful properties. One of the most important is the asymptotic stability
for sequences of curves that satisfy EVI (see
\cite[Sec.\,2.6]{DanSav14LNGF} or \cite[Thm.\,5.12]{Miel23IAGS}). More
specifically, under weak convergence assumptions for a sequence of functionals
$\sfG_{k},$ to some $\sfG_{\infty},$ we get for free that the sequence of
solutions to the respective EVI converge to a solution of EVI with respect to
the limit functional. One can easily show that the limit case EVI for
$\sfG_{m}=\E_{\alpha,m}^{\gamma}$ for $m \to \infty,$ where
$\E_{\alpha,m}^{\gamma}$ as in Example \ref{example}, corresponds to the
functional
\begin{equation}\label{tumor}
\E^{\gamma}_{\infty}(\mu)=\begin{cases}\gamma\mu(\YX) 
 &\text{for } \mu=\rho\calL \hspace{8pt} \text{with}\hspace{8pt} 
      \rho(x)\leq 1 \text{ a.e.},\\
\infty & \text{otherwise}, 
\end{cases}
\end{equation}
which was studied independently in \cite{GLM19} and \cite{DimChi20TGMH}.

One of the primary motivations for studying the gradient of entropy functionals
on spaces of measures is their frequent identification as solutions to
well-known partial differential equations (PDEs). This connection was initially
demonstrated for the Fokker-Planck and diffusion equations, which were proven
to be gradient flows of the differential entropy on the Wasserstein space. For
further details and results, please refer to \cite{JoKiOt98VFFP, Otto01GDEE,
  LiMiSa16OTCR}.  Further applications in \cite{PeQuVa14HSAM,
  DimChi20TGMH} involve reaction-diffusion equations where the reaction term
exactly correspond to the Hellinger part in $\HK$. \medskip

Our paper is divided into  six sections. In Section \ref{se:EVI.vs.PDE} we
discuss the differences between our EVI approach and other works that provide
weak solutions to the corresponding partial differential equations.  Gradient
flows for $\HK$ were already studied in \cite{KMV16a, KMV16b, GalMon17JKOS, GLM19,
  DimChi20TGMH, Flei20?MMAC} and in \cite{KV19} for $\SHK$. Our approach is
more restrictive as it is based on geodesic semiconvexity, but yields
uniqueness of solutions and Lipschitz continuous dependence on the initial
data, while uniqueness may fail for weak PDE solutions. 

In Section \ref{se:HK}, we present
various equivalent definitions of $\HK$ and discuss relevant lemmas. In Section
\ref{se:density bounds}, we prove that if $\mu_0$ possesses desirable density
bounds, then the minimizers $\mu_1$ in both JKO schemes \eqref{scheme} also
exhibit favorable density bounds. Notably, for the $\SHK$ space, we establish a
discrete maximal principle.

Section \ref{se:GenSavare's method} first introduces the general theory
around the Local-Angle Condition (LAC) and the semiconcavity of the squared
distance together with some examples in Section
\ref{su:Exa.LAC.k-concave}. Then, an abstract existence theorem is
developed for EVI (Evolution Variational Inequality) solutions, the proof of
which is based on LAC and $\kappa$-concavity. 
\begin{GSreplace}
Here, we use a series of results from \cite{MurSavII+III} which are cited in
full detail but without proof, except for our abstract existence result
in Theorem \ref{th:MainExistTheo} that relies on the density of $\cup
A_\kappa$, where $A_\kappa$ are suitable subsets of $\calX$ in which the
squared distance is $\kappa$-concave. 
\end{GSreplace} 
\begin{GSproof}
\COLORW
Here we provide a full proof of our abstract existence result
in Theorem \ref{th:MainExistTheo} that relies heavily on the theory developed
in \cite{Sava11?GFDS, MurSavII+III}. 
\COLORWend
\end{GSproof}
Our abstract existence result extends the approach provided in
\cite{MurSav20GFEV, MurSavII+III} of proving EVI solutions for $\lambda$-convex
functionals to situations where $\kappa$-concavity of the squared distance is
not true globally. We only need that $\kappa$-concavity holds only in
suitable subsets $A_\kappa$ instead of the whole space. We reckon that this
localization approach is an
interesting extension on its own. 

In Section \ref{se:SemicEVI} we then show that the necessary $\kappa$-concavity
for $(\calM(X),\HK)$ and $(\calP(X),\SHK)$ can be obtained from the 
theory developed in \cite{LasMie19GPCA}. Combining this with the theory of
geodesic convexity of \cite{LiMiSa23FPGG} then allows us to establish our
main existence results for EVI solutions in the spaces $(\calM(X),\HK)$ and
$(\calP(X),\SHK)$.

\section{Evolutionary variational inequalities versus PDEs}
\label{se:EVI.vs.PDE}

Evolutionary variational inequalities are defined for metric gradient systems
$(\calX,\E,\sfd)$, where $(\calX,\sfd)$ is a complete geodesic space and $\E:
\calX\to \R\cup\{\infty\}$ a
geodesically $\lambda$-convex functional. A curve $u:{[0,\infty[}\to \calX$ is
called an EVI$_\lambda$ if it satisfies 
\begin{equation}
  \label{eq:EVI1}
  \begin{aligned}
&\forall \,s,t\geq 0 \text{ with }s<t\ \forall\, w \in\mathrm{dom}(\E):
\\
&\frac12 \sfd^2(u(t),w) \leq \frac12 \ee^{-\lambda(t-s)} \sfd^2(u(s),w) +
M_\lambda(t{-}s) \big( \E(w)- \E(u(t)\big),
\end{aligned}
\end{equation}
where $M_\lambda(r)= \int_0^r \ee^{-\lambda s} \,\d s$. See Proposition \ref{pr:EVI-int}
and \cite{MurSav20GFEV} for several equivalent formulations. 

To have specific example in mind we consider $(\calX,\sfd)= (\calM(X),\HK)$ and
the functional $\E$ from \eqref{EntropyFunctional} with
\begin{equation}
  \label{eq:E.special}
  E(c) = - \sqrt{c} + \frac13\, c^{3/2} \quad \text{for } c\geq 0.
\end{equation}
We restrict to space dimensions $d=1$ or $2$ and observe that $\E$ is
geodesically convex (cf.\ Section \ref{su:HK.EVI.flow} or \cite{LiMiSa23FPGG}.
Hence, our main existence results in Theorem \ref{mainmain} applies and we have
a unique EVI solution for all $\mu_0 \in \calM(X)$. Here, we use that
$\mathrm{dom}(\E)= \rmL^{3/2}(X)$ is dense in $(\calM(X),\HK)$.

We claim that the unique solution starting in $\mu_0=0$ is given by the curve 
\[
\wt\mu(t) =\rho(t) \Ld(\d x) \quad \text{ with }\rho(t)=\tanh^2(t). 
\]
It is difficult to check
that this $\mu$ satisfies \eqref{eq:EVI1}. However, \cite[Thms.\,3.5,
cf.\,(3.17)]{MurSav20GFEV} proves that all EVI solutions $\mu:{[0,\infty[}\to
\calM(X)$ are curves of maximal
slope , i.e.\ 
\begin{equation}
  \label{eq:CurvMaxSlope}
  \frac\d{\d t} \E(\mu(t)) = - \frac12\big|\dot{\mu}\big|^2_\HK(t) -
  \frac12 \big| \pl \E\big|^2_\HK(\mu(t)) \quad \text{for }t>0.
\end{equation}
More importantly. \cite[Thms.\,4.2]{MurSav20GFEV} states the reverse:  
if the EVI flow exists on $\ol{\dom(\E)}$ then every curve of maximal slope
is an EVI solution. Thus, $\wt\mu$ the unique EVI solution starting at
$\mu_0=0$, if it is a curve of maximal slope for $\calM(X),\E,\HK)$.  

To check this, we first observe that
$[0,S_*] \ni s \mapsto s^2 \mu_*$ is a $\HK$ geodesics for all $\mu_*\in
\calM(X)$ one finds $\HK(r_0^2\mu_*,r_1^2\mu_*)
=|r_0{-}r_1|\mu_*(X)^{1/2}$. Applying this to $\wt \mu$ we find $
\HK(\wt\mu(t_1),\wt\mu(t_2)) = \big|\tanh(t_1)- \tanh(t_2)\big|\,\Ld(X)^{1/2}$
and conclude that the metric speed satisfies 
\[
|\dot{\wt\mu}|^2_\HK(t)=
(\tanh'(t))^2\Ld(X)=\big(1{-}\tanh^2(t)\big)^2 \Ld(X).
\] 

Next we observe that $\E(\wt\mu(t)) = \Ld(X) (\frac13 \tanh^3(t)-\tanh t)$
such that 
\[
\frac{\d}{\d t} \E(\wt\mu(t)) = - \tanh'(t) \big( 1{-} \tanh^2(t)\big)=
-\big(1-\tanh^2(t)\big)^2. 
\]
Finally, from $\xi= \rmD \E(\wt\mu(t))$ and
$|\pl\E|^2_\HK = \int_X \big( \rho|\nabla \xi|^2+ 4\rho \xi^2\big) \d x $, we
obtain, with $\xi = E'(\rho) =\frac12 (\rho(t){-}1)/(2\sqrt{\rho(t)})$, the
metric slope
\[
\big| \pl \E\big|^2_\HK(\wt\mu) =4\rho \Big( \frac{\rho{-}1}{2\sqrt\rho}
\Big)^2 = (\rho(t){-}1)^2=\big(1{-}\tanh^2(t)\big)^2.  
\]
Thus, $\wt\mu$ is shown to be a curve of maximal slope, which is unique because
of the uniqueness of EVI solutions. 

Of course, the EVI formulation can be compared to corresponding weak
formulations of the associated partial differential equation occurring to as
gradient-flow equations. Generalizing the above energy to a sum of an internal
energy $\E$ and a potential energy $\mu \mapsto -\int_X V \mu(\d x)$ we consider
the gradient systems $(\calM(X),\mathsf F, \HK_{\alpha})$ and 
$\calP(X),\mathsf F,\SHK_{\alpha,\beta})$ with $\mathsf
F(\mu)=\E(\mu)-\int_X V\mu(\d x)$. Here, $\HK_{\alpha,\beta}$ and
$\SHK_{\alpha,\beta}$ are scaled versions of $\HK=\HK_{1,4} $ and
$\SHK = \SHK_{1,4}$, respectively, that are defined via the Onsager
operators
\begin{align*}
\bbK^\HK_{\alpha,\beta}(\mu) \,\xi 
&= - \DIV\left( \alpha \rho \nabla \xi\right) + \beta \rho \xi,
\\
\bbK^\SHK_{\alpha,\beta}(\mu) \,\xi 
&= - \DIV\left( \alpha \rho \nabla \xi\right) + \beta  \rho \left( \xi -\ts
\int_X \rho\xi \d x\right).
\end{align*} 
Assuming further that $ \Ld $ is the $d$-dimensional Lebesgue measure
(restricted to $X$) such that $\E(\mu)= \int_X E(\rho)\d x$ we have the derivative $\rmD
\E(\rho\d x)= E'(\rho)$ and the Otto calculus (see \cite{Otto01GDEE,
  AmGiSa05GFMS, LieMie13GSGC, LiMiSa16OTCR}) leads to the gradient-flow equation
\begin{align*}
\dot \rho&= - \bbK^\HK_{\alpha,\beta}(\mu)\,\rmD \mathsf F(\mu)= \alpha 
\DIV\!\big( \rho\, \nabla (E'(\rho){-}V) \big)  - \beta \rho (E'(\rho){-}V),
\\
\dot \rho&= - \bbK^\SHK_{\alpha,\beta}(\mu)\,\rmD \mathsf F(\mu)= \alpha 
\DIV\!\big( \rho \,\nabla(E'(\rho){-}V)\big) - \beta \rho
\left(E'(\rho){-}V 
- \ts \int_X \rho (E'(\rho){-}V) \d x \right) .
\end{align*} 

The existence of weak solutions to the above reaction-diffusion equations
was established for various cases already. The gradient structure in
$(\calM(X),\HK)$ was first exploited in \cite{KMV16a,KMV16b}, for models in
which existence of solutions follows from classical PDE techniques.  In
\cite{GalMon17JKOS} the convergence of a time-splitting scheme is shown, were
minimizing movements steps are alternating between the Kantorovich-Wasserstein
distance and the Hellinger-Fisher-Rao distance. In \cite{Flei20?MMAC} the
static characterization of the $\HK$ distance from \cite{LiMiSa16OTCR,
  LiMiSa18OETP} is employed for showing that the curves generated by the JKO
scheme converge to solutions of the above reaction-diffusion PDE.

The gradient-flow equations for $(\calP(X),\mathsf F,\SHK)$ were studied in
\cite{KV19}, where existence is established via approximation by classical
solutions and suitable a priori estimates. Moreover, convergence into steady
states as well as functional inequalities are obtained.\medskip

The special gradient system $(\calM(X),\E,\HK)$ with $E$ from \eqref{eq:E.special}, $
V\equiv 0$, and $(\alpha,\beta)=(1,4)$ leads to the above EVI from the
beginning of this section. Recall that this EVI has a unique
solution starting at $\mu=0$. The associated PDE for this case reads 
\begin{equation}
  \label{eq:PDE.special}
  \dot\rho= \Delta \big( \frac12\sqrt\rho + \frac16 \,\rho^{3/2}\big) + 2
  \,\big( \sqrt\rho - \rho^{3/2}\big) \text{ in } X, \quad \nabla\rho\cdot\nu =
  0 \text{ on }\pl X,
\end{equation}
which is indeed a parabolic equation (with fast diffusion). However, this
equation is non-Lipschitz at $\rho=0$ and admits a one-parameter family of
spatially constant solutions, namely for $\zeta \geq 0$
\[
\rho^{(\zeta)}(t)=\left\{ \ba{cl}\tanh^2(t{-}\zeta)& \text{for } t\geq \zeta,
  \\
0 &\text{for }t\in [0,\zeta].  \ea \right. 
\]
Clearly, the solutions are not curves of maximal slope because they are not EVI
solutions (cf.\ \cite[Thms.\,3.5+4.2]{MurSav20GFEV}). This, we see a clear
difference between the set of EVI solutions and PDE solutions. It is expected,
but still unproved, that the generalized minimizing movements obtained in
\cite{Flei20?MMAC} coincide with the EVI solutions.

\section{The metric spaces $(\calM(X),\HK)$ and $(\calP(X),\SHK)$}
\label{se:HK}

\subsection{Notation and preliminaries}
\label{subsec:notation}

We will denote by $\calM(\YX)$ the space of all nonnegative and finite Borel
measures on $\YX$ endowed with the weak topology induced by the duality with
the continuous and bounded functions of $\rmC_b(\YX)$. The subset of measures
with finite quadratic moment will be denoted by $\calM_2(\YX)$. The spaces
$\calP(\YX)$ and $\calP_2(\YX)$ are the corresponding subsets of probability
measures. If $\mu\in \calM(\YX)$ and $\bfT:\YX\to \YX$ is a Borel map, $\bfT_\sharp \mu$
will denote the push-forward measure on $\calM(\YX)$, defined by
\begin{equation}
  \label{eq:push_forw}
\bfT_\sharp\mu(B):=\mu(\bfT^{-1}(B))\quad
\text{for every Borel set $B\subset \YX$}.
\end{equation}
We will often denote elements of $\YX\times\YX$ by $(x_{0},x_1)$ and the
canonical projections by $\pi^i:(x_{0},x_1)\to x_i$ for $i=0,1$.  A transport
plan on $\YX$ is a measure $M_{01}\in \calM(\YX{\times} \YX)$ with marginals
$\mu_i:=\pi^i_\sharp M_{01}$.

\subsection{The logarithmic-entropy transport formulation}
\label{su:LogEntrF}

Here we first provide the definition of the $\HK(\mu_0,\mu_1)$ distance in
terms of a minimization problem that balances a specific transport problem of
measures $\sigma_0\mu_0$ and $\sigma_1\mu_1$ with the relative entropies of
$\sigma_j\mu_j$ with respect to $\mu_j$. For the characterization of the
Hellinger--Kantorovich distance via the static Logarithmic-Entropy Transport
(LET) formulation, we define the logarithmic entropy density
$F:\left[0,\infty\right[\to[0,\infty[$ via $F(r) =r \log r - r+1 $ and the cost
function $\ell:\left[0,\infty\right[\to[0,\infty]$ via
$\ell(R) = -2 \log\left(\cos\left(R\right)\right)$ for $R<\frac{\pi}{2}$ and
$\ell\equiv+\infty$ otherwise. For given measures $\mu_0,\mu_1$ the LET
functional $\LET(\,t\,;\mu_0,\mu_1): \calM(X\times X)\to[0,\infty[$ reads
\begin{equation}\label{def:let}
\LET(\bfH_{01};\mu_0,\mu_1):=
\int_X F(\sigma_0)\d\mu_0+ \int_X F(\sigma_1)\d\mu_1+
\iint_{\YX\times \YX}  \ell(\sfd_X(x,x_1)) \d\bfH_{01} 
\end{equation}
with $\eta_i:=  (\pi_i)_\sharp
\bfH_{01}=\sigma_i\mu_i\ll\mu_i$.  With this, the equivalent
formulation of the Hellinger--Kantorovich distance as
entropy-transport problem reads as follows.

\begin{theorem}[{LET formulation, \cite[Sec.\,5]{LiMiSa18OETP}}]
\label{thm:LET}
For all $\mu_0,\mu_1\in \calM(\YX)$ we have
\begin{equation}\label{eq:HK.ET}
  \HK^2(\mu_0,\mu_1)=\min\big\{ 
  \LET(\bfH_{01};\mu_0,\mu_1)\,\big|\,
  \bfH_{01}\in\calM(\YX {\times} \YX),\  (\pi_i)_\sharp
  \bfH_{01}  \ll\mu_i\big\}.
\end{equation}
\end{theorem}

An optimal transport plan $\bfH_{01}$, which always exists, gives the
effective transport of mass. Note, in particular, that only
$\eta_i\ll\mu_i$ is required and the cost of a deviation of $\eta_i$
from $\mu_i$ is given by the entropy functionals associated with
$F$. Moreover, the cost function $\ell$ is finite in the case $
 \sfd_X(x_{0} , x_1)  <\frac{\pi}{2}$, which highlights the
sharp threshold between transport and pure absorption-generation
mentioned earlier.\medskip

Amongst the many characterizations of $\HK$ discussed in
\cite{LiMiSa18OETP} there is one that connects $\HK$ with 
the classical Kantorovich-Wasserstein distance on the cone $\mfC$
over the base space $(X,\sfd_{X})$ with metric 
\begin{equation}
\label{eq:7}
 \sfd^{2}_{\mfC}(z_0,z_1):=
r_0^2+r_1^2-2r_0r_1\cos_{\pi/2}\left(\sfd_{X}(x_{0},x_{1})\right),\quad
z_i=[x_i,r_i],
\end{equation} 
where  as above $\cos_{b}(a)=\cos(\min\{b,a\})$.
Measures in $\M(X)$ can be ``lifted'' to measures in $\M(\mfC)$, 
e.g.\ by considering the measure $\mu\otimes\delta_{1}$ for $\mu\in\M(X)$.
Moreover, we can define the projection of measures in $\M_2(\mfC)$
onto measures in $\M(X)$ via
\[
\mathfrak{P}:\left\{
\begin{array}{ccc}
\M_2(\mfC)&\to&\M(X),\\[0.2em]
\lambda&\mapsto&\int_{r=0}^\infty r^2\,\lambda(\cdot,\rmd r).\
\end{array}\right.
\]
For example, the lift
$\lambda = m_0\delta_{\{\boldsymbol{0}\}} +
\mu\otimes\frac{1}{r(\cdot)^2}\delta_{r(\cdot)}$, with $m_0\geq 0$ and
$r:\mathrm{supp}(\mu)\to\left]0,\infty\right[$ arbitrary, gives
$\mathfrak{P}\lambda= \mu$. Now, the cone space formulation of the
Hellinger--Kantorovich distance of two measures $\mu_0$, $\mu_1\in\M(X)$ is
given as follows.

\begin{theorem}[Optimal transport formulation on the cone]
\label{thm:OTcone}
For $\mu_0,\mu_1\in \M(X)$ we have
\begin{eqnarray*}   
&\HK^2(\mu_0,\mu_1)&= 
  \min\Big\{\sfW_{\sfd_{\mfC}}^{2}(\lambda_0,\lambda_1) \, \Big|\, 
    \lambda_i\in \mathcal{P}_2(\mfC),\ \mathfrak{P}\lambda_i = \mu_i\Big\} \\    
&&=\min\Big\{\iint_{\mfC\times\mfC} \sfd_{\mfC}^{2}(z_0,z_1) 
  \rmd  \Lambda_{01} (z_0,z_1)\,\Big|\, \pi^i_{\sharp}\Lambda_{01}=
 \lambda_i,~\text{and}~\mathfrak{P}\lambda_i=\mu_i\Big\}.
\end{eqnarray*}
\end{theorem}
 This result will be needed for proving $K$-semiconcavity in Theorem
\ref{thm:K.Semi}. 

\subsection{Transport-growth systems}\label{s:HK.Dilation-Transportation}

As a slight generalization of \cite[Def.\,2.7]{LiMiSa23FPGG} we define \TRAGRO\
couples which generalizes the transport map for the pure
Kantorovich-Wasserstein case and allows for growth (or decay) of mass as well.

\begin{definition}\label{dcouple}[Transport-growth couple]
A quintuple $(\nu,q_0,\bfT_0,q_1,\bfT_1)$ with
$\nu\in\calM(\calY)$, $\bfT_i:\calY\to \YX$, and $q_i\in\rmL^2(Y;\nu)$
is called \TRAGRO\ couple for $(\mu_0,\mu_1)$, if
\begin{equation}
\label{eq:dilationtransport}
(\bfT_i)_\sharp (q_i^2\nu)=\mu_i.
\end{equation}
is satisfied.  If the \TRAGRO\ couple, has the form
$(\mu_{0},1,\mathbf{I},q,\bfT ),$ then we call $(q,\bfT )$ a \TRAGRO\ couple
from $\mu_{0}$ to $\mu_{1}$. If for the \TRAGRO\ couple we have
\begin{equation}
\begin{split}
\label{eq:95}
\HK^2(\mu_0,\mu_1)=\int_Y 
\left((q_0{-}q_1)^2+4q_0q_1\sin^{2}(|\bfT_0{-}\bfT_1|/2\land\pi/2)\right)\,\d\nu,
\end{split}
\end{equation} 
then we will call that an \emph{optimal \TRAGRO\ couple} from $\mu_{0}$ to
$\mu_{1}$.
\end{definition}

\begin{definition}\label{primes}
For $\mu_{0},\mu_{1}\in\calM(\YX),$ we define the following sets:
\begin{equation}
  \A'_{i}=\{x\in\YX: \dist(x,\supp(\mu_{1-i}<\pi/2))\}, \hspace{32pt} 
\A_{i}'':=\YX\setminus \A_{i}'. 
\end{equation}
We also define the following measure
\begin{equation}
\mu'_{i}(t):=(\mu_{i})_{|_{\A'_{i}}}(t)=\mu_{i}(t\cap \A'_{i}),  \hspace{44pt} \mu''_{i}(t):=(\mu_{i})_{|_{\A''_{i}}}(t)=\mu_{i}(t\cap \A_{i}).
\end{equation}
\end{definition}

\begin{definition}[Reduced couple]\label{reduced couple}
A couple of measures $(\mu_{0},\mu_{1})\in\calM(\YX)^{2}$ is called reduced if $\mu_{0}=\mu'_0, \mu_{1}=\mu'_{1}.$

\end{definition}

For every couple $(\mu_{0},\mu_{1}),$ the couple $(\mu'_{0},\mu'_{1})$ is
always reduced. Now, we have the following theorem that is a simplified version
of \cite[Corollary 3.5]{LiMiSa23FPGG}.

\begin{theorem}\label{injective}
  Let $\YX\subset\mathbb{R}^{d}$ be a compact, convex set with nonempty
  interior. Let $(\mu_{0},\mu_{1})\in\calM(\YX)^{2},$ and
  $\mu_{0}\simeq\calL^{d} .$  Then there exists  an optimal
  \TRAGRO\ couple $(q,\bfT )$ from $\mu_{0}$ to $\mu_{1},$ with
  $|\bfT (x){-}x |<\pi/2$ \ $ \Ld $-a.e.  If furthermore
  $\mu'_{1}\ll\calL^{d} ,$ then $\widetilde{\bfT }$ is essentially injective.
\end{theorem}

\begin{remark} A transport plan $\bfT$ as in Theorem \ref{injective} has a version
  that is fully injective. From now on, without loss of generality we will make
  the assumption that $\bfT$ is fully injective to simplify the arguments. Even
  more, a couple like that, will be called an injective optimal
  \TRAGRO\ couple from $\mu_{0}$ to $\mu_{1}.$
\end{remark}
 
In the next lemma, we will show that if the couple $(q,\bfT )$ is an
injective optimal \TRAGRO\ couple from $\mu_{0}$ to $\mu_{1},$
then it acts as an injective optimal \TRAGRO\ couple from
$(\mu_{0})_{|_{\A}}$ to $(\mu_{1})_{|_{\bfT(\A)}},$ for every measurable set $\A.$
Even more for any partition $\{\A_{i}\}$ of $\YX,$ the total
\TRAGRO\ cost squared is equal to the sum of the squares of the
\TRAGRO\ costs for each part of the partition. This
straightforward lemma will be used in the next section, for the construction of
measures that violate the minimum assumption for the MM scheme if the
minimum candidate does not have nice density bounds. These construction will be
achieved by cutting and gluing the potential candidate with other measures.
 
\begin{lemma}\label{split}
  Let $(\mu_{0},\mu_{1})\in\calM(\YX)^{2}$ with $\mu'_{0},\mu'_{1}\ll\calL^{d} $
  and $\mu_{1}''=0$. Let $(q,\bfT )$ be an injective optimal
  \TRAGRO\ couple from $\mu_{0}$ to $\mu_{1}$ and let
  $\A_{i}, i=1,\dots, n,$ be a partition of $\YX$. If $\mu^{i}_{0}$ is the
  restriction of $\mu_{0}$ on $\A_{i},$ and $\mu^{i}_{1},$ the restriction of
  $\mu_{1}$ on $\bfT (\A_{i}),$ then we have:
\begin{itemize} 
\item $(q,\bfT )$ is an injective optimal \TRAGRO\ couple from
  $\mu^{i}_{0}$ to $\mu^{i}_{1}$.
    \item $\sum_{i}\mu^{i}_{1}=\mu_{1}.$
    \item $\HK^{2}(\mu_{0},\mu_{1})=\sum_{i=1}^{n}\HK^{2}(\mu^{i}_{0},\mu^{i}_{1}).$
\end{itemize}	
\end{lemma}
\begin{proof}
 We  have 
\begin{equation*}
\begin{split}
&\int_{\YX}\zeta(x)\mu^{i}_{1}(\d x)=\int_{\YX}\mathbb{I}_{\bfT (A^{i})}(x)\zeta(x)\mu_{1}(\d x)=\int_{\YX}\mathbb{I}_{\bfT (A^{i})}(\bfT (x))\zeta(\bfT (x))q^{2}(x)\mu_{0}(\d x)\\&=\int_{(\YX\setminus \A^{i})}\mathbb{I}_{\bfT (A^{i})}(\bfT (x))\zeta(\bfT (x))q^{2}(x)\mu_{0}(\d x)+\int_{A_{i}}\mathbb{I}_{\bfT (A^{i})}(\bfT (x))\zeta(\bfT (x))q^{2}(x)\mu_{0}(\d x)\\&=0+\int_{A_{i}}\zeta(\bfT (x))q^{2}(x)\mu_{0}(\d x)=\int_{\YX}\zeta(\bfT (x))q^{2}(x)\mu^{i}_{0}(\d x).
\end{split}
\end{equation*}
By summing over $i,$ we also obtain $\sum\mu^{i}_{1}=\mu_1.$ Since
$|\bfT (x){-}x|<\frac{\pi}{2},$ it holds that for every $i\in{1,\dots,n},$
the couple $(\mu^{i}_{0},\mu^{i}_{1})$ is reduced and we can construct an
optimal \TRAGRO\ couple $(q^{i},\bfT ^{i})$ for
$(\mu^{i}_{0},\mu^{i}_{1}),$ with cost, $\HK^{2}(\mu^{i}_{0},\mu^{i}_{1}).$ For
every $i,$ we have that $(q,\bfT )$ is a \TRAGRO\ couple
between $\mu^{i}_{0},$ and $\mu^{i}_{1}.$ Therefore we have
\begin{equation}
 \label{assass}
\begin{split}
\HK^{2}(\mu^{i}_{0},\mu^{i}_{1})&{=}\int_{\YX}\left(1+q^i(x)^{2}-2q^{i}(x) 
  \cos(|x{-}\bfT ^{i}(x)|)\right)\mu^{i}_{0}(\d x)\\
&\leq\int_{\YX}\left(1+q^{2}(x)-2q(x)\cos(|x{-}\bfT (x)|)\right)\mu^{i}_{0}(\d x).
\end{split}
\end{equation}
We now define $(\widetilde{q},\widetilde{\bfT })$ by
\begin{equation}
\widetilde{\bfT }(x)=\bfT ^{i}(x),\hspace{8pt} x\in A^{i} ,\hspace{32pt} \widetilde{q}(x)=q^{i}(x),\hspace{8pt} x\in A^{i}, 
\end{equation}
we have
\[
\mu_{1}=\sum_{i=1}^{n}\mu_{1}^{i}=\sum_{i=1}^{n} 
\widetilde{\bfT }_{\#}(\widetilde{q}^{2}\mu^{i}_{0})
=\widetilde{\bfT }_{\#}(\widetilde{q}^{2}\mu_{0}),
\]
and therefore it is a \TRAGRO\ couple for $(\mu_{0},\mu_{1})$
with total cost
\begin{equation}
\sum_{i=1}^{n}\HK^{2}(\mu^{i}_{0},\mu^{i}_{1})
 \leq \sum_{i=1}^{n}\int_{\YX}\left(1+q^{2}(x)-2q(x) 
 \cos\big(|x{-}\bfT (x)|\big)\right)\mu^{i}_{0}(\d x)=
\HK^{2}(\mu_{0},\mu_{1})
\end{equation} If at least one of the estimates in \ref{assass} is strict then
the above inequality is also strict which implies that
$(\widetilde{q},\widetilde{\bfT })$ is a \TRAGRO\ couple
that has less cost than  $(q,\bfT ),$ which contradicts the fact that the
latter is optimal. From the above we have that \eqref{assass} is an  equality
for every $i$, and therefore  $(q,\bfT )$ is an injective optimal
\TRAGRO\ couple between $\mu^{i}_{0},$ and $\mu^{i}_{1},$ and 
$\sum_{i=1}^{n}\HK 
^{2}(\mu^{i}_{0},\mu^{i}_{1})=\HK 
^{2}(\mu_{0},\mu_{1})$. 
\end{proof} 

We will also use the following results, which was discussed carefully in
\cite{LasMie19GPCA}. 

\begin{proposition}[Scaling property of $\HK$]
\label{pr:ScalHK}
For all $\mu_{0},\mu_{1}\in\calM(\YX)$ and $t_0,t_1 \geq 0$ we have
\begin{equation}
  \label{eq:ScalHK}
  \HK^{2}(t_0^2\mu_{0}, t_1^2\mu_{1})=t_0t_1 
  \HK^{2}(\mu_{0},\mu_{1}) +(t_0^2{-}t_0t_1)\mu_{0}(\YX) + (t_1^2
  {-} t_0t_1)\mu_{1}(\YX).
\end{equation}
Even more, if $\bfH_{01}$ is an optimal plan for the $\LET$ formulation of
$ \HK (\mu_{0},\mu_{1}),$ then $\bfH_{01}^{t_0t_1}= t_0t_1\bfH_{01}$ is an
optimal plan for $ \HK(t_0^2\mu_{0}, t_1^2\mu_{1}).$
\end{proposition}

Choosing $t_0/t_1=(\mu_1(X)/\mu_0(X))^{1/2}$ in \eqref{eq:ScalHK} we
obtain the lower bound 
\begin{equation}\label{lower bound}
  \HK^{2}(\mu_{0}, \mu_{1})\geq
  \left(\sqrt{\mu_{0}(X)}-\sqrt{\mu_{1}(X)}\right)^{2}. 
\end{equation}

\subsection{The Spherical $\HK$ distance $\SHK$}
\label{su:SpherHK}

The Spherical Hellinger-Kantorovich space $(\calP(\YX),\SHK)$ was introduced in
\cite{LasMie19GPCA}, and the distance metric is related to the
Hellinger-Kantorovich distance $\HK$ restricted to
$\calP(\YX){\times}\calP(\YX)$ through the formula by
\[ \SHK(\nu_0,\nu_1) = \arccos\left( 1- \frac12 \,\HK^2(\nu_0,
      \nu_1)\right)=2\arcsin\left(\frac12\,\HK(\nu_0, \nu_1)\right) .
\]

The important point is that $(\calP(\YX),\SHK)$ is still a geodesic space in
the sense of Definition \ref{de:Geod}. Furthermore, the work
\cite{LasMie19GPCA} demonstrates how all geodesics connecting $\nu_0$ and
$\nu_1$ in $(\calP(X),\SHK)$ can be obtained by suitably projecting, which
involves dividing by the total mass, and reparametrizing the geodesics in
$(\calM(X),\HK)$.

\section{Density bounds for the MM scheme}
\label{se:density bounds} 

The objective of this section is to demonstrate that when starting with a
$\mu_0$ exhibiting favorable density bounds, the minimizers of the MM schemes
\eqref{scheme} also possess similar bounds. These bounds will later be utilized
to recover concavity properties for the squared distance along the geodesics
that interpolate the points generated by the scheme. The approach is an
extension of an idea by Felix Otto developed in \cite{Otto96DDDE}. There, the
author proved that for a specific class of functionals $\E,$ and for measures
$\mu_{0}$ that have density bounded from below by a number $c_{\min}>0,$ the
one step minimizer of the MM scheme,
$\mu_{1}=\arg\min\left\{\frac{W_{2}(\mu_{0},\mu)}{2\tau}+\E(\mu)\right\},$ has
also the same property. The main argument was, that if the set of points with
density smaller than $c_{min}$ is not essential empty, then mass must have
moved outside from it to an another set, resulting in density bigger than
$c_{min}.$ This would imply, that keeping some of the mass at place would not
only have been cheaper with respect to the Wasserstein distance, but also would
have resulted to a more ``uniform" distribution of the density and therefor a
smaller value of the functional, getting a contradiction to the assumption that
$\mu_{1}$ is a minimizer. Felix Otto, used this argument in \cite{Otto96DDDE}
to prove that the MM scheme for the p-density functionals, converge to
solutions of the doubly degenerate diffusion equations. Our arguments are of
similar nature, although we have to take into account that our setting also
allows for destruction and creation of mass.

\subsection{ Motivation of density bounds}
\label{su:MotivBounds}

When analyzing the minimizing movement scheme for diffusion equations in
the Kantorovich-Wasserstein setting, comparison principles for parabolic
equations play a crucial role. When adding reaction terms like for $\HK$ or
$\SHK$ gradient-flow equations, it is important to see how these comparison
principles can still be exploited. This is especially nontrivial in the case of
$\SHK$, where a nonlocal term appears. 

As the diffusion is well understood, it is worthwhile to look at the pure ODE
case. Consider any convex domain $\YX\subset \R^d$ with $\calL^d(\YX)=1.$
Moreover, we restrict our view to measures with spatially constant Lebesgue
density, i.e.\ $\mu(t) = c(t) \calL^d$ as special solutions for the gradient
system $(\calM(\YX),\E,\HK)$. Clearly, the equation for the scalar $c$ is
\[
\dot c  = - 4\,c\, E'(c). 
\]
Because of the above choices we have $\HK(c_0 \calL^d,c_1\calL^d)=\left(
\sqrt{c_1}{-}\sqrt{c_0}\right)^2$, and 
the MM scheme reduces to 
\begin{equation}
  \label{eq:Scalar.EuLag}
  \frac1{2\tau} \left( \sqrt{c_1}{-}\sqrt{c_0}\right)^2+ E(c_1)\leadsto
\min\limits_{c_1\geq 0} 
\quad \longleftrightarrow \quad
1- \sqrt{\tfrac{\ds c_0}{\ds c_1}} + 2\tau E'(c_1)=0. 
\end{equation}
Assuming $ E'(c_\mafo{low}) \leq 0 \leq E'(c_\mafo{upp})$ we obtain the
following trivial observations for the MM scheme solutions:
\begin{align*}
\text{(D1)} &\qquad c_0 \geq c_\mafo{upp} \quad \Longleftrightarrow 
\quad c_1 \leq c_0,
\\
\text{(D2)} &\qquad  c_0 \leq c_\mafo{low} \quad \Longleftrightarrow 
\quad c_1 \geq c_0 .
\\
\intertext{However, in the case inbetween, we obtain nontrivial estimates:}  
\text{(D3)} &\qquad c_1 \leq \max\Big\{a, \,\frac{
  \ds c_0}{(1{+}2\tau\min\{E'(a),0\})^2}\Big\} 
\text{ whenever }2\tau E'(a)>-1;
\\
\text{(D4)}& \qquad c_1 \geq \min \Big\{ b,\:\frac{c_0}{(1 {+} 2\tau \max\{
  E'(b),0\})^2} \Big\} \text{ for all } b \geq 0.  
\end{align*}
To see that the upper estimate in (D3) holds we set $\mu_a:=\min\{E'(a),0\}$
with $0\geq \mu_a >-1/(2\tau)$ and assume that (D3) does not hold, i.e. (i)
$c_1>a$ and (ii) $c_1> c_0/(1{+}2\tau \mu_a)^{-2}$. Then, by (i) and
monotonicity of $E'$ we have
$0\leq (1{+}2\tau \mu_a)^2 \leq (1{+}2\tau E'(c_1))^2$. Exploiting the
Euler-Lagrange equation in \eqref{eq:Scalar.EuLag} we continue
\[
(1{+}2\tau \mu_a)^2 \leq (1{+}2\tau
E'(c_1))^2  \overset{\text{EL eqn.}}= c_0/c_1 \overset{\text{(ii)}} <
(1{+}2\tau \mu_a)^2 ,  
\]   
which is the desired contradiction. The lower estimate in (D4) follows similarly. 

We have considered the simple ODE case because it turns out that similar
estimates hold for the densities in a true minimization step for
$(\calM(\YX),\E,\HK)$, see \eqref{eq:DensUppBd.HK} in Proposition
\ref{pr:discr.max.princ.HK.upp} for the upper estimate (D3) and Proposition
\ref{pr:discr.max.princ.HK.low} for the lower estimate (D4).

The ``spherical'' case for $(\calM(X),\E,\SHK)$ is in fact much better,
because no sign conditions for $E'(c)$ are needed. To see this we again
consider the pure Spherical Hellinger space $(\calP(X),\E,\SHe)$ with
$\SHe(\nu_0,\nu_1) = 2 \arcsin\big(\He(\nu_0,\nu_1)/2\big)$ and
$\He(\mu_0,\mu_1)^2 = \int_X \big[\big(\frac{\rmd \mu_0}{\rmd\mu}\big)^{1/2} -
\big(\frac{\rmd \mu_0}{\rmd\mu}\big)^{1/2}\big]^2 \d \mu$ for any $\mu$ with
$\mu_0{+}\mu_1\ll \mu$. The corresponding gradient flow for for absolutely
continuous measures $\nu(t)=c(,\cdot)\rmd x$ leads to 
\[
 \dot c(t,x)  = - 4 c(t,x) \Big( E'(c(t,x)) - 
 \int_X c(t,y) E'(c(t,y)) \,\d  y\Big). 
\]      
For this flow it can be shown that $t\mapsto \inf c(t,\cdot)$ is increasing and
$t \mapsto \sup c(t,\cdot)$ is decreasing. For this we consider any smooth
convex function $\varphi:{]0,\infty[}\to \R$ and observe 
\begin{align*}
 \frac14\frac{\rmd}{\rmd t} \int_X \varphi(c(t,x))\ddd x 
  & = \frac14\int_X \varphi'(c) \dot c \ddd x
    = \int_X E'(c)c \ddd x \int_X \varphi'(c)c \ddd x - \int_X
      E'(c)\varphi'(c) c \ddd x 
 \\
 & = \int_0^\infty E'(c) \ddd F_t(c)  \int_0^\infty \varphi'(c) \ddd F_t(c)
    - \int_0^\infty E'(c)\varphi'(c) \ddd F_t(c) \overset{*}\leq 0,
\end{align*}  
where $  F_t(b):=\int_X c(t,x) 1\!\!1_{c(t,\cdot)\leq b}(x) \ddd x $ with
$0\leq F_t(c) \leq F_t(\infty)=1$. The estimate $\overset*\leq $ is a
well-known rearrangement estimate following from the monotonicities of $E'(c)$
and $\varphi'$, see \cite[Ch.\,10.13]{HaLiPo34I}. 

Given $c_0$ with $\underline c \leq c_0(x) \leq \overline c$ we set $c_*=
\frac12\big( \overline c{+}\underline c\big)$ and $\delta = \overline c -
c_*$. For $p\geq 2$ we choose $\varphi(c) =
|c{-}c_*|^{p}$. Using $\calL^d(X)=1$ again we find
$\|c(t){-}c_*\|_{\rmL^p} \leq \|c_0{-}c_*\|_{\rmL^p} \leq \delta$. In the limit
$p\to \infty$ we are left with $\|c(t){-}c_*\|_{\rmL^\infty}  \leq \delta$,
which implies $c(t,x) \in [c_*{-}\delta, c_*{+}\delta] = [\underline c,
\overline c]$ as desired.

\subsection{A density estimate for one-step minimizers}
\label{su:DensEstimOneStep}

We will first prove a lemma that works for both spaces and their respective
minimization movement schemes
\begin{equation}
   \label{schemerepeat}
 \mu_{1}=\inf_{\mu \in\calM(\YX)}\left\{\frac{\HK^{2}(\mu_{0},\mu)}{
     2\tau}+\E(\mu)\right\}, 
\qquad \mu_{1} =\inf_{\mu \in \calP(\YX) }
\left\{\frac{\SHK^{2}(\mu_{0},\mu)}{2\tau}+\E(\mu)\right\}, 
\end{equation}  
under the sole assumption that the function $E,$ which generates $\E$ in
\eqref{EntropyFunctional}, is convex. We will write \eqref{schemerepeat}$_\HK$
and \eqref{schemerepeat}$_\SHK$ to distinguish the two different incremental
minimization schemes.  

Note that for our bounded and convex domains $X\subset \R^d$ the convergence
with respect to $\HK$ or $\SHK$ is equivalent to the weak* convergence of
measures. Moreover, by our assumptions $\E$ is weakly* lower semicontinuous and
linearly bounded from below. Hence in both cases the functionals to be
minimized are coercive (for $\tau>0$ sufficiently small) and weakly* lower
semicontinuous. Hence, the minimizing movement scheme is well defined for
arbitrarily many steps.

We will start by providing a lemma that is a generalization of Otto's argument
in \cite[Lem.\,1.1.2]{Otto96DDDE}.  Let $\mu_{1}$ be defined as the
one-step solution to either of the minimization schemes
\eqref{schemerepeat}. According to the Lemma, mass can be transferred from a
set $A$ to some set $\bfT(A),$ only if it results in a situation where the
density $\rho_{1}=\frac{\d\mu_{1}}{\d \Ld }$ at the destination $\bfT(A),$ is
less than the  density  in the original $A$. In the case of the
Wasserstein distance where mass can only be transported, this is enough to
prove bounds for $\rho_{1},$ by studying the set where the density is below the
minimum of $\rho_{0}.$ However, for  our distances $\HK$ and $\SHK$
the arguments are a bit more involved.

\begin{lemma}
\label{uniformitypreference}
Let $X \subset \R^d$ be a compact, convex set with nonempty
interior. Furthermore let $\E$ be as in \eqref{EntropyFunctional} with $E$
being a convex real-valued function.  Let finally $\mu_{0}\simeq \Ld ,$ and
$\mu_{1} \in \calM(\YX)$ as in \eqref{schemerepeat}$_\HK$ or
\eqref{schemerepeat}$_\SHK.$ Then, either $\mu_{1}\equiv 0$ or
$\mu_{1}\simeq \Ld $ and, in both cases $\HK$ and $\SHK$, the injective optimal
\TRAGRO\ couple $(q,\bfT )$ from $\mu_{0}$ to $\mu_{1}$ provided by Theorem
\ref{injective} satisfies 
\begin{equation}
\label{eq:densT.leq.dens}
 \rho_{1}(\bfT (x))\leq \rho_{1}(x) \hspace{8pt}\left(\text{alternatively }  \rho_{1}(\bfT ^{-1}(x))\geq \rho_{1}(x)\right),\hspace{8pt}  \text{almost everywhere.}
 \end{equation} 
\end{lemma}

\noindent
\begin{proof}
We are going to prove the statement in two parts. First we are going to show
that $\mu_{1}\ll  \Ld $, with  
$\rho_{1}(\bfT (x))\leq \rho_{1}(x).$ On the second part we are going to
prove $ \Ld \ll\mu_{1}$. \smallskip

\underline{\em Step A: Proving $\mu_{1}\ll \Ld $ and
  $\rho_{1}(\bfT (x))\leq \rho_{1}(x).$ Constructing the counterexample.}   \medskip
  
We note that since $\mu_{0}\simeq \Ld ,$ it exits a \TRAGRO\ couple from
$\mu_{0}$ to $\mu_{1}.$ Let $\mu^{s}$ denote the singular part of $\mu$ in the
Lebesgue decomposition. In case $\mu^{s}_{1}\not\equiv 0,$ it exists set $\B$
such that $\mu^{s}_{1}(\B)\neq0,$ and $ \Ld (\B)=0.$ The identity

\begin{equation}
   0<\mu_{1}(B)=\int_{B}\mu_{1}(\d x)=\int_{T^{-1}(B)}q^{2}(x)\mu_{0}(\d x ),
\end{equation}
 guarantees that for the set $\A=\bfT^{-1}(\B)$ we have that $\mu_{0}(\A)\neq 0$ and therefore $ \Ld (\A)>0.$
Having in mind to generate a contradiction to the assumption, we get that in both cases where either $\mu^{s}_{1}\not\equiv 0,$ or $\mu^{s}_{1}\equiv 0,$ if the assumption is violated then there exists  $a < b$ and  a set $\A$ with
$ \Ld (\A)>0$ such that 
$\rho_{1}(x)< a < b < \rho_{1}(\bfT (x))$ for every $x\in \A$.   \medskip
  When $\mu^{s}_{1}(\bfT(\A))\neq 0,$ while $ \Ld (\bfT(\A))=0,$ we set $\rho_{1}(\bfT (x))$ equal to $\infty.$  We define $\mu^{\bfT (\A)}_{1}$ to be the restriction of $\mu_{1}$ onto $\bfT (\A)$, and
$\mu^{\A}_{0}$ the restriction of $\mu_{0}$ onto $\A$.  For $0<t<1$, we define the
measure
$\mu^{t}_{1}=\mu_{1} - t \mu^{\bfT (\A)}_{1}+
t\frac{\mu^{\bfT (\A)}_{1}(\YX)}{\mu_{0}^{\A}(\YX)}\mu_{0}^{\A},$ which satisfies
$\mu^t_1(\YX)=\mu_1(\YX)$. Moreover, by assumption we have $A\cap
\bfT(A)=\emptyset$ and can decompose $\mu^t_1$ as 
\[
\mu^{t}_{1} = \mu_1\big|_{X\setminus(A\cap \bfT(A))} + (1{-}t)
\mu_1\big|_{\bfT(A)} + \Big( \mu_1\big|_A + t \frac{\mu_1(\bfT(A))}{\mu_0(A)}
\mu_0\big|_A\Big). 
\]
Thus, nothing is changed on $X\setminus(\A\cap \bfT(\A))$, while mass is taken
away on $\bfT(\A)$ proportional to $\mu_1^{\bfT(\A)}=\mu_1|_{\bfT(\A)}$ and added on $\A$
proportional to $\mu_0^{\A}=\mu_0|_{\A}$.  \medskip

We will
prove that the $\HK $ distance between $\mu_{0}$ and $\mu_{1}$ is not smaller than
the resulting cost for this new \TRAGRO\ couple and therefore
not smaller than $\HK (\mu_{0},\mu^{t}_{1}).$ At the same time, we will show that
$\E(\mu_{1}^{t})< \E(\mu_{1}),$ for small enough $t,$ leading this way to a
contradiction.\medskip

\underline{\em Step A.1: Proving that the constructed measure is closer to $\mu_{0}$.} By applying Lemma \ref{split} for $\A$ and $\YX\setminus \A$, we obtain By applying Lemma \ref{split} for $\A$ and $\YX\setminus \A$, we
obtain 
\begin{equation*}
\begin{split}
&\HK^{2}(\mu_{0},\mu_{1})=\HK^{2}(\mu_{0}{-}\mu^{\A}_{0},\mu_{1}{-}\mu^{\bfT (\A)}_{1})+
\HK^{2}(\mu^{\A}_{0},\mu^{\bfT (\A)}_{1})
\\
&\ \ =  
  \HK^{2}(\mu_{0} {-} \mu^{\A}_{0} , \mu_{1} {-} \mu^{\bfT (\A)}_{1})
  +\HK^{2}\left(\mu^{\A}_{0},(1{-}t)\mu^{\bfT (\A)}_{1} +
    t\frac{\mu^{\bfT (\A)}_{1}(\YX)}{\mu_{0}^{\A}(\YX)}\mu_{0}^{\A}\right) 
\\
&\hspace{54pt}+ \HK^{2}(\mu^{\A}_{0},\mu^{\bfT (\A)}_{1})  - \HK
^{2}\left(\mu^{\A}_{0},(1{-}t)\mu^{\bfT (\A)}_{1} +
  t\frac{\mu^{\bfT (\A)}_{1}(\YX)}{\mu_{0}^{\A}(\YX)}\mu_{0}^{\A}\right) 
\\
&\overset{\text{subadd.}}\geq \!\!
\HK^{2}(\mu_{0},\mu^{t}_{1}) + \HK^{2}(\mu^{\A}_{0},\mu^{\bfT (\A)}_{1})  - \HK
^{2}\left(\mu^{\A}_{0},(1{-}t)\mu^{\bfT (\A)}_{1} +
  t\frac{\mu^{\bfT (\A)}_{1}(\YX)}{\mu_{0}^{\A}(\YX)}\mu_{0}^{\A}\right)
\\
& \overset{\text{subadd.}}\geq\!\!
\HK^{2}(\mu_{0},\mu^{t}_{1}) + \HK^{2}(\mu^{\A}_{0},\mu^{\bfT (\A)}_{1})  - \HK
^{2}\big((1{-}t)\mu^{\A}_{0},(1{-}t)\mu^{\bfT (\A)}_{1} \big) -\HK^{2}\Big(
  t\mu^{\A}_{0}, t \frac{\mu^{\bfT (\A)}_{1}(\YX)}{\mu_{0}^{\A}(\YX)}\mu_{0}^{\A}\Big)
\\
& \overset{\text{\eqref{eq:ScalHK}}}\geq 
\HK^{2}(\mu_{0},\mu^{t}_{1}) + t\left(\HK^{2}(\mu^{\A}_{0},\mu^{\bfT (\A)}_{1}) -\HK
  ^{2}\left(
    \mu^{\A}_{0},\frac{\mu^{\bfT (\A)}_{1}(\YX)}{\mu_{0}^{\A}(\YX)}\mu_{0}^{\A}\right)\right)\geq \HK^{2}(\mu_{0},\mu^{t}_{1}) .
\end{split}
\end{equation*}
(See \cite[Lem.\,7.8]{LiMiSa18OETP} for the subadditivity of $\HK^2$.) 
In the last estimate we used that starting from a measure $\mu_{0}$, then among
all measures with a given mass $M$, the measure $M\mu_{0}$ has the least
distance. Indeed, we have 
\begin{equation*}
\begin{split}
\HK^{2}\left(\mu^{\A}_{0},\frac{\mu^{\bfT (\A)}_{1}(\YX)}{\mu_{0}^{\A}(\YX)}\mu_{0}^{\A}\right)&\overset{\eqref{eq:ScalHK}}=\sqrt{\frac{\mu^{\bfT (\A)}_{1}(\YX)}{\mu_{0}^{\A}(\YX)}}
	\HK^{2}(\mu^{\A}_{0},\mu^{\A}_{0}) +\left(\frac{\mu^{\bfT (\A)}_{1}(\YX)}{\mu_{0}^{\A}(\YX)}{-}\sqrt{\frac{\mu^{\bfT (\A)}_{1}(\YX)}{\mu_{0}^{\A}(\YX)}}\right)\mu^{\A}_{0}(\YX)\\
	&\hspace{-6em} + \left(1
	{-}\sqrt{\frac{\mu^{\bfT (\A)}_{1}(\YX)}{\mu_{0}^{\A}(\YX)}}\right)\mu^{\A}_{0}(\YX)=\left(\sqrt{\mu^{\bfT (\A)}_{1}(\YX)}-\sqrt{\mu^{\A}_{0}(\YX)}\right)^{2}\
      \overset{\eqref{lower bound}}\leq\HK^{2}(\mu^{\A}_{0},\mu^{\bfT (\A)}_{1})
\end{split}
\end{equation*}
We will treat the case where $\mu_{1}$ has a singular part and the case
$\mu_{1}\ll  \Ld $ separately.\medskip

\underline{\em Step A.2.1: Entropy functional estimate; the case $\mu^{s}\equiv 0$} 
\begin{equation}
  \label{differe}
\begin{split}
\E(\mu_{1}^{t})-\E(\mu_{1})&=\int_{X}\left(E(\rho_{1}^{t}(x))-E(\rho_{1}(x)) \right) \Ld (\d x)
\\&\leq\int_{X}E'(\rho_{1}^{t}(x))\left(\rho_{1}^{t}(x)-\rho_{1}(x)\right) \Ld (\d x)\\&=
\int_{X}\left(E'(\rho_{1}^{t}(x)-E'\left(\tfrac{a+b}{2}\right)
\right)(\rho_{1}^{t}(x){-}\rho_{1}(x)) \Ld (\d x),
\end{split}
\end{equation}
where in the last term the constant $E'(\tfrac{a+b}{2})$ could be inserted
because of $\int_X \rho^t \d  \Ld  = \mu^t(\YX) = \mu_1(\YX) = \int_X \rho_1 \d
 \Ld $. The integrand in the last term of \eqref{differe} can be estimated
as follows:
\begin{align*}
\nonumber
&\left(E'(\rho_{1}^{t}(x))-E'\left(\tfrac{a+b}{2}\right)
\right)(\rho_{1}^{t}(x)-\rho_{1}(x))
\\&
= t\Big(E'\Big(\rho_{1}(x) {+} 
  t \frac{\mu^{\bfT (\A)}_{1}(\YX)}{\mu_{0}^{\A}(\YX)}\rho_{0}^{\A}(x) {-} 
  t\rho^{\bfT(\A)}_{1}(x)\Big) - E'\big(\frac{a+b}{2}\big)\Big) 
  \Big(
  \frac{\mu^{\bfT (\A)}_{1}(\YX)}{\mu_{0}^{\A}(\YX)}\rho_{0}^{\A}(x) 
  - \rho^{\bfT(\A)}_{1}(x) \Big)
\\&
\leq
  \begin{cases}
    t\Big(E'\big(a+t\frac{\mu^{\bfT (\A)}_{1}(\YX)}{\mu_{0}^{\A}(\YX)}\rho_{0}(x)
      \big)-E'\left(\frac{a+b}{2}\right)\Big)
    \frac{\mu^{\bfT (\A)}_{1}(\YX)}{\mu_{0}^{\A}(\YX)}\rho_{0}(x) & \,\text{on }\A, \\
    t\left(E'\left(\frac{a+b}{2}\right) - E'\left(b - t\rho_{1}(x)
      \right)\right)\left(\rho_{1}(x) \right) &\,\text{on } \bfT (\A).
  \end{cases}
\end{align*}
In both cases the factor multiplying $t$ is negative for very small
values of $t$. Hence, we have  shown $\E(\mu^t)< E(\mu_1)$ which is the
desired contradiction.

\underline{\em Step A.2.2: Entropy functional estimate; the case $\mu^{s}_{1}\not\equiv 0.$}
Let $\B$ such that $ \Ld (\B)=0$ and $\mu^{s}_{1}(\B^{c})=0$
\begin{equation}
  \label{differe2}
\begin{split}
&\E(\mu_{1}^{t})-\E(\mu_{1})=\int_{\B}\left(E(\rho_{1}^{t}(x))-E(\rho_{1}(x)) \right) \Ld (\d x) +E'_{\infty}\mu^{t}_{1}(\B^{c})-E'_{\infty}\mu_{1}(\B^{c})
\\&\leq\int_{\B}E'(\rho_{1}^{t}(x))\left(\rho_{1}^{t}(x)-\rho_{1}(x)\right) \Ld (\d x)+E'_{\infty}\mu^{t}_{1}(\B^{c})-E'_{\infty}\mu_{1}(\B^{c})\\&=
\int_{\B}\left(E'(\rho_{1}^{t}(x))-E'\left(\tfrac{a+b}{2}\right)
\right)(\rho_{1}^{t}(x){-}\rho_{1}(x)) \Ld (\d x)\\
&\quad + \left(E'_{\infty}-E'\left(\frac{a+b}{2}\right)\right) 
  ( \mu^{t}_{1}(\B^{c})-\mu_{1}(\B^{c})), 
\end{split}
\end{equation}
where in the last term the constant $E'(\tfrac{a+b}{2})$ could be inserted
because of $ \mu^t(\YX) = \mu_1(\YX).$ The integral in the last term of \eqref{differe} can be estimated
as follows:
\begin{equation}
\begin{split}
&\int_{\B}\left(E'(\rho_{1}^{t}(x))-E'\left(\tfrac{a+b}{2}\right)\right)(\rho_{1}^{t}(x){-}\rho_{1}(x)) \Ld (\d x)\\&\leq
\int_{\A}\left(E'(\rho_{1}^{t}(x))-E'\left(\tfrac{a+b}{2}\right)\right)(\rho_{1}^{t}(x){-}\rho_{1}(x)) \Ld (\d x)\\&\leq \int_{\A}t\left(E'\left(a+t\frac{\mu^{\bfT (\A)}_{1}(\YX)}{\mu_{0}^{\A}(\YX)}\rho_{0}(x)\right)-E'\left(\frac{a+b}{2}\right)\right)\frac{\mu^{\bfT (\A)}_{1}(\YX)}{\mu_{0}^{\A}(\YX)}\rho_{0}(x),
\end{split}
\end{equation}
while
\begin{equation}
  \mu^{t}_{1}(\B^{c})-\mu_{1}(\B^{c})= (\mu^{t}_{1}(\bfT(\A))-\mu_{1}(\bfT(\A)))=  t\left(E'\left(\frac{a+b}{2}\right) - E'_{\infty} \right) \mu_{1}(\bfT(\A))
\end{equation}
In both cases the factor multiplying $t$ are negative for very small
values of $t$. Hence, we have  shown $\E(\mu^t)< E(\mu_1)$ which is the
desired contradiction.\medskip

\underline{\em Step B: Proving $\mu_{1}\equiv 0$ or
  $ \Ld \ll\mu_{1}$.} We will assume that $\mu_{1}\not\equiv 0$ and
that there exists $\B=\{x:\rho_{1}(x)=0\}$ with $ \Ld (\B)>0$
to reach a contradiction.\smallskip

Since  $ \Ld (\B)>0,$ there exist $x_{0}\in X$ and $r_0\in
(0,\frac{\pi}{4}) $ such that $ \Ld (B(x_{0},r_{0})\cap \B)> 0$ 
and $ B(x_{0},r_{0})\subset \YX$. 
By the assumption $\mu_1\not\equiv0$, the set $B^{c}=X\setminus B$ satisfies
$ \Ld (B^{c})>0$,  and therefore there exist  $x_{1}\in X$ and 
$r_1\in (0,\frac{\pi}{4})$ with $\mathcal{L}^{p}(B(x_{1},r_{1})\cap \B^{c})> 0$
and $B(x_{1},r_{1}) \subset \YX$.

 We set $r_\theta=(1{-}\theta)r_0 + \theta r_1$ and $x_\theta =
(1{-}\theta)x_0 + \theta x_1$ and observe that the convexity of $\YX$ implies 
$B(x_\theta,r_\theta) \subset \R^d$ for $\theta\in [0,1]$. As $\mu_1 $ is
absolutely continuous the functions
$\beta(\theta):=\calL^d(B(x_\theta,r_\theta) \cap B)$ and $\gamma
(\theta):=\calL^d(B(x_\theta,r_\theta) \cap B^c)$ are continuous and satisfy
$\beta(\theta)+ \gamma(\theta)= c_d r_\theta^d>0$ for all $\theta\in
[0,1]$. With $\beta(0)>0$ and $\gamma(1)>0$ we conclude that there exists
$\theta\in[0,1] $  such that $ \Ld (B(x_\theta,r_\theta)\cap B)> 0$ and 
$ \Ld (B(x_\theta,r_\theta)\cap \B^{c})> 0.$ For economy of notation,
we denote $\B_{0}=B(x_\theta,r_\theta)\cap B$ 
and $B_{1}=B(x_\theta,r_\theta)\cap B^{c}.$ 

From the assumption that $\mu_{0}\simeq \Ld ,$ we have
$\mu_{0}(B_{0})>0.$ Also by definition of $B^{c}$ and therefore of $\B_1,$ we
have $\mu_{1}(B_{1})>0.$ Furthermore it holds that
$\sup_{x\in B_{0},y\in B_{1}}|x{-}y| <2 r_\theta <\frac{\pi}{2}$ which
means that $\mu_{1}$ 
has positive value in a set that has distance less than $\pi/2$ from points in
$B_{0}.$ More specifically, it holds $B_{0}\subset\supp(\mu'_{0}),$ and
$B_{1}\subset\supp(\mu'_{1}),$ where $\mu'_{0}$ and $ \mu'_{1}$ are as in Definition
\ref{primes}.

However, the optimality conditions in \cite[Cor.\,3.5, (3.31)]{LiMiSa23FPGG}
provide $|\bfT(x){-}x|<\pi/2$ and
$\sigma_{0}(x)\sigma_{1}(\bfT(x)) = \cos^2(|x{-}\bfT(x)|)$ on
$B_{0}\subset\supp(\mu'_{0})$ which means that mass in $B_{0}$ could not be
destroyed, but instead it must be transferred from $B_{0}$ somewhere in
$B^{c}.$ This implies that $\bfT(x)\in \B^{c}$ for all $x\in B_{0}.$ However in
Step A1, we proved $\rho_{1}(\bfT(x))\leq\rho_{1}(x)$ almost surely, therefore
$$0<\rho_{1}(\bfT(x))\leq\rho_{1}(x)=0\  \text{ for almost every}\hspace{8pt} x\in B_{0}.$$ This
gives us
\[
0<\int_{\B_{0}}\rho_{1}(\bfT(x)) \Ld( \d x) \leq \int_{\B_{0}}\rho_{1}(x)
  \Ld(\d x)  = \mu_{1}(\B_{0})=0
\]
which is a contradiction.\medskip

\underline{\em The $\SHK$ case:} The proof for $\SHK$ is exactly the same, because
the constructed counterexample has the same mass as the original and
because the distance $\SHK$ is a strictly monotone function of the $\HK$ distance, namely 
$ \SHK(\mu_0,\mu_1)=2\arcsin\big(\frac12\,\HK(\mu_0, \mu_1) \big).$ 
\end{proof}

\subsection{A single minimization step for $\HK$}
\label{su:SingleStepHK}

We proceed with proving upper and lower bounds for the density $\rho_{1}$ of
the measure $\mu_{1}$ defined by \eqref{schemerepeat}$_\HK,$ given that
$\mu_{0}$ satisfies some density bounds of its own,
i.e. $c_{\min}<\rho_{0}<c_{\max}$. The upper bound we retrieve is the same as
in (D3), and the proof is relatively straightforward. We assume that the density
$\rho_{1}$ of the minimizer $\mu_{1}$ in \eqref{schemerepeat}$_\HK$ is bigger
than the expected density given by (D3) in a set of positive Lebesgue measure. By
applying Lemma \ref{uniformitypreference}, we infer excessive creation of mass
in some set $B$. We conclude that, what we ``gain'' in $\HK$ distance by
restricting the growth on $B$ is more than what we ``lose'' for the entropy
function. Therefore ending up with a measure that has total energy less than
$\mu_{1},$ which is a contradiction. A similar bound was retrieved in
\cite{DimChi20TGMH} for a class of Entropy functionals that were studied in
that paper.

\begin{proposition}[Upper bound for incremental densities  for $\HK$]
\label{pr:discr.max.princ.HK.upp}
Let $X \subset \R^d$ be a compact, convex set with nonempty
interior.  Furthermore consider $\E$ as in \eqref{EntropyFunctional} with
$E$ being convex. Let finally 
$\mu_{0} = \rho_0  \Ld $ with $\rho_{0}(x)\leq c_{\max}$ and
$\mu_{1} \in \calM(\YX)$ as in \eqref{schemerepeat}$_\HK$. Then, for all
$c_\mafo{upp} \geq 0$ and $\tau>0$ with $E'(c_\mafo{upp}) \tau >-1/2$ we have
\begin{equation}
  \label{eq:DensUppBd.HK}
\rho_{1}(x)\leq \max\Big\{ c_\mafo{upp} , 
 \frac{c_{\max}}{\left(1+ 2\tau \min\{ E'(c_\mafo{upp}) ,0\}\right)^2}\Big\}.
 \end{equation}
\end{proposition}
\begin{proof}
We start by setting 
\[
k_{\epsilon}:=\frac{1}{1+2\tau \,\min\{E'(c_\mafo{upp}) 
    ,0\}}\:+\epsilon \geq 1+\epsilon.
\]

\underline{\em Step 1: Construction of sets.} 
In order to arrive a a contradiction we define the set
\[
\B:= \bigset{x \in X}{ \rho_1(x) \geq \xi_\epsilon } \quad \text{with
}\xi_\epsilon:=\,\max\{ (1+\epsilon)^{2}c_\mafo{upp},
   k_\epsilon^2c_\mafo{max}\}
\]
and assume $\mu_{1}(\B)>0$. By Lemma \ref{uniformitypreference}, for
almost every $x\in \mathbf{\bfT^{-1}}(\B)$, we have 
\[
\rho_{1}(\bfT ^{-1}(x))\geq \rho_{1}(x)\geq \xi_\epsilon, 
\] 
which yields  $x\in \B$, and we conclude $\bfT^{-1}(\B) \subset \B$. By using this we find
\begin{equation*}
\begin{split}
&\mu_1(\B)
=\int_{\B}\rho_{1}(x) \Ld (\d x)\geq\int_{\bfT ^{-1}(\B)}\rho_{1}(x) \Ld (\d x)
\\
& \geq 
\int_{\bfT ^{-1}(\B)}k^{2}_{\epsilon}c_{\max} \Ld (\d x)\geq\int_{\bfT ^{-1}(\B)}k^{2}_{\epsilon}\rho_{0}(x) \Ld (\d x)
=  k^{2}_{\epsilon}\mu_{0}(\bfT ^{-1}(\B)),
\end{split}
\end{equation*}
where in the last estimates we applied both
$\rho_1(x)\geq k^{2}_\epsilon c_{\max}$ and $c_{\max} \geq \rho_{0}(x)$.  By
combining the estimate above with the definition of the injective optimal
transport couple we have,
\[
k^2_{\epsilon} \mu_{0}(\bfT ^{-1}(\B))\leq\mu_{1}(B) 
=\int_{\bfT ^{-1}(\B)}q^{2}\mu_{0}(\d x), 
\]
which leads to the existence of a set $\A\subset\bfT ^{-1}(\B),$ for which
$\mu_{0}(A)>0$ and $q^{2}(x)\geq k^{2}_{\epsilon}$ for a.a.\ $ x\in\A$.  

\underline{Step 2: Comparison of $\HK$ distances.} 
We denote by $\mu^{\A}_{0}$ the restriction of $\mu_{0}$ on $\A,$ and by
$\mu^{\bfT(\A)}_{1}$ the restriction of $\mu_{1}$ on $\bfT (\A)$. Since
$(\bfT ,tq)$ is a \TRAGRO\ couple for
$\mu^{\A}_{0},t^{2}\mu^{\bfT(\A)}_{1}$ for $t\in[0,1]$, we get 
\begin{equation}
\HK^{2}\left(\mu^{\A}_{0},t^{2}\mu^{\bfT(\A)}_{1}\right)\leq \int_{\A}\left( 1 +
  (tq(x))^{2} -2tq(x)\cos\left( |\bfT (x){-} x|)\right) \right)\mu^{\A}_{0}(\d x), 
\end{equation}
from which we obtain 
\begin{equation}
  \label{estest}
\begin{split}
&\HK^{2}\left(\mu^{\A}_{0},\mu^{\bfT(\A)}_{1}\right) -	\HK
^{2}\left(\mu^{\A}_{0},t^{2}\mu^{\bfT(\A)}_{1}\right) \geq \int_{\YX}\!\!\!\left( 1 +
  (q(x))^{2} -2q(x)\cos\left( |\bfT (x){-} x|\right)\right)\mu^{\A}_{0}(\d x)
\\ 
& \hspace*{0.33\textwidth}
- \int_{\YX}\!\!\left( 1 + (tq(x))^{2} -2tq(x)\cos\left( |\bfT (x){-} x| 
    \right)\right)\mu^{\A}_{0}(\d x)
\\&
=\int_{\YX} (1{-}t^{2})q^{2}(x)-2(1{-}t)q(x)\cos\left( |\bfT (x){-} x|\right)
\mu^{\A}_{0}(\d x) 
\\&
\geq\int_{\YX} \!\left[(1{-}t^{2})q^{2}(x)-\frac{2(1{-}t)}{q(x)}q^{2}(x) 
\!\right]\mu^{\A}_{0}(\d x)\geq (1{-}t)\left(1+t-\frac{2}{k_{\epsilon}}\right)\int_{\YX}q^{2}(x)\mu^{\A}_{0}(\d x)\\&=
(1{-}t)\left(1+t-\frac{2}{k_{\epsilon}}\right)\mu^{\bfT(\A)}_1(\YX),
\end{split}
\end{equation}  
where in the last estimate, we applied $q(x)\geq k_{\epsilon}.$
Now, for $0\leq t \leq 1,$ we define the measure
\[
\mu^{t}_{1}:=\mu_{1}-\mu^{\bfT(\A)}_{1}+ t^{2}\mu^{\bfT(\A)}_{1}=\mu_{1} +
(t^{2}{-}1)\mu^{\bfT(\A)}_{1} = \left( \bm1_{X\setminus \bfT(A)} + t^2
\bm1_{\bfT(\A)}\right) \rho_1  \Ld  . 
\]
Applying Lemma \ref{split} for $\A$ and $\YX\setminus \A,$ we get
\begin{equation}
\begin{split}
&\HK^{2}(\mu_{0},\mu_{1})=\HK^{2}(\mu_{0}-\mu^{\A}_{0},\mu_{1}-\mu^{\bfT(\A)}_{1}) +   \HK^{2}\left(\mu^{\A}_{0},\mu^{\bfT(\A)}_{1}\right) \\&\geq \HK^{2}(\mu_{0}-\mu^{\A}_{0},\mu_{1}-\mu^{\bfT(\A)}_{1}) +  \HK^{2}\left(\mu^{\A}_{0},t^{2}\mu^{\bfT(\A)}_{1}\right) +  \HK^{2}\left(\mu^{\A}_{0},\mu^{\bfT(\A)}_{1}\right)- \HK^{2}\left(\mu^{\A}_{0},t^{2}\mu^{\bfT(\A)}_{1}\right)\\&\geq \HK^{2}(\mu_{0}, \mu^{t}_{1})+(1{-}t)\left(1+t-\frac{2}{k_{\epsilon}}\right)\mu^{\bfT(\A)}_1(\YX),
\end{split}
\end{equation} 
where the last inequality is a result of the sub-additivity of the squared distance
and \eqref{estest}. So, for the new measure $\mu^{t}_{1}=\mu_{1} +
(t^{2}{-}1)\mu^{\bfT(\A)}_{1}$, we have  
\begin{equation}
\label{eq:HK.lowerEst}
\frac1{2\tau}\left(\HK^{2}(\mu_{0},\mu_{1})-\HK^{2}(\mu_{0},\mu^{t}_{1})\right)
\geq
\frac{1{-}t}{2\tau} \left(1+t-\frac{2}{k_{\epsilon}}\right)\mu^{\bfT(\A)}_1(\YX) . 
\end{equation}
For later use we recall that
\begin{equation}
  \label{eq:mu1*(\YX)}
  \mu_1^{\bfT(\A)}(\YX)= \int_X q^2 \mu_0^{\A}(\d x) =\int_A q^2 \mu_0(\d x) \geq \int_A
  k_\epsilon^2 \mu_0(\d x) \geq (1{+}\epsilon)^2 \mu(\A)>0. 
\end{equation}

\underline{Step 3: Comparison of entropies.} 
On $\bfT(\A)$ we have $\rho_1^t= t^2 \rho_1 \geq t^2
\xi_\epsilon$, and the convexity of $\E$ gives 
\begin{equation*}
\begin{split}
& \E(\mu_{1})-\E(\mu^{t}_{1})=\int_{\bfT (\A)} 
 \left(E(\rho_{1}(x))-E(\rho^{t}_{1}(x))\right) \Ld (\d x) 
\\&\geq\int_{\bfT (\A)}E'(\rho_1^t) 
(\rho_{1}(x)-\rho^{t}_{1}(x)) \Ld (\d x) \geq E'(t^2 \xi_\epsilon) \int_{\bfT (\A)} 
 (\rho_{1}{-}\rho^{t}_{1}) \Ld (\d x). 
\end{split}
\end{equation*}
For all
$t\in\left((1{+}\epsilon)^{-1/2},1\right)$  we have $t^{2}\xi_{\epsilon}\geq c_\mafo{upp}$,
and therefore $E'(t^2\xi_\epsilon)\geq E'(c_\mafo{upp})$ by the monotonicity of
$E'$. With this we arrive at the lower bound 
\begin{equation}
\label{eq:Ent.lowerEst}
\begin{split}
\E(\mu_{1})-\E(\mu^{t}_{1})&\geq E'(c_\mafo{upp}) 
\int_{\bfT (A)}(\rho_{1}(x)-\rho^{t}_{1}(x)) \Ld (\d x)
\\&
=E'(c_\mafo{upp}) \int_{\bfT (\A)}(1{-}t^{2}) \rho_{1}(x) \Ld (\d x)
=(1{-}t^{2})E'(c_\mafo{upp}) \mu^{\bfT(\A)}_{1}(\YX).
\end{split}
\end{equation}

\underline{Step 4: Minimization provides contradiction.} Because $\mu_1$
is a minimizer we reach a contradiction if we find a $t\in (0,1)$ such that
\[
\eta(t):= \frac1{2\tau } \HK^2(\mu_0,\mu_1) +\E(\mu_1) - \left( \frac1{2\tau }
\HK^2(\mu_0,\mu_1^t) + \E(\mu_1^t)\right) >0.  
\]
Combining the estimates \eqref{eq:HK.lowerEst} and \eqref{eq:Ent.lowerEst} we
find, for $t\in\left((1{+}\epsilon)^{-1/2},1\right)$, the lower estimate
\[
\eta(t)\geq \ol\eta(t)\mu_1^*(\YX) \quad \text{with } 
\ol\eta(t):= \frac{1{-}t}{2\tau } \left(1+t-\frac{2}{k_{\epsilon}}\right) +
(1{-}t^{2}) E'(c_\mafo{upp}). 
\] 
Clearly, we have $\ol\eta(1)=0$ and find 
\[
\ol\eta{}'(1)= -  \frac1\tau  \left( 1 +  2\tau  E'(c_\mafo{upp})
-\frac1{k_\epsilon}\right)  
\leq - \frac1\tau \left( \frac1{k_\epsilon{-}\epsilon} -\frac1{k_\epsilon}\right)
= - \frac1\tau \:\frac\eps{(k_\epsilon{-}\epsilon)k_\epsilon}  <0. 
\]
Recalling $\mu_1^{\bfT(\A)}(\YX)>0$ from \eqref{eq:mu1*(\YX)} we obtain $\eta(t)>0$ for all
$t<1$ that are sufficiently close to $t=1$. Thus, the assumption $\mu(\B)>0$
must have been false, and we conclude $\rho_1(x) \leq \xi_\epsilon$ a.e.\ in
$X$. As $\epsilon >0$ was arbitrary, the assertion is established.  
\end{proof}

{ We proceed with the lower bound.

\begin{proposition}[Lower bound for incremental densities]
\label{pr:discr.max.princ.HK.low}
Let $X \subset \R^d$ be a compact, convex set with nonempty
interior. Furthermore let $\E$ be as in
\eqref{EntropyFunctional} with convex $E.$   Let finally $\mu_{0}\simeq \Ld $ with $\rho_{0}(x)\geq c_\mathrm{min},$ and $\mu_{1}\in \calM(\YX)$ as in
\eqref{schemerepeat}$_\HK$. Is it true that 
\begin{equation}
    \rho_1(x) \geq \min \Big\{ c_\mathrm{low},\:\frac{c_\mathrm{min}}{(1 {+} 2\tau \max\{
  E'(c_\mathrm{low}),0\})^2} \Big\} \quad \text{for all } c_\mathrm{low} \geq 0
\end{equation}
\end{proposition}
\begin{proof}
We start by setting 
\[
k_{\epsilon}:=\frac{1}{1+2\tau \,\max\{E'(c_\mafo{low}) 
    ,0\}}\:-\epsilon \leq 1-\epsilon.
\]

\underline{\em Step 1: Construction of sets.} 
In order to arrive a a contradiction we define the set
\[
\B:= \bigset{x \in X}{ \rho_1(x) \leq \xi_\epsilon } \quad \text{with
}\xi_\epsilon:=\,\min\{ (1-\epsilon)^{2}c_\mafo{low},
   k_\epsilon^2c_\mafo{min}\}
\]
and assume $\mu_{1}(\B)>0$. By Lemma \ref{uniformitypreference}, for
almost every $x\in \B$, we have 
\[
\rho_{1}(\bfT(x))\leq \rho_{1}(x)\leq \xi_\epsilon, 
\] 
which yields  $\bfT(x)\in \B$, and we conclude $\bfT(\B) \subset \B$. By using this we find
\begin{equation*}
\begin{split}
&\mu_1(\bfT(\B))
=\int_{\bfT(\B)}\rho_{1}(x) \Ld (\d x)\leq\int_{\B}\rho_{1}(x) \Ld (\d x)
\\
& \leq 
\int_{\B}k^{2}_{\epsilon}c_{\min} \Ld (\d x)\leq\int_{\B}k^{2}_{\epsilon}\rho_{0}(x) \Ld (\d x)
=  k^{2}_{\epsilon}\mu_{0}(\B),
\end{split}
\end{equation*}
where in the last estimates we applied both
$\rho_1(x)\leq k^{2}_\epsilon c_{\min}$ and $c_{\min} \leq \rho_{0}(x)$.  By
combining the estimate above with the definition of the injective optimal
transport couple we have,
\[
k^2_{\epsilon} \mu_{0}(\B)\leq\mu_{1}(\bfT(\B)) 
=\int_{\B}q^{2}\mu_{0}(\d x), 
\]
which leads to the existence of a set $\A\subset\B,$ for which
$\mu_{0}(A)>0$ and $q^{2}(x)\leq k^{2}_{\epsilon}$ for a.e.\ $ x\in\A$.  For this $\A,$ without loss of generality we can even assume that $\rho_0(x)<c_{\max},$ for some $c_{max}>0.$

\underline{Step 2: Comparison of $\HK$ distances.} 
We denote by $\mu^{\A}_{0}$ the restriction of $\mu_{0}$ on $\A,$ and by
$\mu^{\bfT(\A)}_{1}$ the restriction of $\mu_{1}$ on $\bfT (\A)$. Since
$(\bfT ,q/t)$ is a \TRAGRO\ couple for
$\mu^{\A}_{0},(1/t^{2})\mu^{\bfT(\A)}_{1}$ for $t\in[0,1]$, we get 
\begin{equation}
\HK^{2}\left(\mu^{\A}_{0},\frac{1}{t^{2}}\mu^{\bfT(\A)}_{1}\right)\leq \int_{\A}\left( 1 +
  \left(\frac{q(x)}{t}\right)^{2} -2\frac{q(x)}{t}\cos\left( |\bfT (x){-} x|)\right) \right)\mu^{\A}_{0}(\d x), 
\end{equation}
from which we obtain 
\begin{equation}
  \label{estest2}
\begin{split}
&\HK^{2}\left(\mu^{\A}_{0},\mu^{\bfT(\A)}_{1}\right) -	\HK
^{2}\left(t^{2}\mu^{\A}_{0},\mu^{\bfT(\A)}_{1}\right) \geq 
\HK^{2}\left(\mu^{\A}_{0},\mu^{\bfT(\A)}_{1}\right) -	t^{2}\HK
^{2}\left(\mu^{\A}_{0},\frac{\mu^{\bfT(\A)}_{1}}{t^{2}}\right)\geq \\
&\int_{\YX}\left( 1 +
  (q(x))^{2} -2q(x)\cos\left( |\bfT (x){-} x|\right)\right)\mu^{\A}_{0}(\d x)
\\ 
& \hspace*{0.33\textwidth}
- \int_{\YX}\left( t^{2} + q(x)^{2} -2tq(x)\cos\left( |\bfT (x){-} x| 
    \right)\right)\mu^{\A}_{0}(\d x)
\\&
=\int_{\YX} (1{-}t^{2})-2(1{-}t)q(x)\cos\left( |\bfT (x){-} x|\right)
\mu^{\A}_{0}(\d x) 
\\&
\geq\int_{\YX} \!\left[(1{-}t^{2})-2(1{-}t)q(x) 
\!\right]\mu^{\A}_{0}(\d x)\geq (1{-}t)\left(1+t-2k_{\epsilon}\right)\int_{\YX}\mu^{\A}_{0}(\d x)\\&=
(1{-}t)\left(1+t-2k_{\epsilon}\right)\mu^{\A}_0(\YX),
\end{split}
\end{equation}  
where in the last estimate, we applied $q(x)\leq k_{\epsilon}.$
Now, for $0\leq t \leq 1,$ we define the measure
\[
\mu^{t}_{1}:=\mu_{1}+(1 {-} t^{2})\mu^{\A}_{0}. 
\]
Applying Lemma \ref{split} for $A$ and $\YX\setminus A,$ we get
\begin{equation}
\begin{split}
&\HK^{2}(\mu_{0},\mu_{1})=\HK^{2}(\mu_{0}-\mu^{\A}_{0},\mu_{1}-\mu^{\bfT(\A)}_{1})
+   \HK^{2}\left(\mu^{\A}_{0},\mu^{\bfT(\A)}_{1}\right) \\
&=
\HK^{2}(\mu_{0}-\mu^{\A}_{0},\mu_{1}-\mu^{\bfT(\A)}_{1}) +
\HK^{2}\left(\mu^{\A}_{0},(1{-}t^2)\mu^{\A}_{0}+\mu^{\bfT(\A)}_{1}\right) \\
&+
\HK^{2}\left(\mu^{\A}_{0},\mu^{\bfT(\A)}_{1}\right)-
\HK^{2}\left((1-t^2)\mu^{\A}_{0} 
+t^{2}\mu^{\A}_{0},(1{-}t^2)\mu^{\A}_{0}+\mu^{\bfT(\A)}_{1}\right)
\\
&\geq \HK^{2}(\mu_{0}, \mu^{t}_{1}) 
 +\HK^{2}\left(\mu^{\A}_{0},\mu^{\bfT(\A)}_{1}\right)
 -\HK^{2}\left((1-t^2)\mu^{\A}_{0},(1{-}t^2)\mu^{\A}_{0}\right)
-\HK^{2}\left(t^2\mu^{\A}_{0},\mu^{\bfT(\A)}_{1}\right)\\
&\geq \HK^{2}(\mu_{0}, \mu^{t}_{1})+(1{-}t)\left(1+t-2k_{\epsilon}\right)\mu^{\A}_0(\YX),
\end{split}
\end{equation} 
where the last inequality is a result of the sub-additivity of the squared distance
and \eqref{estest2}. So, for the new measure, we have  
\begin{equation}
\label{eq:HK.lowerEst2}
\frac1{2\tau}\left(\HK^{2}(\mu_{0},\mu_{1})-\HK^{2}(\mu_{0},\mu^{t}_{1})\right)
\geq
\frac{1{-}t}{2\tau} \left(1+t-2k_{\epsilon}\right)\mu^{\A}_0(\YX) . 
\end{equation}

\underline{Step 3: Comparison of entropies.} 
On $\A$ we have $\rho_1^t= \rho_1 + (1{-}t^2)\rho_{0}(x)\leq \xi_\epsilon+(1{-}t^2)c_{\max}$, and the convexity of $\E$ gives 
\begin{equation*}
\begin{split}
& \E(\mu^{t}_{1})-\E(\mu_{1})=\int_{\A} 
 \left(E(\rho^{t}_{1}(x))-E(\rho_{1}(x))\right) \Ld (\d x) 
\\&\leq\int_{\A}E'(\rho_1^t) 
(\rho^{t}_{1}(x)-\rho_{1}(x)) \Ld (\d x) \leq E'( \xi_\epsilon+(1{-}t^2)c_{\max})\int_{\A} 
 (\rho^{t}_{1}-\rho_{1}) \Ld (\d x). 
\end{split}
\end{equation*}
For all
$t\in\left((1{+}\epsilon/c_{\max})^{-1/2},1\right)$  we have $\xi_{\epsilon} +(1{-}t^2)c_{\max} \leq c_\mafo{low}$,
and therefore $E'(\xi_{\epsilon} +(1{-}t^2)c_{\max})\leq E'(c_\mafo{low})$ by the monotonicity of
$E'$. With this we arrive at the lower bound 
\begin{equation}
\label{eq:Ent.lowerEst2}
\begin{split}
\E(\mu^{t}_{1})-\E(\mu_{1})&\leq E'(c_\mafo{low}) 
\int_{\A}(\rho^{t}_{1}(x)-\rho_{1}(x)) \Ld (\d x)
\\&
=E'(c_\mafo{low}) \int_{\A}(1{-}t^{2}) \rho_{0}(x) \Ld (\d x)
=(1{-}t^{2})E'(c_\mafo{low}) \mu^{\A}_{0}(\YX).
\end{split}
\end{equation}

\underline{Step 4: Minimization provides contradiction.} Because $\mu_1$
is a minimizer we reach a contradiction if we find a $t\in (0,1)$ such that
\[
\eta(t):= \frac1{2\tau } \HK^2(\mu_0,\mu_1) +\E(\mu_1) - \left( \frac1{2\tau }
\HK^2(\mu_0,\mu_1^t) + \E(\mu_1^t)\right) >0.  
\]
Combining the estimates \eqref{eq:HK.lowerEst} and \eqref{eq:Ent.lowerEst} we
find, for $t\in\left((1{+}\epsilon/c_{\max})^{-1/2},1\right)$, the lower estimate
\[
\eta(t)\geq \ol\eta(t)\mu_0^{\A}(\YX) \quad \text{with } 
\ol\eta(t):= \frac{1{-}t}{2\tau } \left(1+t-2k_{\epsilon}\right) -
(1{-}t^{2}) E'(c_\mafo{low}). 
\] 
Clearly, we have $\ol\eta(1)=0$ and find 
\[
\ol\eta{}'(1)=  \frac1\tau  \big( k_\epsilon-\left(1 {-}  2\tau
  E'(c_\mafo{low})\right)\big) \leq  \frac1\tau \left(k_{\epsilon}-\frac{1}{1+2\tau E'(c_{\mafo{low}})}\right)  <0. 
\]
Recalling $\mu_0^{\A}(\YX)>0$ from \eqref{eq:mu1*(\YX)} we obtain $\eta(t)>0$ for all
$t<1$ that are sufficiently close to $t=1$. Thus, the assumption $\mu(\B)>0$
must have been false, and we conclude $\rho_1(x) \geq \xi_\epsilon$ a.e.\ in
$X$. As $\epsilon >0$ was arbitrary, the assertion is established.  
\end{proof}}

\subsection{A single minimization step for $\SHK$}
\label{su:SingleStepSHK}

We will proceed with the theorem for the spherical Hellinger Kantorovich
describing the propagation of density bounds for the incremental 
minimization scheme for the gradient system $(\calP(\YX),\E,\SHK)$. 
 
The main argument for the spherical Hellinger Kantorovich goes as follows. If
the assumption for either bounds of $\rho_{1}$ is violated, then the fact that
the mass remains constant, guarantees the existence of two sets $\A,\B$ of
positive measure that will lead to a contradiction. More specifically, for mass
leaving $\A$, we have growth (i.e. $q>1$), with a resulting density at the
target bigger than some constant $c.$ At the same time for $\B,$ we have that
the final density is strictly smaller than $c$ and part of the mass that left
$\B$ was destroyed, (i.e. $q<1$).  One can show that it is cheaper to reduce
the growth of the mass leaving $\A$ and going $\bfT(\A)$ resulting in less
density in $\bfT(\A)$, where at the same time for the mass leaving $\B$ we
retain a portion at place instead of destroying it during the transportation,
which again leads to a cheaper cost. This way we can construct a new measure
that contradicts the optimality of $\mu_{1}.$

\begin{proposition}[Density bounds for $\SHK$]
\label{discrete maximal principle SHK}
Let $X \subset \R^d$ be a compact, convex set with nonempty
interior. Furthermore let $\E$ be as in \eqref{EntropyFunctional} with convex
$E.$ Let finally $\mu_{0}=\rho \Ld $ with
$c_{\min}\leq\rho_{0}(x)\leq c_{\max},$ and $\mu_{1}\in \calP(\YX)$ as in
\eqref{schemerepeat}$_\SHK$. Then, we have the density bounds
\begin{equation}
\label{eq:UppLowSHK}
c_{\min} \leq\rho_{1}(x)\leq c_{\max}\hspace{8pt}  
\text{almost everywhere in } X.
\end{equation}
\end{proposition}
\begin{proof} 
\underline{Step 1: Description of proof strategy.}
Let $(q,\bfT)$ be an optimal \TRAGRO\ couple from $\mu_{0}$ to $\mu_{1}$. By
Lemma \ref{uniformitypreference} we have that $\mu_{1}\simeq\calL^{d}.$ We
remind the reader that in this case, $\bfT$ is injective with no loss of
generality. We will prove that if the lower or upper bound in equation
\eqref{eq:UppLowSHK} is violated in a set of positive measure, then there exist
sets $A$ and $B$ of positive Lebesgue measure and a positive constants 
$\epsilon$ such that
\begin{enumerate}[label=\textbf{S.\arabic*}]
\item $\calL(\A)>0$ \ \ and \ \ $ \forall\, x\in \A:\ \ q^{2}(x) >
  1{+}\epsilon \ \text{ and } \ \rho_{1}(\bfT(x))> (1{+}\epsilon) \ol c.$  
  \label{S.1} 
\item $\calL(\B)>0$ \ \ and \ \ $\forall\, x\in \B:\ \    q^{2}(x)< 1{-}\epsilon
   \text{ and } \rho_{1}(x)<  \left(1 {-} \epsilon\right) \ol c$\label{S.2},
\end{enumerate}
 where $\ol c= c_{\min}$ if the lower bound is violated and $\ol c=c_{\max}$
if the upper bound is violated. 

Based on this, we can then  construct a new measure
$\mu^{t,s}_{1}$ with unit mass, but with lower density in $\bfT(\A)$ and
higher density in $\B$, resulting in a lower value in the minimizing
scheme than $\mu_{1}.$ This will contradict the assumption that $\mu_1$ is a
minimizer.\bigskip

\underline{Step 2: Construction of $A$ and $B$ if the lower bound is
  violated.}
We define the sets 
\[
X_{\leq c}:=\bigset{x \in X}{ \rho_{1}(x)\leq c} \quad \text{and} \quad 
X_{> c}:=\bigset{x \in X}{ \rho_{1}(x)> c},
\]
such that $X=X_{\leq c}\overset{.}\cup X_{>0}$. Throughout, we fix a measurable 
representative for the density $\rho_1$. 
The violation of the lower bound meant that 
\[
\exists\, c_1 < c_{\min}: \quad \Ld(X_{\leq c_1})>0.
\]

\underline{Step 2.1:  Construction of  $B$.}
By Lemma \ref{uniformitypreference}, we have
$\rho_{1}(\bfT(x))\leq \rho_{1}(x)$ a.e.\ on $X$. Hence for all $c>0$ we find
\begin{equation*} 
x\in X_{\leq c}  \ \Rightarrow \
  \rho_{1}(x)\leq c  \ \Rightarrow \ 
   \rho_1(\bfT(x))  \leq c \ \Rightarrow \  \bfT(x)\in X_{\leq c} \quad
   \text{a.e.\ on } X.
\end{equation*}
This implies $ \Ld(\bfT(X_{\leq c})) \leq \Ld(X_{\leq c} ),$ which in turn implies that
\begin{equation*}
\begin{split}
\mu_{1}(\bfT(X_{\leq c}))&=\int_{\bfT(X_{\leq c})}\rho_{1}(y) \Ld (\rmd y) 
 \leq  c \Ld (\bfT(X_{\leq c} ))
\leq c \Ld \left(X_{\leq c} \right) \leq \frac c{c_{\min}} \,\mu_{0}(X_{\leq c_{\min}}),  
\end{split}
\end{equation*} 
since $\rho_0(x)\geq c_{\min}$ a.e.\ on $X$.  However, by the definition of the
\TRAGRO\ couple $(\bfT,q)$, we have
$\mu_{1}(\bfT(X_{\leq c})) = \int_{X_{\leq c}} q^{2}\mu_{0}(\rmd x)$.  Hence
for $c\in [c_1,c_{\min}]$ there must be a set $B_c \subset X_{\leq c}$ with
$\calL( B_c)>0$ on which $q^{2}\leq c/c_{\min}$ a.e. Therefore, \ref{S.2} is
satisfied for all $\epsilon \in {]0,1{-}c_1/c_{\min}[}$.\smallskip

\underline{Step 2.2: Construction of $A$.}
In Step 2.1 we showed that $\bfT(X_{\leq c} ) \subset X_{\leq c}\cup N$ for
a null set $N$. Using that $M_0\subset M_1$ implies $\bfT^{-1}(M_{0})\subset \bfT^{-1}
(M_1)$ and that $\bfT$ is injective, we find 
\[
X_{\leq c} \overset{\text{inject}}= \bfT^{-1}(\bfT(X_{\leq c})) \subset \bfT^{-1} 
\big( X_{\leq c}\cup N\big) = \bfT^{-1} (X_{\leq c}) \cup \wt N. 
\]
This implies $ \Ld (X_{\leq c} )\leq \Ld (\bfT^{-1}(X_{\leq c} ))$, and as in
Step 2.1 we obtain  
\begin{equation}
\begin{split}\label{laters}
\mu_{1}(X_{\leq c})
&=\int_{X_{\leq c}}\rho_{1}(y) \Ld (\rmd y) 
 \leq  c \Ld (X_{\leq a} ) \leq c \Ld (\bfT^{-1}(X_{\leq c})) 
  \leq \frac c{c_{\min}} \,\mu_{0}(\bfT^{-1}(X_{\leq c})) . 
\end{split}
\end{equation} 

 We claim that there is a $c_2> c_{\min}$ such that
\begin{equation}
 \label{eq:a*.geq.cmin}
 m_1(c_2):=\mu_{1}(X_{\leq c_2})
\leq \frac12\big( m_1(c_{\min})+ m_0(c_{\min})\big) < m_0(c_{\min}), 
\end{equation} 
where $m_0(c):=\mu_0(\bfT^{-1}(X_{\leq c}))$. For this we first observe 
\[
\mu_1(X_{\leq c_{\min}})  \leq c_{\min} \Ld \big( 
X_{\leq c_{\min}}{\setminus}  X_{c_1} \big) +c_1 \Ld(X_{\leq c_1}) < c_{\min}
\Ld(X_{\leq c_{\min}}). 
\]
This implies that \eqref{laters} holds for $c=c_{\min}$ with strict inequality,
i.e.\ $m_1(c_{\min}) < m_0(c_{\min})$. We set $\delta:=(
m_0(c_{\min}){-}m_1(c_{\min}))/2>0$. 

Next, we observe that $m_0$ and $m_1$ are non-decreasing functions by
definition of $X_{\leq a}$. Finally, $m_1$ (and also $m_0$) is continuous from
the right, because $a_k\downarrow a$ implies
$X_{\leq a} = \cap_{k=1}^\infty X_{\leq a_k}$ and the measure $\mu_1$ is
continuous along non-increasing sequences.  With this, we find $c_2>c_{\min}$
such that $m_1(c_2)\leq m_1(c_{\min}){+}\delta = 
m_0(c_{\min}){-}\delta< m_0(c_{\min})$,
and \eqref{eq:a*.geq.cmin} is established.

Switching to $X_{>a}=X\setminus X_{\leq a}$ and using $\mu_j(X_{>a})=1-
\mu_j(X_{\leq a})$ we find, for all $c\in [c_{\min},c_3]$ with $c_3=\min\{
c_2, (1{+}\delta)c_{\min}\}$, the estimate
\begin{align*}
\mu_1(X_{>c})& =1 - m_1(c) \geq 1-m_1(c_2) \geq 1-m_2(c_{\min}) + \delta
\\
& = \mu_0(X_{>c_{\min}}) + \delta \geq \mu_0(X_{>c}) + \big(\frac{c}{c_{\min}}
- 1\big) \ \geq \ \frac{c}{c_{\min}}\, \mu_0(X_{>c_{\min}}). 
\end{align*}
By the definition of $(\bfT,q)$ we also have $\mu_1(X_{>c})=
\int_{\bfT^{-1}(X_{>c})} q^2 \mu_0(\d y) $, which implies, for all $c\in
{]c_{\min},c_3[}$, the existence of a
measurable set $A_c\subset X_{>c}$ with $\Ld(A_c)>0$ and $q^2>c/c_{\min}$ a.e.\
in $A_c$. Hence, \ref{S.1} is satisfied for all $\epsilon \in
(c_3{-}c_{\min})/c_{\min}$.\medskip 

\underline{Step 3.0:  Construction of $A$ and $B$ if the upper  bound
  is violated.} 
We proceed completely analogous to Step 2 by interchanging the
inequality signs and replacing $\bfT$ by $\bfT^{-1}$. The violation of the
upper bound means that
\[
\exists \, c_4> c_{\max}: \quad \Ld(X_{\geq c_4})>0, 
\]
where now $X_{\geq c}=\bigset{x\in X}{\rho_1(x)\geq c}$ and $X_{<
  c}=\bigset{x\in X}{\rho_1(x) < c}$. \smallskip

\underline{Step 3.1: Construction of the set $\A$.}
Lemma \ref{uniformitypreference}  provides $\rho_{1}(\bfT^{-1}(x))\geq
\rho_{1}(x)$ a.e.\ on $\YX$, and hence $\Ld(X_{\geq c}) \geq
\Ld(\bfT^{-1}(X_{\geq c}))$. As in Step 2.1 we have  
\begin{equation*}
\begin{split}
\int_{\bfT^{-1}(X_{\geq c})} \!\! q^2 \mu_0(\d x)=\mu_{1}(X_{\geq c})\geq c
\Ld(X_{\geq c}) 
\geq c \Ld(\bfT^{-1}(X_{\geq c}))  \geq \frac c{c_{\max}}
\mu_0\big(\Ld(\bfT^{-1}(X_{\geq c})) \big) .
\end{split}
\end{equation*}
Hence, for all $c\in {]c_{\max},c_4]}$ there exists $A_c \subset X_{\geq c} $
such that $q^2 \geq c/c_{\max}$ a.e.\ on $A_c$, i.e.\ \ref{S.1} holds for all $\epsilon \in
{]0, (c_4{-}c_{\max})/c_{\max}[}$.\smallskip

\underline{Step 3.2: Construction of the set $\B$.}  Using Lemma
\ref{uniformitypreference} and the injectivity of $\bfT$ we obtain
$\Ld(\bfT(X_{\geq c})) \geq \Ld(X_{\geq c})$ and furthermore 
\begin{equation*}
M_1(c):=\mu_{1}(\bfT(X_{\geq c} )) \geq c \Ld(\bfT(X_{\geq c})) \geq c
\Ld(X_{\geq c}) \geq \frac{c  M_0(c)}{c_{\max}} \text{ with } M_0(c):=
\mu_0(X_{\geq c}).
\end{equation*} 
Using $0<M_0(c_4)=\Ld(X_{\geq c_4})$ with $c_4>c_{\max}$ we easily see $
M_1(c_{\max}) > M_0(c_{\max})$. 

Using the $M_0$ and $M_1$ are non-increasing in $c$ and that $M_1$ is
continuous from the left, we can argue as in Step 2.2 to find $c_5< c_{\max}$ such that 
for $c\in {]c_5,c_{\max}[}$ we have 
\[
\int_{X_{<c}} \!\! q^2 \mu_0(\d y) = \int_{\bfT(X_{<c})}  \!\! \mu_1(\d y)
=\mu_1(\bfT(X_{<c})) = 1{-}M_1(c) \leq \frac{c(1{-}M_0(c))}{c_{\max}}
= \frac{c\mu_0(X_{<c})}{c_{\max}} .
\]
Thus, there exists $B_c \subset X_{<c} $ with $\Ld( X_{<c})>0$ and $q^2 \leq c/c_{\max}$
a.e.\ on $X_{<c}$, and \ref{S.2} holds for all $\epsilon \in {]0,(c_{\max} {-}
  c_5)/c_{\max} [} $.\medskip

\underline{Step 4: Construction of a counterexample.}
 We start by defining $\mu^{\A}_{0}$ the restriction of $\mu_{0}$ on $\A,$
$\mu^{\bfT(\A)}_{1}$ the restriction of $\mu_{1}$ on $\bfT(\A),$ $\mu^{\B}_{0}$ the
restriction of $\mu_{0}$ on $\B,$ and $\mu^{\bfT(\B)}_{1}$ the restriction of
$\mu_{1}$ on $\bfT(\B)$.  We define the comparison measure by 
\[
=\mu_{1} -s\mu^{\bfT(\A)}_{1} + t\mu^{\B}_{0}
\]
where small $s>0$ and $t>0$ are chosen later. 

Applying the splitting from Lemma \ref{split} to $(\bfT,q)$ we obtain
\[
\HK^{2}(\mu_{0},\mu_{1})= \HK^{2}(\mu_{0}{-}\mu^{\A}_{0} {-}\mu^{\B}_{0}, 
 \mu_{1}{-}\mu^{\bfT(\A)}_{1}{-}\mu^{\bfT(\B)}_{1})+ \HK^{2}(\mu^{A}_{0},\mu^{\bfT(\A)}_{1})
+ \HK^{2}(\mu^{B}_{0},\mu^{\bfT(\B)}_{1}).
\]
Moreover, using the decompositions 
\[
\ba{c@{\,}c@{\,}c@{\,}c@{\,}c@{\,}c@{\,}c@{\,}c@{\,}c@{\,}c}
\mu_0&= &\big(\mu_0{-}\mu_0^A{-}\mu_0^B\big)&+& \mu_0^A &+&
 t\mu_0^B &+&(1{-}t)\mu_0^B& \text{ and } 
\\
\mu_1^{t,s}&=& \big(\mu_1{-}\mu_1^{\bfT(A)}{-}\mu_1^{\bfT(B)}\big) &+& 
(1{-}s)\mu_1^{\bfT(A)}& + &t \mu_0^B& +& \mu_1^{\bfT(B)}, \ea
\]
we can apply the subadditivity for $\HK^2$ (cf.\
\cite[Lem.\,7.8]{LiMiSa18OETP}) and obtain 
\begin{align*}
\HK^2(\mu_0,\mu_1^{t,s} )&\leq
\HK^2(\mu_0{-}\mu_0^A{-}\mu_0^B,\mu_1{-}\mu_1^{\bfT(A)}{-}\mu_1^{\bfT(B)}) 
\\ &\quad 
+\HK^2(\mu_0^A,(1{-}s)\mu_1^{\bfT(A)} )+\ 0\ +\HK^2((1{-}t)\mu_0^B,\mu_1^{\bfT(B)} ). 
\end{align*}

Combining  this with the above splitting identity we find 
\begin{equation*}
\begin{split}
\HK^{2}(\mu_{0},\mu_{1}) &\geq
\HK^2(\mu_0{-}\mu_0^A{-}\mu_0^B,\mu_1{-}\mu_1^{\bfT(A)}{-}\mu_1^{\bfT(B)}) 
\\ &\quad
  + \HK^{2}(\mu^{A}_{0},\mu^{\bfT(\A)}_{1}){-}\HK^2(\mu_0^A,(1{-}s)\mu_1^{\bfT(A)} )
 \\ &\quad
+ \HK^{2}(\mu^{B}_{0},\mu^{\bfT(\B)}_{1}){-} \HK^2((1{-}t)\mu_0^B,\mu_1^{\bfT(B)} ).
\end{split}
\end{equation*}

The terms in the second and third line can be further estimated by the
conditions \ref{S.1} and \ref{S.2}, namely
\begin{align*}
\mu_1^{\bfT(A)}(X)&=\mu_1(\bfT(A))\geq (1{+}\epsilon)\mu_0(A)=(1{+}\epsilon)\mu_0^A(X)
\text{ \ and \ } 
\\
\mu_1^{\bfT(B)}(X)&=\mu_1(\bfT(B))\leq (1{-}\epsilon)\mu_0(B)=(1{-}\epsilon)\mu_0^B(X),
\end{align*}
and the scaling of $\HK^2$ in Proposition \ref{pr:ScalHK}. For $s \in
[0,s_\epsilon]$ with $s_\epsilon :=1-1/(1{+}\epsilon)^2 $ we find 
\begin{equation*}
\begin{split}
  \HK^{2}(\mu^{A}_{0},(1{-}s)\mu^{\bfT(A)}_{1})&
=\sqrt{1{-}s}\,\HK^{2}(\mu^{A}_{0},\mu^{\bfT(\A)}_{1}) - 
\big(1-\sqrt{1{-}s}\,\big)\big(\sqrt{1{-}s}\,\mu_1(\bfT(A)) - \mu_0(A)\big)
\\ & 
\leq \HK^{2}(\mu^{\A}_{0},\mu^{\bfT(\A)}_{1}).
\end{split}
\end{equation*} 
\EEE Similarly, for $t\in [0,t_\epsilon]$ with  $t_\epsilon=
1-(1{-}\epsilon)^2=\epsilon(2{-}\epsilon)$,  we have 
\begin{equation*}
\begin{split}
  \HK^{2}((1{-}t)\mu^{\B}_{0},\mu^{\bfT(\B)}_{1})&
  =\sqrt{1{-}t}\,\HK^{2}(\mu^{\B}_{0},\mu^{\bfT(\B)}_{1})
 - \big( 1 -\sqrt{1{-}t}\,\big)\big(\sqrt{1{-}t}\, \mu_0(B) - \mu_1(\bfT(B))
 \big)\\
& \leq  \HK^{2}(\mu^{\B}_{0},\mu^{\bfT(\B)}_{1}).
\end{split}
\end{equation*}

We can now choose $t_* \in {]0,t_\epsilon[}$ and $s_* \in {]0,s_\epsilon[}$
with $s_*\mu^{\bfT(A)}_{1}(X) = t_*\mu^{B}_{0}(X)$ such that 
$\mu^{t_*,s_*}(\YX)=1$.  By the above construction we have
$\HK^{2}(\mu_{0},\mu_{1})\geq \HK^{2}(\mu_{0},\mu^{t_*,s_*}_{1}) $ which
implies on $(\calP(X),\SHK)$ the desired estimate 
\[
\SHK^2(\mu_{0},\mu_{1})\geq \SHK^{2}(\mu_{0},\mu^{t_*,s_*}_{1}).
\]
To see that $\E(\mu_{1}^{t,s})<\E(\mu_{1}),$ we repeat the arguments in Steps 3
and 4 of the proofs of Propositions \ref{pr:discr.max.princ.HK.upp} and
\ref{pr:discr.max.princ.HK.low}.
\end{proof}

\subsection{Many minimization steps}
\label{su:SingeStepHK}

Applying the MM scheme means to apply the minimization problems
\eqref{schemerepeat} iteratively. 
If we repeat the minimization problems, we will show that the density bounds
are such that we keep good a priori bounds that depend only on the actual time 
$t=n\tau$ but not on the number of steps.

For the MM scheme of $(\calM(\YX),\E,\HK)$ we have derived the upper and lower
bounds for $\rho_1$ that depend on 
\[
\ul \rho_0:=\mafo{ess\,inf}\: \rho_0 \quad \text{and} \quad  
\ol \rho_0= \mafo{ess\,sup} \:\rho_0
\]
in the form 
\begin{align}
\label{eq:HK.iter.gen}
\HK:\ &\min \left\{ b, \frac{\ul\rho_0}{(1
  {+} 2\tau \max\{    E'(b),0\})^2} \!\right\} 
   \leq \rho_1(x) \leq 
\max\left\{ a, \frac{\ol\rho_0 }{(1{+}2\tau \min\{E'(a),0\})^2} \!\right\},
\\
\label{eq:SHK.iter}
\SHK:\ & \phantom{\min  b \frac{\ul\rho_0}{(1
  {+} 2\tau \max\{ E'(b),0\})^2 \quad}}\ul\rho_0 \leq \rho_1(x) \leq \ol\rho_0,
\end{align}
where $a,b>0$ are arbitrary as long as $2\tau E'(a)>-1$. 

Thus, when constructing $\mu^\tau_n$ by the MM scheme
\eqref{schemerepeat} we 
easily can apply these bounds and obtain $\mu^\tau_n = \rho_n\nu$ with the
corresponding density bounds. If $E'$ has a positive zero $c_*>0$, then 
changes sign, we immediately obtain a global bound for all iterates  in  the
form 
\[
\min\{ c_* , \ul\rho_0\} \leq   \rho_k(x) \leq \max\{ c_*, \ol\rho_0\} \quad
\text{a.e.\ in }\YX. 
\]
Thus, the more difficult cases are when $E$ is either strictly decreasing (i.e.\ $E'(c)
< 0$ for all $c>0$) of strictly increasing (i.e.\ $E'(c)>0$ for all
$c>0$). Both these cases  can occur and are relevant, e.g.\ for the choices
$E(c)=- \sqrt{c}$ or $E(c)=c^2$. 

Also in these general cases we are able to provide suitable upper and lower
density bounds that only depend on $k\tau$, which is the original time in the
gradient-flow equation.

\begin{proposition}[General lower and upper density bounds for
  $\HK$]\label{pr:UppLow.HK}  Assume that $E:[0,\infty) \to \R$ is lsc, convex.
Assume $0<\ul\rho_0 \leq \rho(x) \leq \ol\rho_0<\infty$ and set 
and set 
\[
\ul S:= \sup\bigset{ E'(c) }{c\in (0,\ul\rho_0] } \leq \ol S :=\inf\bigset{ E'(c)
}{c\geq \ol\rho_0}>-\infty.
\]
Assume further $\tau \ol S \geq -1/4$,     
then for all $k\in \N$ we have $\mu^\tau_k= \rho^\tau_k \nu$ where
  $\rho^\tau_k$ satisfies, for all $k\in \N$, the general density bounds
\begin{equation}
  \label{eq:GenIterLowUpp}
 \ul\rho_0\,\ee^{-4\max\{\ul S,0\}\, k\tau} \leq \rho^\tau_k(x) \leq  
\ol\rho_0\,\ee^{ 8\max\{- \ol S, 0\} \,k\tau} \text{ a.e.\ in }\YX . 
\end{equation} 
\end{proposition} 
\begin{proof} The result follows essentially by iterating
  \eqref{eq:HK.iter.gen}.

\ul{Step 1: Upper bound.}
We construct a nondecreasing sequence $(a_k)$ via
  $a_0=\ol\rho_0$ and the recursion $a_{k+1}= a_k/(1{+}2\tau \min\{
  E'(a_k),0\})^2\geq a_k$. Using the upper estimate in \eqref{eq:HK.iter.gen}
  an an induction over $k$, we obtain for
  $\ol\rho_k:=\mafo{ess\,sup}\rho^\tau_k$,  the estimate  
\[
\rho^\tau_k(x)\leq \ol\rho_k \leq a_k \quad \text{for all } k \in \N. 
\]
However, the monotonicities of $(a_k)$ and of $E'$ imply 
$ \min\{ E'(a_k),0\} \geq  \min\{ E'(a_k),0\}= \min\{ \ol S,0\}=$. Thus,
the recursion for $a_k$ implies 
\[
a_k \leq \frac{a_0}{\big(1{+} 2\tau  \min\{ \ol S,0\}\big)^{2k}}
\overset{*}\leq a_0\big(1{-} 4\tau  \min\{ \ol S,0\}\big)^{2k} \leq  \ol
\rho_0\, \ee^{8\max\{-\ol S,0\} \,k\tau},
\] 
where in $\overset{*}\leq$ we used $\tau\ol S \geq -1/4$ and the estimate
$(1{-}r)^{-2} \leq 1{+}2r$ for all $r\in [0,1/2]$. 

\ul{Step 2: Lower bound.} We proceed analogously and define $(b_k)_{k\in \N}$ 
via $b_0=\ul\rho_0$ and the recursion $b_{k+1} =
b_k/(1{+}2\tau \max\{E'(b_k),0\})^2 \leq b_k$. This yields
\[
\rho^\tau_k(x)\geq \ul\rho_k \geq b_k \geq \frac{\ul\rho_0}{\big(1{+}2\tau
  \max\{\ol S,0\}\big)^{2k}} \geq \ul\rho_0 \,\ee^{-4\max\{\ul S,0\}\,k\tau}, 
\] 
which is the desired result. 
\end{proof} 

The upper bound can be improved from exponential in $t=k\tau$ into quadratic,
if we impose a suitable lower bound for $E'$, namely 
$E'(c) \geq -e_*/\sqrt{c}$ for $c\geq c_*$ and some $e_*\geq 0$. 

\begin{proposition}[Iteration of the upper bound for $\HK$]
\label{pr:Iteration.HK} Assume that $E:[0,\infty) \to \R$ is lsc, convex, and
satisfies $E'(c) \geq -e_*/\sqrt{c}$ for $c\geq c_*$. 
Then, $\ol\rho_k = \mafo{ess\,sup}\,\rho^\tau_k$  satisfies 
\begin{equation}
  \label{eq:IterUpperEst}
  \ol\rho_k \leq \big( \sqrt{\max\{ \ol\rho_0, c_*, 4\tau^2 e_*^2\} } + 4e_* \,k
\tau\,\big)^2 \text{ for all } k \in \N,
\end{equation}
\end{proposition}
\begin{proof} \underline{Iteration of the upper bound:}  We apply the upper
  bound for $\HK$ iteratively using $\ol \rho_k = \mafo{ess\,sup}\, \rho_k$
  with $\ol \rho_0=c_{\max}$ by
  choosing suitable $a=a_k$. Clearly, the estimate is monotone in $\ol\rho_k$,
  hence we may replace $\ol\rho_0$ by $\max\{\ol\rho_0, c_*, 4\tau^2
  e_*^2,c_{max}\}$. Thus, we obtain $\ol\rho_k\geq \ol\rho_0\geq c_*$ and
  choosing $a_k=\ol\rho_k$ is admissible because $2\tau E'(a_k)\geq  -2\tau
  e_*/\sqrt{\ol\rho_k} \geq - 1/2>-1$. Moreover, we obtain
\[
\sqrt{\ol\rho_{k+1}} 
 \leq \frac{\sqrt{\ol\rho_k}} {1+ 2\tau e_*/ \sqrt{\ol\rho_k} }  
 \leq \sqrt{\ol\rho_k}\:\big( 1 +4\tau e_* /\sqrt{\ol\rho_k}\big) 
   =  \sqrt{\ol\rho_k} + 4 \tau e_*.
\]
where we used the estimate $(1{-}\alpha)^{-1} \leq 1+ 2 \alpha $ for
  $\alpha \in [0,1/2]$. Thus, the desired estimate \eqref{eq:IterUpperEst}
  follows. 
\end{proof}

\section{Existence for EVI using local $\kappa$-concavity}
\label{se:GenSavare's method}

Following \cite{Sava07GFDS, LasMie19GPCA} we first give precise definitions of
geodesics curves in a general metric space $(\calX,\sfd)$ and the local-angle
condition (LAC). Based on these fundamental concepts, \cite{Sava11?GFDS,
  MurSavII+III} introduces the two geometry-descriptive functions
$\langle \cdot, \cdot\rangle_{\rmu\rmp}$ and $\invd(\cdot, \cdot)$ (cf.\
Definition \ref{de:TwoGeodesics}) that allow us to quantify the relationship
between two geodesics emanating from the same point. Next we discuss semiconvex
and semiconcave functions in the sense of geodesic $\kappa$-concavity or
geodesic $\lambda$-convexity. When the squared distance
$\frac12\sfd_{\calX}^2(t, \ob)$ is $\kappa$-concave along a geodesic $\geo{xy}$
with respect to some observer $\ob$ the derivative of
$t\mapsto \frac12\sfd^2_{\calX}(\geo{x y}(t), \ob)$ can be estimated in terms
of the two quantities mentioned above. To illustrate the concepts of the LAC
and the semiconcavity of the squared distance, we consider a few simple
examples in Section \ref{su:Exa.LAC.k-concave}.

Finally, we discuss Evolutionary
Variational Inequalities EVI$_\lambda$ for a metric gradient system
$(\calX,\phi,\sfd_{\calX})$ and show that solutions can be constructed via the minimizing
movement scheme, if $\phi$ is strongly $\lambda$-convex with compact
sublevels and the closure of $\mathrm{dom}(\phi)$ can be written as the closure
of the union of sets $\A_{\kappa}$, where $\frac12\sfd^2_\calX$ is
$\kappa$-concave. See Theorem \ref{th:MainExistTheo} for the exact statement.

The main arguments for our existence theory are based on results developed
in \cite{Sava07GFDS, Sava11?GFDS} which was a prelude for the work in
\cite{MurSav20GFEV} and \cite{MurSavII+III}. Hence, we provide the
corresponding result here without our own proofs expecting a soon publication
of the latter work. This contains in particular Proposition
\ref{pr:EstimDerDist} and the estimates on the Minimizing Movement scheme in
Section \ref{su:MinimMovem}. We would like however to remark that we
extend some of the results in \cite{MurSavII+III} by weakening the assumption
that the squared distance must be universally $K$-concave for some $K>0.$
Instead, we assume that we have concave bounds $K_{n}$ for a collection of
nested sets $A_{n}$, whose closure is the domain of $\phi.$ Proofs that depend
on this modified assumption are provided in both versions of the paper.  
 
\subsection{Geodesic spaces and the local angle condition (LAC)} 
\label{su:GeodSpaces}

We now provide some basic definitions for geodesics in metric spaces and some
of their properties.

\begin{definition}[Geodesics]
\label{de:Geod}
Let $(\calX,\sfd_{\calX})$ be a metric space. A curve
$\geo{x y}: [0,1] \to \calX$ is called a \emph{(constant-speed) 
  geodesic joining $x$ to $y$} if
\begin{align*}
&\geo{x y}(0)=x,\quad \geo{x y}(1)=y, \quad\text{and }
\\
&\sfd_{\calX}(\geo{x y}(t_{1}),\geo{x y}(t_{2}))=|t_{2}{-}t_{1}|\sfd_{\calX}(x,y)
\text{ for all } t_{1},t_{2}\in[0,1].
\end{align*}
We will denote the set of all such geodesics with $\Geod(x,y)$.

The metric space $(\calX,\sfd_{\calX})$ is called a \emph{geodesic space}, if for all 
points $x,\,y \in \calX$ the set $\Geod(x,y)$ is nonempty. The metric
$\sfd_{\calX}$ is then called a \emph{geodesic distance}. 
\end{definition}

For our theory  it will be important to introduce the concept of geodesic covers
of an arbitrary set. It plays the role of the convexification in Banach spaces,
however, our covering notion is not idempotent, i.e.\ $\GeoCov{\A} \subsetneqq
\GeoCov{\left(\GeoCov{\A}\right)}$ is possible.   

\begin{definition}[Geodesic cover]
\label{de:GeodCover}
Let $(\calX,\sfd_{\calX})$ be a metric space and $\A$ a subset of $\calX$. We define
the \emph{geodesic cover} $\GeoCov{\A}$ of $\A$ by 
\begin{equation}
  \GeoCov{\A}:= \bigset{\geo{xy}(t)}{ t\in [0,1] ,\
    \geo{xy}\in\Geod(x,y),\  x,y\in\A }.
\end{equation}
\end{definition}

In the following we introduce the two geometric concepts in geodesic metric
spaces called the \emph{local angle condition} (LAC) and
\emph{$\kappa$-concavity}.  These properties are going to be utilized in the sequel
to prove that the curves occurring by geodesically interpolating the points
produced by the minimizing scheme, converge to solutions of the
EVI$_\lambda,$ when the minimization step $\tau$ tends to zero. 
In order to introduce LAC, we first introduce the notions of
comparison angle between three points and of local angle between two
geodesics emanating from the same point. 

\begin{definition}[Comparison and local angles]
\label{de:CompLocAngle}
Let $(\calX,\sfd_{\calX})$ be a metric space. For three points
$x,y,z\in \calX$ with $x\not\in\{ y,z\}$, we set
\begin{equation*}
  A(x;y,z):=\frac{\sfd_{\calX}^{2}(x,y)+\sfd_{\calX}^{2}(x,z)-\sfd_{\calX}^{2}(y,z)}
  {2\sfd_{\calX}(x,y)\sfd_{\calX}(x,z)}  \in [-1,1] 
\end{equation*}
and define the \emph{comparison angle} $\varangle(x;y,z)\in[0,\pi]$
with vertex $x$ via
\begin{equation*}
{\varangle(x;y,z):=\arccos(A(x;y,z))
        \in [0,\pi].  }
\end{equation*}

Let $\geo{x y}$ and $\geo{x z}$ be two geodesics in $\calX$ emanating from point
$x=\geo{x y}(0)=\geo{x z}(0).$ The \emph{upper angle}
$\sphericalangle_{\rmu\rmp}(\geo{x y},\geo{x z})\in[0,\pi]$ and the \emph{lower
  angle} $\sphericalangle_{\rml\rmo}(\geo{x y},\geo{x z})\in[0,\pi],$ between
$\geo{x y}$ and $\geo{x z}$ are defined by
\begin{equation}
	\label{eq:LocAng}
\begin{aligned}
\sphericalangle_{\rmu\rmp}(\geo{x y},\geo{x z})&:=\limsup_{s, t \downarrow 0}
	\varangle(x,\geo{x y}(s),\geo{x z}(t)) , 
\\
\sphericalangle_{\rml\rmo}(\geo{x y},\geo{x z})&:=\,\liminf_{s, t \downarrow 0}\:
	\varangle(x;\geo{x y}(s),\geo{x z}(t)) . 
\end{aligned}
\end{equation}   
If $\sphericalangle_{\rmu\rmp} (\geo{x y},\geo{x z})=
\sphericalangle_{\rml\rmo} (\geo{x y},\geo{x z})$ holds, we say that the 
\emph{local angle} exists in the strict sense and write
$\sphericalangle(\geo{x y},\geo{x z})$.
\end{definition}

The local angle condition concerns three geodesics emanating from one
point and states that the sum of the three local angles does not exceed $2\pi$.  

\begin{definition}[Local angle condition]
\label{altdef}
We say that a point $x$ of a geodesic metric space $(\calX,\sfd)$
satisfies \emph{LAC}, if for any three geodesics
$\geo{x y},\geo{x z}, \geo{x w}$ emanating from $x,$ we have
\[
\sphericalangle_{\rmu\rmp}(\geo{x y},\geo{x z}) + \sphericalangle_{\rmu\rmp}(\geo{x z},\geo{x w}) +
\sphericalangle_{\rmu\rmp}(\geo{x w},\geo{x y}) \leq 2\pi. 
\]
\end{definition}

In geodesic spaces $(\calX,\sfd)$ the set of all geodesics emanating from a point
$x$ may be considered as a (nonlinear) surrogate of the tangent space 
defined in the case of manifolds. We now introduce a kind of scalar
product between two geodesics emanating from one point as a generalization of
the classical inner product on the tangent space.  Moreover, we define the
function $\invd$ that measures how much two geodesic curves emanating from one
point are exactly opposite to each other.

\begin{definition}[Comparing two geodesics] 
\label{de:TwoGeodesics}
 In the geodesic space $(\calX,\sfd_{\calX})$ consider three points
$x,y,z \in \calX$ and two geodesics $\geo{x y} \in \Geod(x, y),$ $\geo{x z} \in \Geod(x,z)$
and define 
\begin{align}
\label{eq:TwoGeod.SP}
& \begin{aligned} 
  &\langle\geo{x y},\geo{x z} \rangle_{\rmu\rmp}
    :=\sfd_{\calX}(x,y)\sfd_{\calX}(x,z) 
      \cos(\sphericalangle_{\rmu\rmp}(\geo{x y},\geo{x z}))  \\
  &\qquad =\liminf_{s,t\downarrow 0}\frac1{2st}\left(\sfd_{\calX}^{2}(x,\geo{x y}(s))
   +\sfd_{\calX}^{2}(x,\geo{x z}(t))-\sfd_{\calX}^{2}(\geo{x y}(s),\geo{x z}(t))  \right)
  \end{aligned}
\\
\label{eq:Delta2}
&
\invd^{2}(\geo{x y},\geo{x z})=\sfd_{\calX}^{2}(x,y)+\sfd_{\calX}^{2}(x,z)+2
\langle\geo{x y},\geo{x z}\rangle_{\rmu\rmp}\geq 0. 
\end{align}
\end{definition}
 The second form of $\langle\geo{x y},\geo{x z} \rangle_{\rmu\rmp}$ given
in \eqref{eq:TwoGeod.SP} is easily derived from the above definition when
taking into account that $\cos$ is decreasing on $[0,\pi]$, which turns the
limsup in $\sphericalangle_{\rmu\rmp}$ into a liminf. 

Considering a Hilbert space with scalar product $(x|y)$ and norm $\| z \|$
the geodesic curves are given via
$\geo{x y}=(1{-}t) x+t y,\, \geo{x z}=(1{-}t) x+t z$ and we easily find
\[
\langle\geo{x y},\geo{x z} \rangle_{\rmu\rmp} =\left(y{-}x\big|
z{-}x\right) \quad \text{ and } \quad 
\invd^2(\geo{x y},\geo{x z})= \big\| (y{-}x) - (z{-}x) \big\|^2,
\]
i.e.\ we have $\invd^2(\geo{x y},\geo{x z})=0$ if and only if $x =
\frac12(y{+}z)$. For general geodesic spaces we may consider a geodesic 
$\geo{yz}\in \Geod(y,z)$, choose $x$ as the midpoint 
$\geo{yz}(1/2)$, and set 
\[
\geo{x y}(t)=\geo{yz}\left((1{-}t)/2\right) \quad \text{and} \quad 
\geo{x z}(t)=\geo{yz}\left((1{+}t)/2\right) \quad \text{for }t\in
[0,1]. 
\]
Then, $\geo{x y}\in \Geod(x,y),$ $\geo{x z}\in \Geod(x,z)$ and
$\invd^2(\geo{x y},\geo{x z})=0$.

\begin{GSproof}
\COLORW
Subsequently, we will consider sequences $(\widehat x_n)_{n=0,1,...,N}$ obtained from
incremental minimization and will look at triples $(x,y,z)=(\wh
x_{n-1},\wh x_n,\wh x_{n+1})$ such that the above measures allow us to control
the changes in speed and direction of the discrete evolution. 

Now we will provide a Lemma that gives a bound of the first derivative of the
squared distance along a geodesic $\geo{x y} \in \Geod(x,y)$ with respect to
an arbitrary observer $\ob \in\calX$ in term of the inner product
$\langle\geo{x y},\geo{x \ob}\rangle_{\rmu\rmp}$ of the two geodesics.

\begin{proposition}\label{pr:FirVarInnProd}
Let $x,y,\ob $ in $\calX,$ and $\geo{x y}\in\Geod(x,y),$ then we have
\begin{equation}
\label{eq:FirVarInnProd}
\limsup_{t\downarrow 0} \frac{\sfd_{\calX}^{2}(\geo{x y}(t),\ob )
  -\sfd_{\calX}^{2}(x,\ob )}{2\,t}
\leq 	-\langle\geo{x y},\geo{x \ob}\rangle_{\rmu\rmp} \,.
\end{equation}
\end{proposition}
\begin{proof}
 We start from the left-hand side and use the second definition of
$\langle\geo{x y},\geo{x \ob}\rangle_{\rmu\rmp}$. The joint limsup for
$s,t\downarrow 0$ is certainly bigger or equal the iterated limsup, which gives 
\begin{equation*}
\begin{split}    
&-\langle\geo{x y},\geo{x \ob}\rangle_{\rmu\rmp}
\geq\limsup_{s\downarrow 0}\left(\!\limsup_{t\downarrow 0} 
  \frac{\sfd_{\calX}^{2}(\geo{x y}(t),\geo{x \ob}(s)) 
  {-} \sfd_{\calX}^{2}(x,\geo{x y}(t))
  {-}\sfd_{\calX}^{2}(x,\geo{x \ob}(s))}{2ts}\right). 
\end{split}
\end{equation*}
 With $\sfd_{\calX}^{2}(x,\geo{x y}(t))=t^2 \sfd_{\calX}^2(x,y)$  we have 
$ \lim_{t\downarrow 0}\frac1{2ts} \sfd_{\calX}^{2}(x,\geo{x y}(t))=0$ and can
drop the middle term. Via $a^2-b^2=(a{-}b)(a{+}b)$ we obtain 
\begin{align*}
&-\langle\geo{x y},\geo{x \ob}\rangle_{\rmu\rmp}\geq\limsup_{s\downarrow 0} \left(
\limsup_{t\downarrow 0}\frac{\sfd_{\calX}^{2}(\geo{x y}(t),\geo{x\ob}(s))-\sfd_{\calX}^{2}(x,\geo{x \ob}(s))}{2ts} \right)
\\&
\geq\limsup_{s\downarrow 0}\limsup_{t\downarrow 0} \frac{A_-(t,s)}{t}\,
\frac{A_+(t,s)}{2s} \ \ 
\text{ with } A_\pm(t,s)= \sfd_{\calX}(\geo{xy}(t),\geo{x \ob}(s))\pm
\sfd_{\calX}(x,\geo{x \ob}(s)). 
\end{align*}
 We easily see $\lim_{t\downarrow 0} A_+(t,s)= \sfd_{\calX}(x,\geo{x \ob}(s))+
\sfd_{\calX}(x,\geo{x \ob}(s))=2s \,\sfd_{\calX}(x,\ob)$.

For $A_-(t,s)$ we use the triangle  inequality and $\geo{x \ob}\in\Geod(x,\ob )$  to obtain 
\begin{align*}
A_-(t,s)&=  \underbrace{\sfd_{\calX}(\geo{x y}(t),\geo{x \ob}(s)) 
 + \sfd_{\calX}(\geo{x \ob}(s), \ob)}    -  \underbrace{(\sfd_{\calX}(\geo{x \ob}(s), \ob)  
 + \sfd_{\calX}(x,\geo{x \ob}(s)))}.
\\
&\geq  \qquad  \qquad\sfd_{\calX}(\geo{x y}(t), \ob)  \qquad\qquad 
  -  \qquad\qquad \sfd_{\calX}(x,\ob) 
\end{align*} 
Together we arrive at the estimate 
\begin{align*}
-\langle\geo{xy},\geo{x \ob}\rangle_{\rmu\rmp}&\geq \limsup_{s\downarrow 0} 
\left( \limsup_{t \downarrow 0} \frac{\sfd_{\calX}(\geo{x y}(t),\ob )-
  \sfd_{\calX}(x,\ob)} {t} \: \sfd_{\calX}(x,\ob)\right)
 \\& =\limsup_{t \downarrow 0} \frac{\sfd_{\calX}(\geo{x y}(t),\ob ) {-}
  \sfd_{\calX}(x,\ob)} {t} \;\lim_{t\downarrow 0}\frac{\sfd_{\calX}(\geo{x y}(t),\ob )
 {+}   \sfd_{\calX}(x,\ob)}{2} \\
&= \limsup_{t \downarrow 0} \frac{\sfd_{\calX}^2(\geo{x y}(t),\ob )-
  \sfd_{\calX}^2(x,\ob)} {2\,t}\:.   
\end{align*} 
This is the desired result. 
\end{proof}

For geodesic spaces satisfying  LAC there holds an inequality of Cauchy-Schwartz type, 
which was first formulated in \cite[Cor.\,9.11]{Sava11?GFDS}. It will be
crucial for studying derivatives of distances along geodesics in so-called
semiconcave geodesic spaces, see Proposition \ref{pr:EstimDerDist}. 

\begin{proposition}
\label{pr:LAC.Estim}
If the geodesic space $(\calX,\sfd_{\calX})$ satisfies LAC at
$x \in \calX$, then  for all $x,\,y,\,\ob\in \calX$ and all
$\geo{x y}\in \Geod(x,y),\geo{x z}\in \Geod(x,z)$ we have the estimate 
\begin{equation}
\label{eq:lacresult}
\langle \geo{x y},\geo{x \ob}\rangle_{\rmu\rmp} + 
\langle \geo{x \ob},\geo{x z}\rangle_{\rmu\rmp}
\geq -\sfd_{\calX}(x,\ob )\invd(\geo{x y},\geo{x z})
\end{equation}
\end{proposition}
\begin{proof}
 We introduce the abbreviations
\[\begin{split}
& a_{y} =\sfd_{\calX}(x,y), \quad a_{z} =\sfd_{\calX}(x,z), \quad a_{\ob} =\sfd_{\calX}(x,\ob), \quad 
\\& \theta_{y x \ob} =\sphericalangle_{\rmu\rmp}(\geo{x y}, \geo{x \ob}), 
\quad\theta_{\ob x z} =\sphericalangle_{\rmu\rmp}(\geo{x z}, \geo{x \ob}), 
\quad \theta_{zxy} =\sphericalangle_{\rmu\rmp}(\geo{x z}, \geo{x y}), 
\end{split}\] 
where $\theta_{y x \ob},\theta_{\ob x z},\theta_{zxy}\in [0,\pi]$ and $\theta_{y x \ob} + \theta_{\ob x z} + \theta_{zxy}\leq 2\pi$
because of  LAC. 

In the case $a_{\ob}=0$ estimate \eqref{eq:lacresult} holds because all three terms
are $0$. Hence, we may assume $a_{\ob}>0$ and after dividing all three terms,
\eqref{eq:lacresult} is equivalent to 
\begin{equation}
  \label{eq:ACEstim.simple}
  \left(a_{y}^2 + a_{z}^2 +2a_{y}a_{z} \cos\theta_{zxy} \right)^{1/2} \geq - a_{y} \cos \theta_{y x \ob}
  - a_{z} \cos \theta_{\ob x z}. 
\end{equation}
To establish this inequality we will maximize the right-hand side (RHS) with respect
to $\theta_{y x \ob}$ and $\theta_{\ob x z}$ . For this we distinguish two cases. Without loss of generality, we
may assume $a_{y},\, a_{z}>0$. 

\underline{Case $\theta_{\ob x z}+\theta_{z x y} \in [0,\pi]$:} For $\theta_{y x \ob} \in [0,\pi]$ the RHS 
is maximized at $\theta_{y x \ob} =\pi$. For $\theta_{\ob x z}\in [0,\pi{-}\theta_{z x y}]$ the
maximum of the simplified RHS is attained at $\theta_{\ob x z}= \pi{-}\theta_{z x y}$. Thus,
it remains to show 
\[
 \left(a_{y}^2 + a_{z}^2 +2a_{y}a_{z} \cos\theta_{z x y} \right)^{1/2}\geq  a_{y} + a_{z}\cos \theta_{z x y}. 
\]
However, this identity follows from elementary arguments in Euclidean
geometry. 

\underline{Case  $\theta_{\ob x z}+\theta_{z x y} \in [\pi,2\pi]$:} For $\theta_{y x \ob} \in [0,\pi]\in
[0,2\pi{-}\theta_{z x y}{-}\theta_{\ob x z}] $ RHS is maximized at $\theta_{y x \ob}=
2\pi{-}\theta_{z x y}{-}\theta_{\ob x z}$ giving $-a_{y}\cos \theta_{y x \ob} = -a_{y} \cos
(\theta_{z x y}{+}\theta_{\ob x z})$.  For $\theta_{\ob x z}\in [\pi{-}\theta_{z x y},\pi]$ maximum of the
simplified RHS is attained at $\theta_{\ob x z}=\pi$. This leads to the remaining
inequality \[ \left(a_{y}^2 + a_{z}^2 +2a_{y}a_{z} \cos\theta_{z x y} \right)^{1/2}\geq  a_{y}\cos
\theta_{z x y} + a_{z},\] which holds as in the other case. 
\end{proof}
\COLORWend
\end{GSproof}

\subsection{Semiconvex and semiconcave functions}
\label{su:SemiCvxCcv}

For a geodesic metric space $(\calX,\sfd_{\calX}),$ we now provide the definition of
$\kappa$-(semi)concavity and $\lambda$-(semi)convexity of a function
$\phi:\calX \to (-\infty,\infty]$ along a geodesic $(\calX,\sfd_{\calX})$ or of a
functional $F:\calX\times \calX \to (-\infty,\infty],$ along some geodesic with
respect to some observer.  For this, we recall that a function
$f:[0,1] \to (-\infty,\infty]$ is called $\kappa$-concave or
$\lambda$-convex, if the mapping $t \mapsto f{-} \kappa t^{2}/2$ is concave
or $t \mapsto f{-} \lambda t^{2}/2$ is convex, respectively. We
emphasize that $\kappa$ and $\lambda$ can lie in all of $\R$. Subsequently, we
will shortly say $\kappa$-concave and $\lambda$-convex.

For a $\kappa$-concave function $f:[0,1] \to (-\infty,\infty]$ we have the
inequality
\begin{equation}
  \label{eq:asasa}
f \text{ $\kappa$-concave}\quad \Longrightarrow \quad \forall\,t\in [0,1] :\   f(t)
+ \frac12 \,\kappa \,t(1{-}t) \geq (1{-}t)\, f(0) + t\, f(1) .  
\end{equation}

For the following definition we recall that the elements of $\Geod(x,y )$ are constant-speed geodesics of length and speed
$\sfd_{\calX}(x,y)$.

\begin{definition}[$\kappa$-concavity/convexity]
\label{de:KCcv.lambdaCvx} 
 Let $(\calX,\sfd_{\calX})$ be a geodesic space, $\A$ and $\B$ subsets of $\calX$,
and $\kappa \in \R$.

(A) A function $\phi: \calX \to  (-\infty,\infty]$ is called 
\emph{$\kappa$-concave (convex) on $\A$}, if for all $x, y\in \A$
there exists an $\geo{x y} \in 
\Geod(x,y)$ such that the function $t \mapsto \phi(\geo{xy}(t))$ is $\kappa\,
\sfd_{\calX}^2(x,y)$-concave (convex). $\phi$ is called \emph{strongly
  $\kappa$-concave (convex) on $\A$}, if the previous condition holds for all $\geo{xy}\in
\Geod(x,y)$. If $\A=\calX$, then we simply that $\phi$ is (strongly)
$\kappa$-concave (convex). 

(B) For $x,y\in\calX,$ we say that a functional
$F:\calX\times\calX \to (-\infty,\infty]$ is $\kappa$-concave (convex) along a
geodesic $\geo{xy}\in \Geod(x,y)$ with respect to the observer
$\ob \in\calX$, if $t \mapsto F(\geo{x y}(t),\ob )$ is
$\kappa\,\sfd_{\calX}^{2}(x,y)$-concave (convex). Furthermore, we say
that $F$ is $\kappa$-concave (convex) in $\A$ with respect to observers from
$\B$, if for every couple of points $x,y \in \A,$ there exists a
geodesic $\geo{xy}\in \Geod(x,y)$ such that for every $\ob \in\B$ we have
that $F$ is $\kappa$-concave (convex) along $\geo{xy}$ with respect to $\ob $.
We finally say that $F$ is \emph{strongly} $\kappa$-concave (convex) in $\A$
  with respect to observers from $\B$, if for every couple of points
  $x,y\in\A$, all geodesics $\geo{xy} \in \Geod(x,y)$, and all
  $\ob \in\B$ the function $F$ is $\kappa$-concave (convex) along $\geo{xy}$ with
  respect to $\ob $.
\end{definition}
In the previous definition, the points $\geo{xy}(t)$ don't have to lie in $\A$ for 
$t\in(0,1)$. Of course, we have $\geo{xy}(t) \in \GeoCov{\A}$ for all $t \in
[0,1]$.

A crucial step in the convergence theory of minimizing movement
solutions is to exploit the $\kappa$-concavity of the squared distance with
respect to suitable observers $\ob$. The point is that one obtains an upper
estimate of the upper right Dini derivative 
$\frac{\rmd^{+}}{\rmd t}\frac{1}{2}\sfd_{\calX}^{2}(\geo{xy}(t),\ob )$ in terms of the two
geometric quantities  $\langle\cdot,\cdot\rangle_{\rmu\rmp}$ and $\invd$
introduced in Definition \ref{de:TwoGeodesics}.  Here the upper right Dini
derivative is defined via
\[ 
\frac{\rmd^+}{\rmd t } \zeta(t):= 
\limsup_{h\downarrow 0} \frac1h(\zeta(t{+}h)-\zeta(t)).
\] 
\begin{GSproof}
\COLORW
We will exploit Propositions \ref{eq:FirVarInnProd} and \ref{pr:LAC.Estim}. 
\COLORWend
\end{GSproof}
\begin{GSreplace}
For the proof of the following result we again refer to \cite{Sava11?GFDS,
  MurSavII+III}. 
\end{GSreplace}

\begin{proposition}[Differentiation of squared distance along geodesics]
\label{pr:EstimDerDist}
If\/ $(\calX,\sfd_{\calX})$  is a geodesic space and $x,y\in \calX$ such
that $F=\frac12 \sfd^2_{\calX}$  is $\kappa$-concave along
$\geo{x y}\in \Geod(x,y)$ with respect to an observer $\ob $,
then we have
\begin{equation}\label{akas}
\frac{\rmd^{+}}{\rmd t} \frac{1}{2}\sfd_{\calX}^{2}(\geo{x y}(t),\ob )\leq
-\langle\geo{x y},\geo{x \ob}\rangle_{\rmu\rmp}
+\kappa\sfd_{\calX}(x,y)\sfd_{\calX}(\geo{x y}(t),x)  \text{ a.e.\ on }[0,1].
\end{equation}
If furthermore $(\calX,\sfd_{\calX})$ satisfies LAC at $x,$ then for a.a.\ $t\in [0,1]$ we have 
\begin{equation}
  \label{eq:DerivEstim}
    \frac{\rmd^{+}}{\rmd t}\frac{1}{2}\sfd_{\calX}^{2}(\geo{x y}(t),\ob )
 \leq\langle\geo{x z},\geo{x \ob}\rangle_{\rmu\rmp} +
    \sfd_{\calX}(x,\ob )\invd(\geo{x y},\geo{x z})
    + t\!\; \kappa\,\sfd_{\calX}^{2}(x,y), 
\end{equation}
for every $z\in\calX$ and $\geo{x z} \in  \Geod(x,z).$
\end{proposition}
\begin{GSproof}
\COLORW
\begin{proof}
  For $s\in(0,1)$ we set $x_{s}=\geo{x y}(s)$ and observe that the curves
  $\xx_{0,s}(t)=\geo{x y}(ts)$ and $\xx_{s,0}(t)=\geo{x y}((1{-}t)s)$ belong to
  $\Geod(x,x_{s}),$ and $\Geod(x_{s},x)$ respectively.  Using the
  $\kappa$-concavity of $t \mapsto \frac12\sfd_{\calX}^2(\geo{x y}(t),\ob)$, a simply
  scaling argument shows that $t \mapsto \frac12 \,
  \sfd_{\calX}^2(\geo{x y}(st),\ob)$ is $s^2\kappa$-concave. By  applying
  \eqref{eq:asasa} for  $\xx_{0,s}$, we have
\begin{equation}
  \label{eq:co}
  \sfd_{\calX}^{2}(\xx_{0,s}(t),\ob )+2\kappa t(1{-}t)\sfd^{2} 
  \left(x ,x_{s}\right) \geq  (1{-}t) \sfd_{\calX}^{2}(x,\ob )
  +	t\,\sfd_{\calX}^{2}(x_{s},\ob ).
\end{equation}
from which we obtain
\begin{equation*}
  \sfd_{\calX}^{2}(x_{s},\ob )-\sfd_{\calX}^{2}(x,\ob )\leq
  2\kappa(1{-}t)\sfd_{\calX}^{2}\left(x ,x_{s}\right)+ \frac{
    \sfd_{\calX}^{2}(\xx_{0,s}(t),\ob ) - \sfd_{\calX}^{2}(x,\ob ) }{t}.  
\end{equation*}
By sending $t$ to $1$, we obtain 
\begin{equation}
  \label{xa8hka}
  \begin{split}
    \sfd_{\calX}^{2}(x_{s},\ob )-\sfd_{\calX}^{2}(x,\ob )&\leq
    2\kappa\sfd_{\calX}^{2}\left(x ,x_{s}\right) + s\limsup_{t\downarrow 0}
    \frac{\sfd_{\calX}^{2}(\geo{x y}(ts),\ob ) -\sfd_{\calX}^{2}(x,\ob
      )}{ts}\\& \leq 2\kappa\sfd_{\calX}^{2}\left(x ,x_{s}\right) -
    2s\langle\geo{x y},\geo{x \ob}\rangle_{\rmu\rmp},
  \end{split}
\end{equation}
where we used the estimate \eqref{eq:FirVarInnProd} from Proposition
\ref{pr:FirVarInnProd}.  
	
In analogy to \eqref{eq:co}, by replacing $\xx_{0,s}$ through $\xx_{s,0}$ we have   
\begin{equation}\label{eq:co2}
	\sfd_{\calX}^{2}(\xx_{s,0}(t),\ob )+2\kappa t(1{-}t)\sfd_{\calX}^{2}\left(x
          ,x_{s}\right) \geq  t\,\sfd_{\calX}^{2}(x,\ob ) +
        (1{-}t)\sfd_{\calX}^{2}(x_{s},\ob ).
\end{equation}
 Rearranging the terms and dividing by $ts>0$ we find 
\[
\frac	{\sfd_{\calX}^{2}(x_{s},\ob
          )-\sfd_{\calX}^{2}(\geo{xy}((1{-}t)s),\ob )}{ts}\leq
        \frac{2\kappa(1{-}t)\sfd_{\calX}^{2}\left(x ,x_{s}\right)+
            \sfd_{\calX}^{2}(x_{s},\ob ) -\sfd_{\calX}^{2}(x,\ob )
          }{s}\,. 
\]
 Inserting \eqref{xa8hka} on the right-hand side and taking the limit
$t\downarrow 0$ yields and estimate for the left upper Dini derivative  
\begin{align*}
\frac{\rmd^-}{\rmd s} \frac12\sfd_{\calX}^2(\geo{x y}(s),\ob)&= 
 \limsup_{t\downarrow 0}	\frac	{\sfd_{\calX}^{2}(\geo{x y}(s),\ob) - 
  \sfd_{\calX}^{2}(\geo{xy}(s{-}ts),\ob )}{ts} \\
&\leq 
  \frac{2\kappa\sfd_{\calX}^{2}(x,\geo{x y}(s))
-2s\langle\geo{x y},\geo{x \ob}\rangle_{\rmu\rmp} }{2s} 
\\
&= \kappa\,  \sfd_{\calX}(x,\xx_1)\,\sfd_{\calX}(x,\geo{x y}(s)) 
- \langle\geo{x y},\geo{x \ob}\rangle_{\rmu\rmp}.
\end{align*}
Since $s\mapsto \sfd_{\calX}^{2}(\geo{x y}(s),\ob )$ is concave, it is
also  Lipschitz. Hence, we have
$$ \frac{\rmd^{-}}{\rmd t} \sfd_{\calX}^{2}(\geo{x y}(t),\ob )=
\frac{\rmd^{+}}{\rmd t}\sfd_{\calX}^{2}(\geo{x y}(t),\ob ) \hspace{8pt}
\text{a.e. in}\hspace{8pt}
 (0,1),$$ which is the desired estimate \eqref{akas}.

Finally, \eqref{eq:DerivEstim} follows by inserting \eqref{eq:lacresult}
from Proposition \ref{pr:LAC.Estim} into \eqref{akas} and using the trivial
relation $\sfd_{\calX}(\geo{x y}(t),x) = t \sfd_{\calX}(x,y)$. 
\end{proof}
\COLORWend
\end{GSproof}

For an illustration of the estimates in Proposition \ref{pr:EstimDerDist} we
refer to the simple case discussed in Example \ref{ex:3/4disc}.

\subsection{EVI$_\lambda$ and construction of solutions}
\label{su:EVI.Sols}

We first recall the standard  definitions and results from
\cite[Sec.\,3]{MurSav20GFEV}  and then introduce our notations.

We have the following theorem that provides an alternative form of
\eqref{EVI-diff} that uses integration instead of differentiation. This form
can, in a straightforward manner, be combined with the lower semicontinuity
properties of the distance  $\sfd_{\calX}$ and of $\phi$, thus allowing to
show that various limits of EVI$_\lambda$ solutions are again
EVI$_\lambda$ solutions. 

The two results given in the next theorem are taken from
\cite[Thm.\,3.3+3.5]{MurSav20GFEV}.

\begin{proposition}[Characterizations and properties of EVI solutions]
\label{pr:EVI-int} \mbox{ }\\
\indent (A)  A curve $\xx:[0,T) \to \calX$ satisfies
  EVI$_\lambda$ with respect to $\phi$, if and only if  for all
  $\ob \in \mathrm{dom}(\phi)$ the two maps
  $t\mapsto \phi(\xx(t))$ and $t \mapsto \sfd^{2}(\xx(t),\ob )$ belong
  to $\rmL^1_\mathrm{loc}((0,T))$ and 
\begin{align}
\nonumber
\frac{1}{2}\sfd_{\calX}^{2}(\xx(t),\ob )-\frac{1}{2}\sfd_{\calX}^{2}(\xx(s),\ob )
 +\int_{s}^{t}\!\!\Big(\phi(\xx(r)){+}\frac{\lambda}{2}\sfd_{\calX}^{2}(\xx(r),\ob
   ) \Big)  \,\rmd r &\leq (t{-}s)\phi(\ob )
\\
\label{eq:EVI.inte}
 \text{for all } s,t\in (0,T) &\text{ with }s<t. 
\end{align}

 (B) If $\xx:[0,T) \to \calX$ is an EVI$_\lambda$ solution, then $t\mapsto
\phi(\xx(t))$ is non-increasing and hence continuous from the right (by lsc of
$\phi$).
 
(C) If $\xx^1,\xx^2:[0,\infty) \to \calX$ are EVI$_\lambda$ 
solutions, then we have 
\begin{equation}
\label{eq:expanding}
 \sfd_{\calX}(\xx^{1}(t),\xx^{2}(t))\leq
\ee^{-\lambda(t-s)}\sfd_{\calX}(\xx^{1}(s),\xx^{2}(s))  
\quad  \text{for all } s,t\in [0,T) \text{ with } s < t.
\end{equation}
\end{proposition}

We are now able to formulate our main existence result for EVI$_\lambda$
solutions for metric gradients systems $(\calX,\sfd_{\calX},\phi)$ which relies on
the geodesic structure of $(\calX,\sfd_{\calX})$, the $\lambda$-convexity of the
potential $\phi$, and some ``local $\kappa$-concavity'' of
$\frac12\sfd_{\calX}^2(t,\ob)$. The construction of solutions will
be done by the minimizing movement scheme and geodesic interpolation. 
For a given $x\in \calX$ and a time step $\tau$ the discrete solutions
$(x^\tau_n)_{n=0,1,...,N}$ of the minimization movement schemes are defined via 
\begin{equation}
  \label{eq:MiniMove}
   x^\tau_0=x \quad \text{and} \quad x^\tau_n \in
 \arg\min\Bigset{\frac1{2\tau}\sfd_{\calX}^2(x,x^\tau_{n-1}) + \phi(x) }{x
   \in \calX} \text{ for } n \in \N.
\end{equation}
As all our $\phi$ are $\lambda$-convex for some $\lambda\in \R$, the functional
$x\mapsto \frac1{2\tau}\sfd_{\calX}^2(x,x^\tau_{n-1}) + \phi(x)$ is quadratically
bounded from below for all $\tau>0$ with $\lambda +1/\tau>0$. Thus under
suitable assumptions on $\phi$ minimizers exist for sufficiently small $\tau$.

\begin{theorem}[Existence of EVI$_\lambda$ solutions]
\label{th:MainExistTheo} 
Let $(\calX,\sfd_{\calX})$ be a geodesic metric space and
$\phi:\calX \to (-\infty,\infty]$ a strongly $\lambda$-convex functional with
compact sublevels. Further assume that there exists a 
nested sequence of sets $\A_{\kappa}\subset\dom(|\partial\phi|)$ with
$\overline{\cup \A_{\kappa}}=\overline{\dom(\phi)},$ for which the statements (A1)
to (A3) are true:
\begin{enumerate}

\item[(A1)] For every $\kappa\in\mathbb{N},$ we have that
  $\frac12\sfd_{\calX}^{2}:\calX\times\calX \to \R ^{+}$ is strongly $\kappa$-concave in $\A_{\kappa}$
  with observers in $\GeoCov{\A}_{\kappa},$ and satisfies LAC for every
  point in $\A_{\kappa}.$

\item[(A2)] For all $x_0 \in \cup_{\kappa \in \N} \A_\kappa$ there exists
  $\kappa_0$ such that for all $\kappa >\kappa_0$ there exists
  $T(x_0,\kappa)>0$ such that for all time steps $\tau>0$ the $n$-step
  minimization scheme $(x^\tau_k)_{k=0,...,n}$ remains in $\A_\kappa$ as long
  as $n < T/\tau+1$. 

\item[(A3)] For every $x\in \overline{\dom(\phi)},$ there exists a sequence
  $x_{m}\in \cup \A_{\kappa}$ with $x_{m} \to x$ and
  $\phi(x_{m}) \to {\phi(x)}.$
\end{enumerate}
Then, for every $x_0\in \cup \A_{\kappa}$ there exists a unique EVI$_\lambda$
solution $\xx:[0,T_\infty(x_0) ) \to \calX$ with $\xx(0)=x_0$, 
where $T_\infty (x_0)= \lim_{\kappa \to \infty} T(x_0,\kappa)$ (w.l.o.g.\ 
$\kappa \mapsto T(x_0,\kappa)$ is non-decreasing). 

If additionally, assumption (A2) holds for all $x_0\in \cup \A_{\kappa}$ with 
$T_\infty(x_0)=\infty$, then for
all $\widehat x_0 \in \overline{\dom(\phi)}$ there exists a complete 
EVI$_\lambda$ solution $\xx$ with $\xx(0)=\widehat x_0$. 
\end{theorem}

We note that our assumption $x_0 \in \cup \A_{\kappa}$ trivially implies $
x_0\in \mafo{dom}(\pl\phi)$, i.e.\ $|\partial \phi|(x_{0}) < \infty$. For
$\lambda$-convex functionals on always has
$\ol{\mafo{dom}(\pl\phi)}=\ol{\mafo{dom}(\phi)} $, hence general initial
conditions in $\ol{\mafo{dom}(\phi)} $ can be approximated. 

The proof will be completed in Section \ref{su:ProofMainExist} after the
necessary a priori estimates for the discrete minimizing movement solutions are
collected next.

\subsection{Examples and counterexamples for LAC and semiconcavity}
\label{su:Exa.LAC.k-concave}

To help the intuition about the two important geometric conditions for
$(\calX,\sfd)$ we provide some examples and counterexamples for LAC and for
semiconcavity of $\frac12 \sfd(\ob,\cdot)^2$. 

Clearly, it can be easily verified that a Hilbert space satisfies the LAC and
the squared norm is $\kappa$-concavity with $\kappa=1$ (as well as
$\lambda$-convexity with $\lambda=1$). The properties are still valid on convex
subsets of a Hilbert space. More generally, on any smooth Riemannian manifold
LAC is satisfied and $\kappa$-concavity holds on compact sets.

\begin{example}[LAC and semiconcavity not satisfied]\slshape 
\label{ex:LAC.SC.Cross}
We consider the ``cross'' 
\[
M = \bigset{x=(x_1,x_2)\in \R}{ x_1=x_2=0} = \big( \R{\ti}\{0\}\big) \cup \big(
\{0\} {\ti} \R\big) 
\]
together with the distance 
\[
\sfd (x,y) = |x_1{-}y_1| + | x_2{-}y_2| = \|x{-}y\|_1,
\]
which is also the induced metric obtained by restricting the Euclidean distance
in $\R^2$. 

We may consider the following three geodesics starting from $x_*=0$: 
\[
\gamma^{(1)}(s) =(s,0) , \quad \gamma^{(2)}(s)=(-s,0), \quad \gamma^{(3)}(s) =
(0,-s).
\]
Since the (unique) geodesic connecting two points from two different geodesics
always has to pass through $x_*=0$ we easily see that all three angles
$\sphericalangle(\gamma^{(i)}, \gamma^{(j)})$ for $i\neq j$ are equal to $\pi$,
  which violates LAC. 

Moreover, the unique arclength-parametrized  geodesic connecting $y_*=(-1,0)$ and
$z_*=(0,1)$ is given by 
\begin{equation}
  \label{eq:Geod.y*.z*}
\geo{y_*z_*}(s) = \gamma^{(1)}(1{-}s) \text{ for } s\in [0,1] \quad \text{and} 
\quad \geo{y_*z_*}(s) = \gamma^{(3)}(s{-}1) \text{ for } s\in [1,2].
\end{equation}
Hence, choosing the observer $\ob=(1,0)$ we obtain $\sfd(\ob,\geo{y_*z_*}(s))^2
= \max\{(2{-}s)^2,s^2\}$, which is not semiconcave.
\end{example}

The next example shows that LAC and semiconcavity are not to be expected in
general Banach spaces, unless the squared norm is semiconcave. 

\begin{example}[LAC and semiconcavity in Banach spaces]\slshape
\label{ex:LAC.SC.Banach}
LAC is satisfied in a Banach space if and only if it is a Hilbert space, see
\cite[p.\,153]{Sava07GFDS}.  

We consider the Banach space $\R^2$ with the norm
$\|x\|_p=\big(|x_1|^p{+}|x_2|^p\big)^{1/p}$. 

For $p=1$ we may consider the same geodesics $\gamma^{(i)}$ as in the previous
example and obtain   $\sphericalangle(\gamma^{(i)}, \gamma^{(j)})=\pi$ for
$i\neq j$ are equal to $\pi$, which violates LAC. For the points
$\ob,\,y_*,z_*$ as in the previous example $\geo{y_*z_*}$ given in
\eqref{eq:Geod.y*.z*} is still one of the many geodesics in $\Geod(y_*,z_*)$,
and we have $\|\ob{-}\geo{y_*z_*}(s)\|_1^2= \min\{(2{-}s)^2, 1+(s{-}1)^2\}$
which is not semiconcave.  

The case $p=\infty$ is analogous to the case $p=1$. For $p \in {]1,2[}$ it can
be shown that local angles are not defined in the strict sense and that
semiconcavity holds. The case $p=2$ is the Hilbertian case with LAC and
$1$-concavity. 

For $p\in {]2,\infty[}$ one still has semiconcavity because
$\mafo{sn}(x)=\frac12\|x\|_p^2$ is lies in $\rmC^1(\R^2)$ with
$\mafo{Lip}(\nabla \mafo{sn})= c_p\leq p$, i.e.\ we have $p$-concavity. 

For general Banach spaces $(B,\|\cdot\|_B)$, one obtains semiconcavity if the
$\mafo{sn}(x) = \frac12\|x\|_B^2$ is differentiable with a globally
Lipschitz-continuous derivative. 
\end{example}

The next example shows a case where semiconcavity and LAC are present,
but vanish if certain parameter of the space are taken to a nontrivial limit.

\begin{example}[A smoothed three-quarters disc]\slshape
\label{ex:3/4disc} For $0\leq r<R\leq \infty$, we consider the nonconvex
two-dimensional domain  
\[
X_{r,R}:=\bigset{x\in \R^2}{ |x|\leq R \text{ and }\big(x_1\geq 0\text{ or }
  x_2\geq 0\big) } \cup \bigset{x\in [-r,0]^2}{ |x{+}(r,r)|\geq r} ,
\]
where $|\cdot|$ denotes the Euclidean distance, see also Figure
\ref{fig:3/4disc}. The distance $\sfd$ is given by the length of the shorted
curve inside of $X_{r,R}$. It is easy to see that all geodesics are unique and
are either straight lines or they touch the circle of radius $r$ around
$(-r,-r)$. Obviously, the geodesics connecting $x_*=(-R,0)$ and
$y_*=(R,0)$ is the straight line $\gamma^{(1)}(s)=(R{+}s,0)$ for $s\in [0,2R]$,
and the geodesics $\gamma^{(2)}= \geo{x_*z_*}$ connecting $x_*$ and
$z_*=(0,-R)$ coincides with $\gamma^{(1)}$ for $s\in [0,R{-}r]$, then branches
of tangentially to stay on the circle for $s\in [R{-}r,R+(\frac\pi2{-}1)r]$ and
then moves down vertically for $s\in [R+(\frac\pi2{-}1)r,2R+(\frac\pi2{-}2)r]$.  

For $r>0$ the LAC is always satisfied, and it can be shown that the
$\frac12\sfd^2(\ob,\cdot)$ is $\kappa$-concave with $\kappa =\sqrt2 +
R/r$. The most extreme case is achieved by observing $\gamma^{(2)}$ from
$\ob=(R/\sqrt2,R/\sqrt2)$. 

In the limit $r=0$ the LAC is lost, which follows as in Example
\ref{ex:LAC.SC.Cross}. For $r>0$ and $R=\infty$ we still have the LAC, while
semiconcavity is no longer true globally.  However, $X_{r,\infty}$ can be
written as the union of compact sets spaces on which semiconcavity holds with
an increasing $\kappa$. 
\end{example} 
\begin{figure}[ht]\footnotesize
\hspace*{5em}\begin{tikzpicture}
\draw[color=green!20, fill=green!20] (0,0) circle (2);
\draw[color=white,fill=white] (-2,-2) rectangle (0,0);
\draw[color=green!20, fill=green!20] (-0.5,-0.5) rectangle (0,0);
\draw[color=red!10,fill=red!10] (-0.5,-0.5) circle (0.5);
\draw[color=green, very thick] (0,-2) arc (-90:180:2);
\draw[color=green, very thick] (0,-2)-- (0,-0.5);
\draw[color=green, very thick] (-2,0)-- (-0.5,0);
\draw[color=green, very thick] (0,-0.5) arc (0:90:0.5);
\draw[color=blue,thick] (-2,0.06)-- node[pos=0.7, above]{$\gamma^{(1)}$} (2,0.06) ;
\draw[color=red,thick] (0.03,-2)-- node[pos=0.5, right]{$\gamma^{(2)}$} (0.03,-0.5);
\draw[color=red,thick] (-2,0.03)-- (-0.5,0.03);
\draw[color=brown,thick] (0.3,1.8)--node[pos=0.5,right]{$\gamma^{(3)}$}(1.5,0.3);

\draw[color=black,dashed,thin, ->] (0,0.06)-- node[pos=0.6,right]{$R$} (0,2);
\draw[color=black,thin, ->] (-0.5,-0.5)-- node[pos=0.5,below]{$r$} (0,-0.5);
\draw[color=red,thick] (0.03,-0.53) arc (0:90:0.56);
\node[color=black] at (-1,1) {$X_{r,R}$}; 
\node[left, color=black] at (-2,0) {$x_*$}; 
\node[right, color=black] at (2,0) {$y_*$}; 
\node[left, color=black] at (0,-2) {$z_*$}; 
\end{tikzpicture}   \quad
\begin{minipage}[b]{0.53\linewidth}\caption{The green area denotes the 
    metric space $X_{r,R}$. Geodesic curves are either straight lines like
    $\gamma^{(1)}$ or $\gamma^{(3)}$ or they touch the (light red) circle of
    radius $r$ around the center $(-r,-r)$ like $\gamma^{(2)}$.
    Considering $\gamma^{(1)}$ and $\gamma^{(2)}$ we see that geodesics can
    branch.\\[1em] \mbox{} }
\end{minipage}
\label{fig:3/4disc}
\end{figure}
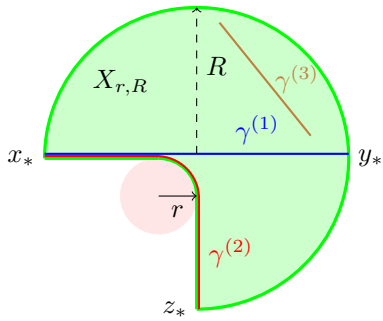

From the above examples one can see that the LAC and the semiconcavity of
$\frac12\sfd^2(\ob,\cdot) $ serve the same purpose, namely controlling the
divergence of initially identical or inially close geodesics for larger values. 
Consider two geodesics $\gamma^{(j)}:[0,1]\to \calX$, $j=1,2$, that coincide for $s\in
[0,s_*]$ with $s_*>0$ and differ for $s>s_*$. Splitting the two geodesics at
the common point $x=\gamma^{(1)}(s_*)=\gamma^{(1)}(s_*)$, we obtain three
geodesics starting at $x_*$ and ending in
$y_0=\gamma^{(1)}(0)=\gamma^{(1)}(0)$, $y_1=\gamma^{(1)}(1)$, and
$y_2=\gamma^{(2)}(1)$, respectively. 
Since $\gamma^{(1)}$ and $\gamma^{(2)}$ are geodesics,
we easily see 
\[
\sphericalangle(\geo{xy_0},\geo{xy_1})= \pi\quad \text{and} \quad 
\sphericalangle(\geo{xy_0},\geo{xy_2})= \pi.
\]
Thus, the LAC implies that $\sphericalangle(\geo{xy_1},\geo{xy_2})=0$, which
means that the splitting of $\gamma^{(1)}$ and $\gamma^{(2)}$ has to be
tangential. See Figure \ref{fig:3/4disc} in Example \ref{ex:3/4disc}, where
also the role of $\kappa>0$ in the
$\kappa$-concavity of $\frac12\sfd^2(\ob,\cdot) $ can be seen. In this example,
$\kappa$ is proportional to $1/r$, which is the curvature of the arc contained
in the geodesic $\gamma^{(2)}$. The exact statement is given in Proposition
\ref{pr:EstimDerDist}. 

While the above simple examples provide some intuition, it is important to note
that the 2-Wasserstein space $(\calP(X),\mathsf W_2)$ is globally $1$-concave,
see Thm.\,7.3.2+Prop.\,9.3.12 in \cite{AmGiSa05GFMS}. The results needed for the
theory developed here (see Theorem \ref{mainmain}) for the spaces
$(\calM(X),\HK)$ and $(\calP(X),\SHK)$ are provided in
\cite[Prop.\,4.1]{LasMie19GPCA} for LAC and in \cite[Sec.\,4.2]{LasMie19GPCA}
for semiconcavity. The latter results are local in the sense that one needs to
restrict to suitable subsets with lower and upper density bounds.  \EEE

\subsection{Estimates for the MM scheme}
\label{su:MinimMovem}

\begin{GSreplace}
  In this subsection we state a few of the results from \cite{Sava11?GFDS,
    MurSavII+III} quite explicitly, especially to emphasize the dependence on
  the semiconvexity parameter $\lambda$ of $\phi$ and the semiconcavity
  parameter $\kappa$ of $\frac12 \sfd^2$. This concerns our Lemma
  \ref{le:EulerEquMM}, Proposition \ref{pr:ErrorProduc}, Corollary
  \ref{co:error}, and Lemma \ref{le:error.estimate}. 
\end{GSreplace}

We first state some very basic estimates that hold true for the discrete
solutions of the MM scheme.  Using geodesic interpolation,
$\lambda$-convexity of $\phi$ and $\kappa$-concavity of the squared distance,
one then obtains sharper estimates that allows us to control the distance between
different approximants. 
\begin{GSproof}
\COLORW 
All the results in this subsection are taken from \cite{Sava11?GFDS,
  MurSavII+III}. 
\COLORWend
\end{GSproof}

These estimates will be used later prove the existence of
curves that satisfy EVI$_\lambda$, but also to bound the distance of the
EVI$_\lambda$ satisfying curve from the approximating curves occurring by
geodesically interpolating the points of the minimizing scheme.
Before we proceed we are going to define the metric slope $|\partial\phi|$ of
$\phi$ at a point $x\in \calX$ via 
\begin{equation}
    |\partial\phi|(x)=\begin{cases}+\infty& \text{for } x\notin \dom(\phi),
\\0& \text{if } x\in \dom(\phi) \text{ is isolated},\\ \limsup_{y\rightarrow
  x}\frac{\max\{0,\phi(x)-\phi(y)\}}{d(x,y)}& \text{otherwise}. 
    \end{cases}
\end{equation}

\begin{lemma}[Euler equation for discrete solutions]
\label{le:EulerEquMM}
For a given $x_{n-1}^{\tau}\in\calX$ and $\tau>0,$ assume that
$x_{n}^{\tau}\in\calX$ is a solution to the minimizing movement scheme
\eqref{eq:MiniMove}.  Then, for every observer $\ob \in\calX$ and all geodesics
$\geo{x^{\tau}_{n} x^{\tau}_{n-1}} \in \Geod(x^{\tau}_{n},x^{\tau}_{n-1})$ and
$\geo{x^{\tau}_{n} \ob} \in \Geod(x^{\tau}_{n},\ob ),$ we have
\begin{equation}
  \label{eq:EulerEstim1}
  \frac{1}{\tau}\langle \geo{x^{\tau}_{n} x^{\tau}_{n-1}},\geo{x^{\tau}_{n}
    \ob}\rangle_{\rmu\rmp}  
  +\frac{\lambda}{2}\sfd_{\calX}^{2}(x_{n}^{\tau},\ob )+\phi(x_{n}^{\tau})\leq \phi(\ob ),
\end{equation}
\begin{GSreplace}
\begin{equation}\label{Euler}
  \begin{aligned}\text{\emph{(a) }}\ |\partial\phi(x_{n}^{\tau})|\leq 
           \frac{\sfd_{\calX}(x_{n-1}^{\tau},x_{n}^{\tau})}{\tau},\quad
  &\text{\emph{(b)} }\ (1{+}\lambda\tau)\frac{\sfd_{\calX}(x_{n-1}^{\tau},x_{n}^{\tau})}{\tau} 
          \leq|\partial \phi|(x_{n-1}^{\tau}),\\
  &\text{\emph{(c)} }\ (1{+}\lambda\tau)|\partial \phi|(x_{n}^{\tau}) 
            \leq|\partial\phi|(x_{n-1}^{\tau}).
  \end{aligned}
\end{equation}
\end{GSreplace}
\begin{GSproof}
\COLORW
\begin{subequations}\label{Euler2}%
\begin{align}	
& \label{Euler2a}
|\partial\phi(x_{n}^{\tau})|\leq\frac{\sfd_{\calX}(x_{n-1}^{\tau},x_{n}^{\tau})}{\tau},
\\
& \label{Euler2b}
(1{+}\lambda\tau)\frac{\sfd_{\calX}(x_{n-1}^{\tau},x_{n}^{\tau})}{\tau}\leq|\partial
\phi|(x_{n-1}^{\tau}),
\\
&\label{Euler2c} 
(1{+}\lambda\tau)|\partial \phi|(x_{n}^{\tau})\leq|\partial\phi|(x_{n-1}^{\tau}). 
\end{align}
\end{subequations}
\COLORWend
\end{GSproof}
\end{lemma}
\begin{GSproof}
\COLORW
\begin{proof}
\ul{Step 1: \eqref{eq:EulerEstim1}}  We define the map
  $$Y:[0,1] \to (-\infty,+\infty],\ s\mapsto
  \frac{1}{2\tau}\sfd_{\calX}^{2}(x^{\tau}_{n-1},\geo{x^{\tau}_{n}  \ob}(s)) +
  \phi(\geo{x^{\tau}_{n} \ob}(s)).$$  Since $x^{\tau}_{n}=\geo{x^{\tau}_{n}\ob}(0)$ is a global
  minimizer, the function $Y$ has a minimum at $s=0$. Therefore, 
\begin{align*}
0&\leq\liminf_{s\downarrow 0}\frac{Y(s)-Y(0)}{s}
 \overset{\text{Prop.\,\ref{pr:FirVarInnProd}}}{\leq} 
  -\frac{1}{\tau}\langle\geo{x^{\tau}_{n} x^{\tau}_{n-1}},\geo{x^{\tau}_{n} \ob}\rangle_{\rmu\rmp} 
+\liminf_{s\downarrow 0}\frac{\phi(\geo{x_{n} \ob}(s))-\phi(x^{\tau}_{n})}{s}
\\
&
\hspace{-0.9em}\overset{\lambda\text{-convex}}{\leq}
 -\frac{1}{\tau}\langle\geo{x^{\tau}_{n} x^{\tau}_{n-1}},\geo{x^{\tau}_{n} \ob}\rangle_{\rmu\rmp} 
 +\phi(\ob )-\phi(x^{\tau}_{n}) 
 -\liminf_{s\downarrow 0}\frac{\lambda}{2}\,
 \frac{s(1{-}s)\sfd_{\calX}^{2}(x^{\tau}_{n},\ob )}{s}
\\&=-\frac{1}{\tau}\langle\geo{x^{\tau}_{n} x^{\tau}_{n-1}},\geo{x^{\tau}_{n} \ob} 
\rangle_{\rmu\rmp}+\phi(\ob )-\phi(x^{\tau}_{n})-\frac{\lambda}{2} 
  \sfd_{\calX}^{2}(x^{\tau}_{n},\ob ).
\end{align*}
This is the desired estimate \eqref{eq:EulerEstim1}.
 
\ul{Step 2: \eqref{Euler2a}} Since $x_{n}^{\tau}$ is a minimizer of the movement scheme \eqref{eq:MiniMove}, we get 
\begin{equation}
        \phi(x_{n}^{\tau})+\frac{\sfd_{\calX}^{2}(x_{n-1}^{\tau},x_{n}^{\tau})}{2\tau}\leq \phi(y)+\frac{\sfd_{\calX}^{2}(x_{n-1}^{\tau},y)}{2\tau}.
\end{equation}
From which, we get
\begin{equation}
\begin{split}
 \phi(y) &\geq \phi(x_{n}^{\tau})+\frac{1}{2\tau}\left(\sfd_{\calX}^{2}(x_{n-1}^{\tau},x_{n}^{\tau})-\sfd_{\calX}^{2}(x_{n-1}^{\tau},y)\right)\\
 &\geq \phi(x_{n}^{\tau})-\sfd_{\calX}(y,x_{n}^{\tau})\frac{1}{2\tau}\left(\sfd_{\calX}(x_{n-1}^{\tau},x_{n}^{\tau})+\sfd_{\calX}(x_{n-1}^{\tau},y)\right)\\
 &= \phi(x_{n}^{\tau})-\frac{1}{\tau}\sfd_{\calX}(y,x_{n}^{\tau})\sfd_{\calX}(x_{n-1}^{\tau},x_{n}^{\tau})-\sfd_{\calX}(y,x_{n}^{\tau})\frac{1}{2\tau}\left(\sfd_{\calX}(x_{n-1}^{\tau},y)-\sfd_{\calX}(x_{n-1}^{\tau},x_{n}^{\tau})\right)\\
  &\geq  \phi(x_{n}^{\tau})-\frac{1}{\tau}\sfd_{\calX}(y,x_{n}^{\tau})\sfd_{\calX}(x_{n-1}^{\tau},x_{n}^{\tau})-\frac{1}{2\tau}\sfd_{\calX}(y,x_{n}^{\tau})
\end{split}
\end{equation}
Now dividing by $\sfd_{\calX}(y,x_{n}^{\tau})$ and taking the limit we get the result.\newline

\ul{Step 3: \eqref{Euler2b} and \eqref{Euler2c}} In \eqref{Euler2} we choose $\ob=x^{\tau}_{n-1}$  
\begin{equation}
\frac{1+\frac{1}{2}\lambda\tau}{\tau}\sfd_{\calX}(x_{n-1}^{\tau},x_{n}^{\tau})\leq \phi(x_{n-1}^{\tau})-\phi(x_{n}^{\tau}) \leq|\partial\phi(x_{n-1}^{\tau})|\sfd_{\calX}(x_{n-1}^{\tau},x_{n}^{\tau})
-\frac{\lambda}{2}\sfd_{\calX}^{2}(x_{n-1}^{\tau},x_{n}^{\tau})\end{equation}
By subtracting $\frac{1}{2\tau}\sfd_{\calX}^{2}(x_{n-1}^{\tau},x_{n}^{\tau}),$ from all terms, we get what we want.
\end{proof}
\COLORWend
\end{GSproof}

We now provide an estimate of how much a geodesic interpolation of a
minimizing scheme deviates from being an EVI$_{\last}$ 
solution for the modified $\last=2\min\{0,\lambda\}-2 \leq -2$,
where $\lambda$ is such that $\phi$ is $\lambda$-convex. We note and highlight
that $\lambda$-convexity of $\phi$ with $\lambda >0$ will not be helpful
here. The following technical estimate  will be used later to prove that
interpolating curves $\xx^{\tau}$ converge to some curve
$\xx$ when the time steps $\tau$ converges to 0. One important feature of
the following result is that the error terms $\Er^\tau_n$ are independent of
the observer point $\ob$.

\begin{proposition}[Discrete error estimates]
\label{pr:ErrorProduc}
  Let $\lambda \leq 0$, $\tau>0$ with $1{+}\lambda\tau>0$, and 
  $x^\tau_0\in \A_{\kappa_{0}}$ be fixed 
  that the discrete solution $\{x^{\tau}_{n}\}_{n\in\mathbb{N}}$ of
  \eqref{eq:MiniMove} satisfies the following: 
   For $T>0,$ let $\kappa\in\mathbb{N}$ be such that $x^{\tau}_{n}\in\A_{\kappa},$
  for  all $n\in \N_0$ with  $n<\frac{T}{\tau}+1$. Then, all geodesic
  interpolants $\xx^{\tau}:[0,\infty) \to \calX$, given by
\begin{equation}
\label{geodesic.interpolation}
  \xx^{\tau}(t)=\geo{x^{\tau}_{n} x^{\tau}_{n+1}}((t{-}n\tau)/\tau) \text{ for } 
 t\in[n\tau,(n{+}1) \tau] \ \text{ with } \geo{x^{\tau}_{n} x^{\tau}_{n+1}}\in
 \Geod(x^{\tau}_{n},x^{\tau}_{n+1}) ,
\end{equation}
satisfies,  for all\/  $\ob \in \GeoCov{\A}_\kappa$ and  almost all
$t\in[0,T]$,  the estimate
\begin{equation}
\label{evi.for.interpol}
  \frac{\rmd}{\rmd t}\frac{1}{2}\sfd_{\calX}^{2}(\xx^{\tau}(t),\ob) 
  + \frac{\last}{2}\sfd_{\calX}^{2}(\xx^{\tau}(t),\ob )\leq \phi(\ob)
  - \phi(\xx^{\tau}(t))+ \Er^{\tau,\kappa}(t), 
\end{equation} 
with $\last=2\min\{0,\lambda\}-2$ and $\Er^{\tau}(t)=\Er^{\tau}_{n}$ for
$t\in[n\tau,(n{+}1)\tau)$, where
\begin{equation}
\label{error}
 \Er^{\tau,\kappa}_{n} =\begin{cases}
    (1{-}2\lambda)\sfd_{\calX}^{2}(x^{\tau}_{0},x^{\tau}_{1})+
   \left(1+(1{+}\lambda\tau)^{-1}\right)|\partial\phi|^{2}(x^\tau_0) 
              &\text{ for } n=0,\\
    \left(1{-}2\lambda{+}  \kappa/\tau \right) 
   \sfd^{2}(x^{\tau}_{n},x^{\tau}_{n+1})+\frac1{\tau^{2}}
    \,\invd^{2}\Big(\geo{x^{\tau}_{n} x^{\tau}_{n-1}}, 
  \geo{x^{\tau}_{n} x^{\tau}_{n+1}}\Big)&\text{ for } n\in\mathbb{N}.
  \end{cases}
\end{equation} 
\end{proposition}
\begin{GSproof}
\COLORW
\begin{proof} \underline{Step 1:} 
First we notice that by $\lambda$-convexity, we have
$(1{-}\theta)\phi(x^{\tau}_{n})+\theta \phi(x^{\tau}_{n+1}) \geq
\phi(\xx^{\tau}((n{+}\theta)\tau))+\frac{\lambda}{8}\sfd_{\calX}^{2}(x^{\tau}_{n},x^{\tau}_{n+1})$.
Moreover, from the minimizing scheme we obtain
$\phi(x^{\tau}_{n})\geq \phi(x^{\tau}_{n+1})$, which now implies
\begin{equation}
  \label{lambda}
  \phi(x^{\tau}_{n})  \geq  \phi(\xx^{\tau}(t))+ 
   \frac{\lambda}{8}\sfd_{\calX}^{2}(x^{\tau}_{n},x^{\tau}_{n+1}) 
       \quad \text{for } t \in [n\tau, (n{+}1)\tau].
\end{equation} 
Using the triangle inequality for $s, t\in (n\tau,(n{+}1)\tau)$ gives
\begin{equation}
\label{tr}
\big| \sfd_{\calX}(\xx^\tau(s),\ob) - \sfd_{\calX}(\xx^\tau(t),\ob) \big| \leq 
\sfd_{\calX}(\xx^\tau(s),\xx^\tau(t)) = \frac{|s{-}t|} \tau \,\sfd_{\calX} ( x^\tau_n,
  x^\tau_{n+1}). 
\end{equation}

\underline{Step 2:} In the first step the geodesic interpolant satisfies  
\begin{equation}
\label{derder}
\begin{split}
	\frac{\rmd}{\rmd t}\frac{1}{2}\sfd_{\calX}^{2}(\xx^{\tau}(t),\ob )& 
\overset{\text{\eqref{tr}}}\leq \frac{\sfd_{\calX}(x^{\tau}_{0},x^{\tau}_{1})}{\tau} \:
  \sfd_{\calX}(\xx^{\tau}(t),\ob ) 
\leq\frac{\sfd_{\calX}^{2}(x^{\tau}_{0},x^{\tau}_{1})}{2\tau^{2}}+\frac{1}{2} 
\sfd_{\calX}^{2}(\xx^{\tau}(t),\ob )
\\
&\leq\frac{1}{2}(1+\lambda\tau)^{-2}|\partial\phi|^{2}(x^{\tau}_{0})+
   \frac{1}{2}\sfd_{\calX}^{2}(\xx(t),\ob), 
\end{split}
\end{equation}
where in the last inequality we applied \eqref{Euler2b}. 
We also get
\begin{equation}\label{diedie}
  \begin{aligned}
    &\phi(\xx^{\tau}(t))-\phi(\ob )\leq\phi(x^{\tau}_{0})-\phi(\ob )+
    \frac{\lambda}{8}\sfd_{\calX}^{2}(x^{\tau}_{0},x^{\tau}_{1})& &\text{\small [use
      \eqref{lambda}]}
    \\
    &\leq |\partial\phi(x^{\tau}_{0})|\sfd_{\calX}(x^{\tau}_{0},\ob
    )+\frac{\lambda}{2}\sfd_{\calX}^{2}(x^{\tau}_{0},\ob
    )+\frac{\lambda}{8}\sfd_{\calX}^{2}(x^{\tau}_{0},x^{\tau}_{1}) &&\text{\small [use
      $\lambda$-convexity of $\phi$]}
    \\
    &\leq |\partial\phi(x^{\tau}_{0})|^{2}+\frac{\sfd^{2}(x^{\tau}_{0},\ob
      )}{4}+\frac{\lambda}{2}\sfd_{\calX}^{2}(x^{\tau}_{0},\ob
    )+\frac{\lambda}{8}\sfd_{\calX}^{2}(x^{\tau}_{0},x^{\tau}_{1}) && \text{\small [use
      $ab\leq a^{2}+b^{2}/4$]}
    \\
    &\leq\frac{|\partial\phi(x^{\tau}_{0})|^{2}}{2}+\left(\lambda+\frac{1}{2}\right)
    \sfd^{2}(\xx(t),\ob )
    +\left(\frac{1}{2}+\frac{\lambda}{8}\right)\sfd_{\calX}^{2}(x^{\tau}_{0},x^{\tau}_{1}) &&
    \text{\small [use \eqref{tr}]}
    \\
    &\leq\frac{|\partial\phi(x^{\tau}_{0})|^{2}}{2}-\left(\lambda-\frac{1}{2}\right)
    \sfd_{\calX}^{2}(\xx(t),\ob ) +(1-2\lambda)\sfd^{2}(x^{\tau}_{0},x^{\tau}_{1})
    &&\text{\small [because of $\lambda<0$]}.
  \end{aligned}
\end{equation}
Here the last step is not needed, but makes it consistent with the case
$n>0$. 
 
\underline{Step 3.} Now combining \eqref{derder} with the \eqref{diedie}, we have 
\begin{equation}
  \frac{\rmd}{\rmd t}\frac{1}{2}\sfd^{2}(\xx^{\tau}(t),\ob )\leq\phi(\ob
  )-\phi(\xx^{\tau}(t))-\frac{\last}{2} \sfd^2( \xx^{\tau}(t),\ob )
  +\Er_{0}^\tau,
\end{equation} 
where $\last:=2\min\{0,\lambda\}-2$ and
$ \Er_{0}^\tau = (1-2\lambda)\sfd^{2}(x^{\tau}_{0},x^{\tau}_{1})+ 
(1+(1+\lambda\tau)^{-1})
|\partial\phi|^{2}(x) $.  For $n\in\mathbb{N}$ and $t\in [n\tau,(n{+}1)\tau],$
by applying \eqref{eq:DerivEstim} to
$\xx^\tau(t) = \geo{x^\tau_n, x^\tau_{n+1}}(t/\tau{-}n)$ and $z=x^\tau_{n-1}$ we have
\begin{equation*}
  \begin{split}
    &\frac{\rmd}{\rmd t}\frac{1}{2}\sfd^{2}(\xx^{\tau}(t),\ob )
\\
& \leq
    \frac{\langle\geo{x^{\tau}_{n} x^{\tau}_{n-1}},\geo{x^{\tau}_{n}
        \ob}\rangle_{\rmu\rmp}}{\tau}+ \frac{\sfd_{\mathcal{X}}(x^{\tau}_{n},\ob
      )\invd(\geo{x^{\tau}_{n} x^{\tau}_{n-1}},\geo{x^{\tau}_{n}
        x^{\tau}_{n+1}})}{\tau}
    +\frac{\kappa (t/\tau-n) \sfd^{2}(x^{\tau}_{n},x^{\tau}_{n+1})}{\tau}
\\ 
&\leq\phi(\ob)-\phi(x_{n}^{\tau})-\frac{\lambda}{2}\sfd_{\mathcal{X}}^{2}(x_{n}^{\tau},\ob ) 
  +  \frac{\sfd_{\calX}(x^{\tau}_{n},\ob )\invd(\geo{x^{\tau}_{n}
        x^{\tau}_{n-1}},\geo{x^{\tau}_{n} x^{\tau}_{n+1}})}{\tau}
    +\frac{\kappa^+}{\tau} \sfd_{\mathcal{X}}^{2}(x^{\tau}_{n},x^{\tau}_{n+1})
\\
    &\hspace{7pt}\leq \phi(\ob
    )-\phi(x_{n}^{\tau})-\frac{\lambda-1}{2}\sfd_{\calX}^{2}(x_{n}^{\tau},\ob )+ \frac{
      \invd^{2}(\geo{x^{\tau}_{n} x^{\tau}_{n-1}},\geo{x^{\tau}_{n}
        x^{\tau}_{n+1}})}{2\tau^{2}}
    +\frac{\kappa^+}{\tau} \sfd_{\calX}^{2}(x^{\tau}_{n},x^{\tau}_{n+1}),
  \end{split}
\end{equation*}	
where $\kappa^+=\max\{0,\kappa\}$ and the middle estimate used
\eqref{eq:EulerEstim1}. Hence, we found 
\begin{align*}
  &\frac{\rmd}{\rmd t}\frac{1}{2}\sfd_{\calX}^{2}(\xx^{\tau}(t),\ob )\leq\phi(\ob
  )-\phi(\xx^{\tau}(t))-\frac{\last}{2}\,\sfd_{\calX}^2( 
 \xx^{\tau}(t),\ob ) +\Er^\tau_{n}. 
\\
&\Er^\tau_{n}=
\left(1-2\lambda+\frac{\max\{\kappa,0\}}{\tau}\right)\,
  \sfd_{\calX}^{2}(x^{\tau}_{n},x^{\tau}_{n+1})+\frac{\invd^{2}\left(\geo{x^{\tau}_{n}  
      x^{\tau}_{n-1}},\geo{x^{\tau}_{n} x^{\tau}_{n+1}}\right)}{\tau^{2}}. 
\end{align*}

\underline{Step 4.} Together with  \eqref{derder} and  \eqref{diedie} the
desired estimate \eqref{evi.for.interpol}. 
\end{proof}
\COLORWend
\end{GSproof}

The next result exploits the strength of the EVI formulation with an
arbitrary observer. If we have two approximate solutions $\xx_i$ to the EVI, where the
error term $\Er_i(t)$ does not depend on the observer, then we obtain a control
on the distances between $\xx_1$ and $\xx_2$. 

\begin{lemma}[Distance between approximate EVI solutions]
\label{curves observing each other}
Let $\Er_{i}$, $i=1,2$, be nonnegative real functions in $L^{1}([0,T])$. Let
also $\xx_{i}\in \mathrm{AC}_\mathrm{loc}([0,\infty)),\, i=1,2,$ be two locally absolutely
continuous functions satisfying EVI$_\lambda$ for some $\last<0$ in the form
\begin{equation}
\label{eq:EVI.approx}
  \frac{\rmd}{\rmd t}\frac{1}{2}\sfd_{\calX}^{2}(\xx_{i}(t),\ob
  )+\frac{\last}{2}\sfd_{\calX}^{2}(\xx_{i}(t),\ob )\leq \phi(\ob
  )-\phi(\xx_{i}(t))+\Er_{i}(t),  
\end{equation}  
for a.a.\ $t\in[0,T]$ and every $\ob \in \xx_{3-i}([0,T])$ for
$i=1,2$. Then we have the estimate
\begin{equation}
  \label{eq:Lipschitz}
  \sup_{t\in[0,T]} \mathrm e^{\last t}\sfd_{\calX}(\xx_{1}(t),\xx_{2}(t))\leq
  \sfd_{\calX}(\xx_{1}(0),\xx_{2}(0))+  \big\|2\mathrm e^{2\last t}(\Er_{1}{+}
    \Er_2 )\big\|_{ L^{1}[0,T]}^{1/2}\hspace{-13pt}. 
\end{equation}
\end{lemma}
\begin{proof} For $s,t\in [0,T]$ we define the function
  $q(s,t):=\frac{1}{2}\sfd_{\calX}(\xx_{1} (s), \xx_{2}(t))$. Applying
  \cite[Lem.\,4.3.4]{AmGiSa05GFMS} we obtain the differential
  inequality
\begin{equation}
  \begin{split}
    \frac{\rmd}{\rmd t}\frac{1}{2}\sfd_{\calX}^{2}(\xx_{1}(t),\xx_{2}(t))&\leq
    \limsup_{h\downarrow 0}\frac{q(t,t)-q(t{-}h,t)}{h} +
    \limsup_{h\downarrow 0}\frac{q(t,t)-q(t,t{+}h)}{h}\\
    & = \frac{\rmd}{\rmd
      s}\frac{1}{2}\sfd_{\calX}^{2}(\xx_{1}(s),\xx_{2}(t))|_{s=t} +
    \frac{\rmd}{\rmd r}\frac{1}{2}\sfd_{\calX}^{2}(\xx_{1}(t),\xx_{2}(r))|_{r=t}\\
    &\leq -\last\sfd_{\calX}^{2}(\xx_{1}(t),\xx_{2}(t)) +\Er_{1}(t){+}\Er_{2}(t),
  \end{split}
\end{equation}
for almost every $t>0$, where we used \eqref{eq:EVI.approx} for $\xx_i$
with observer $\ob = \xx_{3-i}$. Multiplying the inequality by 
$2 \mathrm e^{2\last t}$ and integrating in time yields the desired estimate.
\end{proof}

The following result specifies this estimate by looking at the solutions
obtained as geodesic interpolants from MM schemes with two different time steps
$\tau>0$ and $\sigma>0$. 

\begin{corollary}[Comparison of MM solutions]
\label{co:error}
Let $\tau,\sigma,$ two time steps and let $\xx^{\tau}$ and $\xx^{\sigma}$ be two
piecewise geodesic interpolants defined in \eqref{geodesic.interpolation} for
initial conditions  $x^{\tau}_{0}, x^{\sigma}_{0}\in\A_{\kappa_{0}}$,
respectively. Let $\phi$ be $\lambda$-convex with $\lambda\leq0$ and
set $\last=2\min\{0,\lambda\}-2$. Then we have 
\begin{equation}
  \sup_{t\in[0,T]}e^{\last t}\sfd_{\calX}(\xx^{\tau}(t),\xx^{\sigma}(t)) 
 \leq \sfd_{\calX}(x^{\tau}_0,x^{\sigma}_{0})+\big\|2 e^{2\last t}(\Er^{\tau}{+}
   \Er^{\sigma}) \big\|_{L^{1}[0,T]}^{1/2}, 
\end{equation}
where $\Er^{\tau}$ and $\Er^\sigma$ are defined via \eqref{error} in Proposition
\ref{pr:ErrorProduc}.
\end{corollary}

Thus, it remains to control the error functions $\Er^\tau$. The main
problem is to control the terms
$\frac1{\tau} \invd^{2}\left(\geo{x^{\tau}_{n} x^{\tau}_{n-1}},
\geo{x^{\tau}_{n} x^{\tau}_{n+1}}\right) $, which control the change of the
``directions and length'' of the connecting geodesic interpolants. For this we
use the improved incremental energy estimate \eqref{eq:EulerEstim1} which
allows us to invoke a telescope sum by inserting $\ob = x^\tau_{n+1}$.

\begin{lemma}[Controlling the incremental errors]
\label{le:error.estimate}
Let $\lambda \leq 0$, $\tau \in (0,1)$ with $\tau\lambda >-1/2$, 
$ x^{\tau}_{0}\in \A_{\kappa_{0}},$ and let
$x^{\tau}_{n} $ be defined iteratively by the MM scheme. For $T>0$ and
$\kappa\geq \kappa_0 $ assume $x^{\tau}_{n}\in \A_{\kappa}$ as long as 
$n< {T}/{\tau}+1$. Then, the error function $\Er^{\tau}$ defined in 
\eqref{error} in Proposition \ref{pr:ErrorProduc} satisfies the estimate 
\begin{equation}
\label{error.estim2}
 \|\ee^{2\last t}\Er^{\tau}\|_{L^{1}([0,T])}\leq 
   \tau \left( 4 + \tau \kappa\right) |\partial\phi|^{2}(x^{\tau}_{0}),
\end{equation}
where we recall that $\last=2\min\{0,\lambda\}-2.$
\end{lemma}
\begin{GSproof}
\COLORW
\begin{proof}  Throughout the proof we use the abbreviation
$S_0=|\pl\phi|(x^\tau_0)$.

\underline{Step 1. Estimate of $\sfd_{\calX}(x^\tau_n,x^\tau_{n+1})$.} 
By \eqref{Euler2b} and \eqref{Euler2c}  we have 
\[
\sfd_{\calX}(x^\tau_{n-1},x^\tau_n) \leq \frac{\tau}{1+\lambda \tau} 
|\pl \phi|(x^\tau_{n-1}) \quad \text{and} \quad 
|\pl \phi|(x^\tau_{n}) \leq \frac1{1+\lambda \tau} |\pl \phi|(x^\tau_{n-1}). 
\] 
Because of $\lambda \tau\geq -1/2$ we have $1/(1+\lambda \tau) \leq
1+2|\lambda|\tau$ and using induction we obtain 
\begin{equation}
  \label{eq:IncremEstim}
  \sfd(x^\tau_{n-1}, x^\tau_n) 
\leq \tau\,(1{+}2|\lambda|\tau)^n\,S_0 
\leq \tau \,\ee^{2|\lambda| n\tau} \,  S_0 \quad \text{for }
n \in \N.
\end{equation} 

\underline{Step 2. Estimate of $\invd^2$.}
We use the definition of $\invd^2$ in \eqref{eq:Delta2}
  and exploit the discrete energy estimate \eqref{evi.for.interpol} with
  $\ob = x^\tau_{n+1}$:  
\begin{equation*}
  \begin{split}
    &\invd^{2}\left(\geo{x^{\tau}_{n} x^{\tau}_{n-1}},\geo{x^{\tau}_{n}
          x^{\tau}_{n+1}}\right)   \overset{\text{def}}=  
    \sfd_{\calX}^{2}(x_{n-1},x_{n})+\sfd^{2}(x_{n},x_{n+1})+2
        \langle\geo{x^{\tau}_{n} x^{\tau}_{n-1}},\geo{x^{\tau}_{n}
          x^{\tau}_{n+1}}\rangle_{\rmu\rmp}
\\
    &\overset{\eqref{eq:EulerEstim1}}{\leq}\
    \sfd_{\calX}^{2}(x^{\tau}_{n+1},x^{\tau}_{n}) 
   +\sfd_{\calX}^{2}(x^{\tau}_{n},x^{\tau}_{n-1})
    +2\tau \left(\phi(x_{n+1}^{\tau})-\phi(x_{n}^{\tau})\right)
    -\tau \lambda\sfd_{\calX}^{2}(x_{n+1}^{\tau},x_{n}^{\tau}). 
\\ 
& \ \ \leq \left( 1 +|\lambda|\tau\right) \sfd_{\calX}^{2}(x^{\tau}_{n+1},x^{\tau}_{n}) +
  \sfd_{\calX}^{2}(x^{\tau}_{n},x^{\tau}_{n-1}) \leq \left( 2 +|\lambda|\tau\right) 
   \tau^2 \,\ee^{4|\lambda|n\tau} S_0^2,
  \end{split}
\end{equation*}
where the last two estimates follow from $\phi(x^\tau_{n+1}) \leq \phi(x^\tau_n)$
and \eqref{eq:IncremEstim}, respectively.

\underline{Step 3. Full estimate.} Using Step 1, for the definition of $\Er^\tau_0$ in
\eqref{error} we find  
\[
\Er^\tau_0 \leq (1{+}2|\lambda|) \tau^2(1{+}2|\lambda|\tau)^2 S^2_0 + (2 {+}
2|\lambda|\tau) S^2_0 \leq 12 S_0^2, 
\]
where we used $0<\tau<1$ and $-1/2 < \tau\lambda<0$. For $n\in \N$
the definition of $\Er^\tau_n$ gives 
\[
\Er^\tau_n \leq \left( 1 {+} 2|\lambda| {+} \kappa/\tau\right)
\tau^2 \ee^{4|\lambda| n \tau} S_0 + \left( 2{+}|\lambda| \tau\right)
\ee^{4|\lambda | n\tau} S^2_0 \leq \left( 6 + \tau \kappa\right) \ee^{4
  |\lambda| n\tau} S_0^2.
\] 

Recalling $\last = 2\min\{0,\lambda\}-2 \leq -2$ we find 
\begin{align*}
\big\| \ee^{2\last t} \Delta^\tau\big\|_{\rmL^1([0,T])} & \leq
 \sum_{n=0}^{\lfloor T/\tau \rfloor+1} \int_{n\tau}^{(n+1)\tau} \ee^{2\last t} \Er^\tau_n \d t 
 \leq \sum_{n=0}^{\lfloor T/\tau \rfloor+1} \frac\tau{2|\last|}\, 
  \ee^{2\last n\tau} \,\Er^\tau_n
\\
& \leq \frac{\tau}{4}\left(12+ \tau \kappa\right)S_0^2 \sum_{n=0}^\infty
\ee^{-4n\tau} \leq \tau \left( 4+ \tau \kappa\right) S_0^2,
\end{align*}
which is the desired result.  
\end{proof}
\COLORWend
\end{GSproof}

\subsection{Proof of the main abstract result in Theorem
  \ref{th:MainExistTheo}}\label{su:ProofMainExist} 
 
Having prepared the above the preliminary estimates for the solutions
$x^\tau_n$ of the minimizing movement scheme, we are now ready to give
the proof of the abstract existence result. The important point in the proof is
that the value $\kappa$ of the $\kappa$-concavity of $\sfd^2$ is occurring only
in a few places that are well controlled. In particular, it is needed only on
the time-discrete level (see e.g.\ \eqref{akas} and \eqref{error}), 
but it disappears in the EVI 
formulation.\bigskip

\noindent\begin{proof}[Proof of Theorem \ref{th:MainExistTheo}]
\mbox{} \\
Let $\lambda$ such that $\phi$ is geodesically $\lambda$-convex, then $\phi$ is
also geodesically $\min\{\lambda,0\}$-convex. As before we set
$\last=2\min\{\lambda,0\}-2$.

\underline{Step 1. Limit passage on approximate solutions $\xx^{\tau_k}$.} 
We now exploit the assumptions (A1) to (A3) of Theorem \ref{th:MainExistTheo}.  
For a given initial point $x_0 \in \cup_{\kappa \in \N}\A_{\kappa}$ there
exist $\kappa_0\in \N$ such that $T(x_0,\kappa)>0$ for all $\kappa>\kappa_0$.    
Fixing a $\kappa_*>\kappa_0$ we define time steps $\tau_k= T/k$ 
for $k\in \N$, where $T=T(x_0,\kappa_*)$ from assumption (A2). Hence, we have
$x^{\tau_k}_n \in \A_{\kappa_*} $ as long as $n  < T/\tau_k +1$. By
construction of the geodesic interpolants in 
\eqref{geodesic.interpolation} the function
$\xx^{\tau_k}:[0,T] \to \calX$ satisfies $\xx^{\tau_k}(t)\in\A_{\kappa_*}^{\Geod}$
for all $t\in[0,T]$.

Thus, we are able to apply Corollary \ref{co:error} and Lemma
\ref{le:error.estimate} and obtain (for $\tau_k,\tau_{k'}\leq 1$) 
\begin{equation}
	\sup_{t\in[0,T]}\sfd_{\mathcal{X}}(\xx^{\tau_{k}}(t),\xx^{\tau_{k'}}(t))\leq
        \ee^{-\last T} 2\left((\tau_{k}{+}\tau_{k'})(4{+}\kappa_*)\right)^{1/2}|\pl \phi|(x_0).
\end{equation}
Therefore the curves $\xx^{\tau_{k}}:[0,T]\to \calX$ converge uniformly in the
compact and hence complete sublevel $\bigset{ y \in \calX}{ \phi(y) \leq
  \phi(x_0)}$ to a continuous limiting curve $\xx:[0,T] \to  \calX$ with
$\xx(0)=x_0$. 

\underline{Step 2. $\xx$ is the unique EVI$_{\last}$ solution.}  We now return to
the approximate EVI formulation \eqref{evi.for.interpol} for the interpolants
$\xx^{\tau_k}$. 

For a general observer, by assumption A3, for $\ob  \in\overline{\dom(\phi)}$ we can choose a
sequence $(y_m)_{m\in \N}$ with  
\[
y_{m}\in\A_{\kappa_{m}}\subset\calX, \quad 
y_{m}\overset{m \to \infty}{\longrightarrow} \ob , \quad  
\phi(y_{m})\overset{m \to
  \infty}{\longrightarrow}\phi(\ob ).
\]
Without loss of generality we may assume $\kappa_* \leq \kappa_m$. 
Choosing $\ob=y_m \in \A_{\kappa_m}\subset \A^{\Geod}_{\kappa_m}$ in
\eqref{evi.for.interpol} for $\xx^{\tau_k}$ and
integrating over the interval $(s,t)$ we find 
\begin{equation}
  \begin{split}
    &\frac{1}{2}\sfd_{\mathcal{X}}^{2}(\xx^{\tau_{k}}(t),y_{m})
      -\frac{1}{2}\sfd_{\mathcal{X}}^{2}(\xx^{\tau_{k}}(s),y_{m})
         +\int_{s}^{t}\left(\phi(\xx^{\tau_{k}}(r))+\frac{\last}{2}
           \sfd_{\mathcal{X}}^{2}(\xx^{\tau_{k}}(r),y_{m})\right)\d r\\
   &\leq(t{-}s)\phi(y_{m}) + (t{-}s) \ee^{-2\last t} \tau_k 
      (4{+}\tau_k \kappa_{m})\,|\partial\phi|^{2}(x_{0}).
\end{split}
\end{equation} 
Keeping $m$ fixed, taking $k \to \infty $, and using the lower
semicontinuity of $\phi,$ we obtain 
\begin{equation*}
  \frac{1}{2}\sfd_{\mathcal{X}}^{2}(\xx(t),y_{m})-\frac{1}{2}\sfd_{\mathcal{X}}^{2}(\xx(s),y_{m}) 
+\int_{s}^{t}\left(\phi(\xx(r))+\frac{\last}{2} 
\sfd_{\mathcal{X}}^{2}(\xx(r),y_{m})\right)\d r
 \leq (t{-}s)\phi(y_{m}), 
\end{equation*}
Note that $\kappa_m$ has disappeared because of $\tau_k\to 0$ for $k\to
\infty$. Now $m \to \infty$ yields 
\begin{equation*}
  \frac{1}{2}\sfd_{\mathcal{X}}^{2}(\xx(t),\ob )-\frac{1}{2}\sfd_{\mathcal{X}}^{2}(\xx(s),\ob
  )+\int_{s}^{t}\left(\phi(\xx(r))+\frac{\last}{2}\sfd_{\mathcal{X}}^{2}(\xx(r),\ob )\right)
  \d r\leq (t{-}s)\phi(\ob ),
\end{equation*}
where we have convergence on the left-hand side and use the lsc of $\phi$ on
the right-hand side.
As the inequality trivially holds for $\ob \in
\calX\setminus\overline{\dom(\phi)}$, we have shown that
$\xx:[0,T]$ is an 
EVI$_{\last}$ solution. The uniqueness follows by applying Corollary
\ref{co:error} with $\Delta^\tau= \Delta^\sigma=0$. 

\underline{Step 3. Extension to $t\in [0,T_\infty(x_0))$.} In the previous
step the solution $\xx:[0,T(x_0,\kappa_*)]\to \calX$ was well-defined and
unique. However, $\kappa_* > \kappa_0(x_0)$ was arbitrary. Hence, we can extend
the solution uniquely to any interval $[0,T(x_0,\kappa)]$ with
$\kappa>\kappa_0$. Taking the limit $\kappa \to \infty$ we obtain a unique
solution on $[0,T_\infty(x_0)) \subset \cup_{\kappa >\kappa_0}
[0,T(x_0,\kappa)]$. 
 
\underline{Step 4. Complete EVI flow on $\overline{\dom(\phi)} $.} We now
further assume $T_\infty(x_0)=\infty$ for all $x_0 \in \cup_\kappa \A_\kappa$.
We now consider an arbitrary $x_0 \in \overline{\dom(\phi)} $. Since
$\overline{\cup \A_{\kappa}}=\overline{\dom(\phi)},$ there exists a sequence
$(x^m_0)_{m\in \N}$ with $x^{m}_0\in \A_{\kappa_m} \subset \cup_{\kappa\in\N} 
\A_{\kappa}$ and $x^{m}\overset{m \to \infty}{\longrightarrow} x$. Define
$\xx^{m}:[0,\infty)\to X$ to be the unique EVI$_{\last}$ solution starting in
$x^{m}_0$. By \eqref{eq:expanding} in Proposition \ref{pr:EVI-int}(C),
we have
\begin{equation*}
	\ee^{\last t}\sfd_{\mathcal{X}}(\xx^{m}(t),\xx^{m'}(t))\leq  \sfd_{\mathcal{X}}(\xx^{m}_0,\xx^{m'}_0)
        \quad  \text{for all }  t\geq 0.
\end{equation*}
Thus, for all $T>0$ the sequence $\xx^{m}$ is Cauchy in the space
$C([0,T];\calX)$. Therefore, it converges locally uniform to a limit
$ \xx : [0,\infty) \to \calX$, which satisfies the initial condition
$\xx(0)=x_0= \lim x^m_0$. Since each curve $\xx^m$ satisfies the
integrated form \eqref{eq:EVI.inte} of EVI$_{\last}$, the lower semicontinuity
of $\phi$ guarantees that the limit curve $\xx$ is again an EVI solution.

\underline{Step 5. Correcting $\last$ back to $\lambda$.} Above we have
constructed EVI$_{\last}$ solutions, but our functional $\phi$ is geodesically
$\lambda$-convex and $\lambda > \last$. To recover the correct $\lambda$, 
we can apply \cite[Cor.\,3.12]{MurSav20GFEV} because we know that
$\xx$ is an EVI$_{\last}$ solution for $(\calX,\sfd,\phi)$ and that
$\phi$ is $\lambda$-convex with $\lambda \geq \last$. Hence, $\xx$ is
also EVI$_{\lambda}$ solution. 
\end{proof}

\section{Semiconcavity and EVI flows for 
$(\calM(\YX),\E,\HK)$ and   $(\calP(\YX),\E,\SHK)$ } 
\label{se:SemicEVI}

We now combine the theory developed in the previous two sections, namely the
existence result for EVI flows provided in Theorem \ref{th:MainExistTheo} with
the semiconcavity results established in \cite[Sec.\,4]{LasMie19GPCA}. 

\subsection{Semiconcavity of $\frac12 \HK^2$ and $\frac12\SHK^2$}

In order to apply Theorem \ref{th:MainExistTheo} in the case of $\HK, \SHK,$ we
need to provide some semiconcavity results. More specifically, we need to prove
that point (A1) is satisfied for a sequence of sets $A_{\kappa}.$ Before we
proceed, we will define the following two collections of sets.  For
$\delta \in (0,1)$ we define the set
\begin{equation}
\label{eqdef:Mdelta}
\M_{\delta}(\YX)=\bigg\{\mu\in\M(\YX) : \mu\ll \Ld , \  \delta\leq\frac{d\mu}{d \Ld }(x)\leq\frac{1}{\delta}, \text{ for }  \Ld \text{-a.e.\ } x\in X \bigg\}.
\end{equation}
For positive numbers $\sfd_{1},\sfd_{2},$ we also define
\begin{equation}
\label{eqdef:Md}
	\widetilde{\M}_{\sfd_{1},\sfd_{2}}(\YX)=\bigg\{\mu\in\M(\YX)
          : \forall x\in X:\ \sfd_{2}\leq
	\frac{\mu\left(B\left(x,\sfd_{1}\right)\right)} 
         { \Ld (B\left(x,\sfd_{1}\right))} 
	\leq\frac{1}{\sfd_{2}} \bigg\}.
\end{equation} 
It is straightforward to see that for all $\sfd_{1}>0$ it holds
$\M_{\delta} (\YX) \subset
\widetilde{\M}_{\sfd_{1},\delta}(\YX).$ Furthermore all
elements in $\M_{\delta}(\YX)$ have total mass
bounded by $\frac{1}{\delta} \Ld (\YX)$.

In \cite[Thm.\,4.8]{LasMie19GPCA}, it was stated and proved that for a set
$X \subset \R^d$ that is compact, convex and with nonempty interior, there
exists $\kappa(\delta)\in\mathbb{R},$ such that $(\M(\YX),\HK)$ is
$\kappa$-concave on $\M_{\delta}(\YX).$ We clarify at this point,
that in practice Theorem 4.8 was stated for more general metric spaces and for
reference measures $\nu$ that are doubling. However for simplification
we are going to recall any theorems or lemmas we need from \cite{LasMie19GPCA}
directly adapted to to our setting, avoiding all the extra generality related
to doubling measures and abstract metric spaces. We remind the reader, that for
a compact, convex set $X \subset \R^d $ with nonempty interior, the Lebesgue
measure is doubling and the Euclidean distance is 2-concave. Although
\cite[Thm.\,4.8]{LasMie19GPCA} was stated in this weaker form, the given
proof provides a stronger result, namely the following:

\begin{theorem}[$K$-concavity for $(\M(\YX),\HK) $]\label{thm:K.Semi} 
  Let $X \subset \R^d$ be a compact, convex set with nonempty interior. Then,
  there exists $\kappa(\delta)\in\mathbb{R},$ such that $\HK$
  is $\kappa$-concave on $\M_{\delta}(\YX),$ with respect to
  observers in $\M^{\Geod}_{\delta}(\YX).$
\end{theorem}

\noindent
At the moment of writing \cite{LasMie19GPCA}, we were not aware that this
version will be useful, however now this property along with LAC condition, is exactly the
assumption (A1) in our Theorem \ref{th:MainExistTheo}. LAC condition was proven in \cite[Theorem 4.1]{LasMie19GPCA}. Now, we will
recall some lemmas from there and provide a short proof of Theorem~\ref{thm:K.Semi}.

\begin{lemma}[{\cite[Lem.\,4.9]{LasMie19GPCA}}]
\label{unif.bounds.density}  
Let $X \subset \R^d$ be a compact, convex set with nonempty
interior. There 
  exists $0<C_{\min}\leq C_{\max}$ such that for every
  $\mu_{0},\mu_{1}\in \widetilde{\M}_{\sfd_{1},\sfd_{2}}(X)$
  and any optimal plan $\bfH_{01}$ for
  $\LET_{\dd}(\,\cdot\,;\mu_0,\mu_1)$ we have
\begin{equation}
\label{bounds} 
C_{\min}\leq\sigma_{i}(x_{i})\leq C_{\max},\hspace{10pt}
\eta_{i}\text{-a.e.\ } 
\end{equation}
where $\eta_i= \pi^i_\# H_{01}= \sigma_i \mu_i$ for $i=0,1$.  Furthermore, any
transportation happens in distances strictly less than some $\frac{\pi}{2},$
i.e.\ there  exists $\mathfrak{D}<\frac{\pi}{2}$ that depends only on
$\sfd_{1},\sfd_{2},$ such that $\sfd_{X}(x_{0},x_{1})\leq \mathfrak{D}$ for
$H_{01}$ almost every $(x_{0},x_{1}).$
\end{lemma}

\begin{lemma}[{\cite[Lem.\,4.10]{LasMie19GPCA}}] \label{MdeltatoMd1d2}
  Let $X \subset \R^d$ be a compact, convex set with nonempty
interior and $\M_{\delta}(\YX)$ be as in
  \eqref{eqdef:Mdelta}. Then, for each $\delta>0$ there exist
  $\sfd_{1} \in (0,\frac{\pi}{2})$ and $\sfd_2>0$ such that any
  constant-speed geodesic $\bf{\mu}_{01}$ connecting $\mu_{0}$ to $\mu_{1}$
  with $\mu_0,\mu_1\in\M_{\delta}(\YX)$
  satisfies $\bfmu_{01}(t)\in \widetilde{\M}
  _{\sfd_{1},\sfd_{2}}(\YX)$ for all $t\in [0,1].$
\end{lemma}

From Lemma \ref{MdeltatoMd1d2} we obtain
\[
\M_{\delta}(X)\subset \GeoCov{\M_{\delta}}(X) \subset \widetilde{\M}_{\sfd_{1},\sfd_{2}}(X).
\]
  
For the proof of  \cite[Thm.\,4.8]{LasMie19GPCA} also 
 the following result is used. 

\begin{lemma}[{\cite[Lem.\,4.11]{LasMie19GPCA}}] 
\label{exist Lambdas}
Let $X \subset \R^d$ be a compact, convex set with nonempty interior and
$\widetilde{\M}_{\sfd_{1},\sfd_{2}}(\YX)$ be as in \eqref{eqdef:Md}. Then,
there exist $R_{\min},R_{\max}>0$ that depend on $\sfd_{1},\sfd_{2},$ such that
for $\mu_{0},\mu_{1}$ with
$\bm{\mu}_{01}(t)\in\widetilde{\M}_{\sfd_{1},\sfd_{2}}(\YX)$ and
$\mu_{2}\in\widetilde{\M}_{\sfd_{1},\sfd_{2}}(\YX)$ we can find measures
$\lambda_{0},\lambda_{1},\lambda_{2},\lambda_{t} \in \calP_{2}(\mathfrak{C}
[R_{\min},R_{\max}]) $ with
\[
\mathfrak{P}\lambda_{i}=\mu_{i},\hspace{8pt}\mathfrak{P}\lambda_{t} 
=\bm{\mu}_{01}(t),\hspace{8pt}
\sfW_{\sfd_{\mathfrak{C}}} (\lambda_{i},\lambda_{t}) =
\HK(\mu_{i},\bm{\mu}_{01}(t)) \quad \text{ for }
i=0,1,2.
\]
\end{lemma}

The observer's location does not necessarily have to be within $\M_{\delta}$,
as shown in both Lemma \ref{exist Lambdas} and the actual proof of
\cite[Thm.,4.8]{LasMie19GPCA}. It is sufficient for the observer to be situated
in $\widetilde{\M}_{\sfd_{1},\sfd_{2}}(\YX)$, which includes
$\GeoCov{\M}_{\delta}(\YX)$. Consequently, the same proof can be applied
without any changes.\bigskip
 
\noindent
\begin{proof}[Proof of Theorem \ref{thm:K.Semi}]
By Lemma \ref{MdeltatoMd1d2} there exists $0<\sfd_{1}<\frac{\pi}{2}$ and
$0<\sfd_{2}$ such that every geodesic $\mm_{01}$ connecting
$\mu_{0},\mu_{1}\in\M_{\delta}(X)$ satisfies
$\mm_{01}(t)\in \widetilde{\M}_{\sfd_{1},\sfd_{2}}(X)$ for all $ t\in [0,1]$.
We also have
$\mu_{2}\in \widetilde{\M}_{\sfd_{1},\sfd_{2}}(X) \supset
\GeoCov{\M_{\delta}}(X)$.
	
We would like to utilize the equivalent definitions of $K$-concavity provided
in \cite[Cor.\,2.24(iii)]{LasMie19GPCA} where a function
$f:[0,1] \rightarrow \mathbb{R}$ is $K$-concave if for every
$t_{1},t_{2} \in [0,1]$ with $t_1<t_2$ the mapping
$\tilde{f}^{[t_{1},t_{2}]}_{i}(t)=f_{i}\left( t_1 {+} t(t_2 {-}t_1) \right)$
satisfies
\begin{equation}
\tilde{f}^{[t_{1},t_{2}]}_{i}(t)+Kt (1{-}t) (t_2{-}t_1)^2  \geq  (1{-}t)
\tilde{f}^{[t_{1},t_{2}]}_{i}(0) +t\tilde{f}^{[t_{1},t_{2}]}_{i}(1)
 \text{ for all }t\in [0,1]. 
\end{equation}
In that direction, we take
$\tilde{\mu}_{0}=\mm_{01}(t_{1})$, $\tilde{\mu}_{1}=\mm_{01}(t_{2})$ for
$t_{1},t_{2}\in[0,1],$ and
$\tilde{\mm}_{01}(t)=\mm_{01}(t(t_{2}{-}t_{1})+t_{1}).$ By Lemma \ref{exist
  Lambdas}, there exists $R_{\min},R_{\max}$ that depend on
$\sfd_{1},\sfd_{2},$ and therefore on $\delta,$ such that for every
$\tilde{\mu}_{0},\tilde{\mu}_{1},\tilde{\mu}_{2}\in\widetilde{\M}_{\sfd_{1},\sfd_{2}}(X)$
and $0<t<1$ we can find measures
$\lambda_{0},\lambda_{1},\lambda_{2},\lambda_{t} \in \mathcal{P}_{2}(\mfC
[R_{\min},R_{\max}]) $ with
\begin{equation}
  \label{eq:La.mu}
  \mathfrak{P}\lambda_{i}=\tilde{\mu}_{i},\hspace{8pt}\mathfrak{P}\lambda_{t} 
  =\tilde{\mm}_{01}(t),\hspace{4pt} \text{and} \hspace{4pt} 
  \sfW_{\sfd_{\mfC}}(\lambda_{i},\lambda_{t}) 
  =\HK(\tilde{\mu}_{i},\tilde{\mm}_{01}(t)),\hspace{4pt} i=0,1,2,
\end{equation}
 see Theorem \ref{thm:OTcone}. 
Using the geodesic property of $\tilde{\mm}_{01}$ yields 
\begin{align*}
\sfW_{  \sfd_{\mfC}}(\lambda_{0},\lambda_{t}) + \sfW_{
    \sfd_{\mfC}}(\lambda_{1},\lambda_{t}) &=
  \HK(\mu_{0},\tilde{\mm}_{01}(t))+\HK(\mu_{1},\tilde{\mm}_{01}(t))\\
&= \HK(\tilde{\mu}_{0},\tilde{\mu}_{1})\!\leq\!
  \sfW_{ \sfd_{\mfC}}(\lambda_{0},\lambda_{1}).
\end{align*} 
Hence, it is straightforward to see that there exists a geodesic $\lala_{01}$
connecting $\lambda_{0},\lambda_{1},$ such that $\lala_{01}(t)=\lambda_{t}.$
Furthermore, by \cite[Thm.\,6]{Lisini2006a} there is a plan
$\bfLambda_{0\rightarrow 1}$ on the geodesics such that
$\Lambda_{ts}:=(e_{t},e_{s})_{\sharp}\bfLambda_{0\rightarrow1}$ is an optimal
plan between $\lala(t)$ and $\lala(s).$ Now, by using a gluing lemma, we can
find a plan $\bfLambda^{0 \to 1}_{2t}$ in
$\mathcal{P}((C[0,1];\mfC)\times\mfC),$ such that
$\Lambda_{01}=(e_{0},e_{1})_{\sharp}\left(\pi^{0\rightarrow
    1}_{\sharp}\bfLambda^{0 \to 1}_{2t}\right),$ and
$(e_{t}(\pi^{0\rightarrow 1})\times I)_{\sharp}\bfLambda^{0 \to 1}_{2t}$ is an
optimal plan for $ \sfW_{\sfd_{\mfC }}(\lambda_{2},\lala_{01}(t)).$ Finally by
applying the last part of Lemma \ref{unif.bounds.density}, we get the existence
of a $\mathfrak{D}<\frac{\pi}{2}$ such that $|x_{2}{-}x_{t}|<\mathfrak{D}$ for
$(e_{t}(\pi^{0\rightarrow 1})\times I)_{\sharp}\bfLambda^{0 \to 1}_{2t}$ almost
every $(z_{2},z_{t}),$ similarly $|x_{0}{-}x_{1}|<\mathfrak{D}$ for
$\Lambda_{01}$ almost every $[z_{0},z_{1}].$ Therefore, for
$\bfLambda^{0 \to 1}_{2t}$ almost every $(z_{2}, \zz(\cdot,z_{0},z_{1})),$
where $\zz(\cdot,z_{0},z_{1})$ is a geodesic connecting $z_{0},z_{1},$ we have
$x_{0},x_{1},x_{2}, \bar{\xx}(t,z_{0},z_{1})\in
B\left(\bar{\xx}(t,z_{0},z_{1}),d\right).$ By \cite[Prop.\,2.27]{LasMie19GPCA}
we get a $K'$ such that 
\begin{equation}
\label{eq:C.Kconcave}
  \sfd^{2}_{\mfC }(z_{2},\zz(t,z_{0},z_{1})) +
   K't(1{-}t)\sfd^{2}_{\mfC }(z_{0},z_{1})\geq
   (1{-}t) \sfd^{2}_{\mfC }(z_{2},z_{0})+ t\, \sfd^{2}_{\mfC }(z_{2},z_{1}),
\end{equation}
for $\bfLambda^{0 \to 1}_{2t}$ almost every $(z_{2}, \zz(\cdot,z_{0},z_{1})).$
By integrating with respect to $\bfLambda^{0 \to 1}_{2t},$ we find
\begin{equation}
    \sfW_{\sfd_{\mfC }}^{ 2}(\lambda_{2},\lala_{01}(t)) +
    K't (1{-}t) \sfW_{\sfd_{\mfC}}^{ 2}
    (\lambda_{0},\lambda_{1})\geq   
   (1{-}t) \sfW_{\sfd_{\mfC}}^{ 2}(\lambda_{2},\lambda_{0})+ 
   t \, \sfW_{\sfd_{\mfC}}^{ 2} (\lambda_{2},\lambda_{1}).
\end{equation}
Using \eqref{eq:La.mu} we find the desired semiconcavity, and Theorem
\ref{thm:K.Semi} is proved. 
\end{proof}

\subsection{Geodesic semiconvexity of functionals on $\HK$ and $\SHK$}
\label{su:HK.EVI.flow}

In \cite{LiMiSa23FPGG} the question of geodesic $\lambda$-convexity of
functionals $\E$ with reference measure $\calL^d$ on a $d$-dimensional
domain are discussed in detail. It is shown that $\E$ defined as in
\eqref{EntropyFunctional} in terms 
of a lsc and convex density functions $E$ with $E(0)=0$ is
$\lambda$-convex on $(\calM(X),\HK)$ if and only if the auxiliary function
\[
N_{E,\lambda}:(0,\infty)^2 \to \R \cup \{\infty\}; \ (\rho,\gamma) \,\mapsto \, 
\left( \tfrac{\ds\rho}{\ds\gamma}\right)^d E \left( \tfrac{\ds\gamma^{2+d}}{\ds
  \rho^d} \right) - \frac\lambda2 \gamma^2 
\] 
satisfies the following two conditions:
\begin{equation}
  \label{eq:NE.conds}
  \begin{aligned}
  & N_{E,\lambda} : (0,\infty)^2 \to \R \cup\{\infty\} \ \text{ is convex and }
\\
 & \rho \mapsto  (d{-}1)\,N_{E,\lambda} (\rho,\gamma) \ \text{ is non-increasing.} 
  \end{aligned}
\end{equation}
It is shown that the density functions $E$ of the form 
\[
E(c) = \alpha_0 c +   \alpha_1 c^{p_1} + \cdots +  \alpha_m c^{p_m} 
\]
lead to geodesically $2\alpha_0$-convex $\E$ if $\alpha_0\in \R$ and
$\alpha_i\geq 0$ and $p_i >1$ for $i=1,\ldots,m$. Moreover 
in dimensions $d\in \{1,2\}$ the density function $E(c) = - 
\beta c^q$  with $\beta \geq 0$ and $q \in [d/(d{+}2),1/2]$ lead to
geodesically convex functionals $\E$.

So far there doesn't seem to be a theory for semiconvexity on
$\calP(\YX),\SHK)$ which can be used to provide examples. In Appendix
\ref{se:TransferLambdaCvx} we establish the following nontrivial class of examples.

\begin{proposition}\label{pr:q.leq.1} Consider dimension $d\in \{1,2\}$ and a
bounded convex domain $\YX\subset \R^d$ that is the closure of an open
set. Then, the functional $\E_q$ defined via 
\[
\E_q(\mu)= - \int_X \rho^{\,q} \:\d x  \quad \text{for } \mu=\rho \,\d x +
\mu^\perp
\]
is geodesically convex on $(\calP(\YX),\SHK)$ if $q \in [d/(d{+}2), 1/2]$. 
\end{proposition} 

With this result we are sure that the following main existence result for EVI
flows on $(\calP(\YX),\E,\SHK)$ provides at least the solutions to a small, but
nontrivial family of nonlinear partial differential equations.  Following the
discussion in Section \ref{se:EVI.vs.PDE}, the partial differential equations
associated with such functionals. For the case $q=1/2$ we explicitly obtain the
nonlinear PDE
\[
\dot \rho = \frac12 \Delta \sqrt\rho + 2 \Big(\sqrt\rho - \rho \int_X \sqrt\rho
\,\d x  \Big) \text{ in } X, \quad \nabla \rho\cdot \nu=0 \text{ on }\pl X
\]
which defines a contraction semiflow (EVI flow) in the space $\calP(X),\SHK)$.

\subsection{The Main Result}
\label{su:MainResult}

In this section we collect the results from the previous sections and 
provide the proof of our main result, which we repeat here for convenience.

\begin{theorem}\label{mainmain}
Let $X \subset \R^d$ be a compact, convex set with nonempty
interior. Furthermore, let $\E$ be a functional defined as in
\eqref{EntropyFunctional} that satisfies Assumption \ref{BasicAssumpt}. 

Then, for all $\mu_0=\rho_0\calL^d$ with
$0<\ul\rho_0 \leq \rho_0(x) \leq \ol\rho_0< \infty$ a.e.\ in $\YX$, the
geodesically interpolated solutions of the MM schemes corresponding to the
gradient system $(\calM(X),\E,\HK)$ ($(\calP(X),\E,\SHK)$), as in
$\eqref{scheme}_{\HK}$ ( $\eqref{scheme}_{\SHK}$), converge to a complete
solution $\bm{\mu}:[0,\infty)\to \calP(\YX)$ of EVI$_\lambda$.  Moreover, for
all $\mu_0 \in \ol{\dom(\E)}^{\HK} \subset \calM(X)$ (
$\mu_{0}\in \ol{\dom(\E)}^{\SHK}\subset \calP(X)$) there exists a unique EVI
solution, emanating from $\mu_{0}$.
\end{theorem} 
\begin{proof}
We start from $\mu_0=\rho_0\calL^d$ with
$0<\ul\rho_0 \leq \rho_0(x) \leq \ol\rho_0< \infty.$ By Propositions
\ref{discrete maximal principle SHK} (for $\SHK$) and 
\ref{pr:Iteration.HK} (for $\HK$) we know that for every $T$ there
exists a $\tau_{0}>0$, such that for 
all $\tau\leq\tau_{0}$ and $n$ with $n\tau<T$, the solutions $\mu_{n}$ of the
MM scheme satisfy $\mu_{n}=\rho_n\calL^d$ where
$0<\delta\leq \rho_n(x) \leq 1/\delta< \infty$. But this means that 
$\mu_{n}\in \overline{\M}_{\delta}(\YX)$ for all $n$ with $n\tau<T$. Therefore, by
taking $A_{\kappa}= \overline{\M}_{\delta}(\YX)$ and using Theorem
\ref{thm:K.Semi} we see that all assumptions of Theorem \ref{th:MainExistTheo}
are satisfied. We note that LAC was proved in \cite[Theorem 4.1]{LasMie19GPCA}. Therefore, the convergence of the MM scheme for density restricted
initial data $\mu_0$ follows. 

For general initial data $\mu_0$ in the closure $\ol{\dom(\E)}^{\SHK}
\subset \calP(X)$ or $\ol{\dom(\E)}^{\HK}\subset \calM(X)$ of the domain of the
functional $\E$  we can choose approximations $\mu_0^{m}= \rho_0^m\calL^d$
satisfying $\rho_0^m\in [1/m,m]$ a.e.\ in $X$. The associated EVI solutions
$\mu_m$ are complete and converge to the desired, complete EVI solution emanating from
$\mu_0$ by applying \eqref{eq:Lipschitz}. 
\end{proof}

\appendix
\section{Transfer of $\lambda$-convexity between $\HK$ and $\SHK$}
\label{se:TransferLambdaCvx}

We rely on the interpretation of $(\calM(\YX),\HK)$ as a cone of over
$(\calP(\YX), \SHK)$ that was developed in \cite{LasMie19GPCA}. But first we
consider a general geodesic space $(\calX,\sfd) $ and the associated cone
$(\calC,D)$, which take the places of
$(\calP(\YX), \SHK)$ and $\calM(\YX),\HK)$, respectively. 

We present a general result that demonstrates, under appropriate conditions, the geodesic convexity of a $p$-homogeneous functional $\sfF$ on $(\calC,D)$. Specifically, we show that when $\sfF$ is restricted to $(\calX,\sfd)$, it remains geodesically convex.

\begin{proposition}[Transfer for negative, homogeneous functionals]
\label{pr:Transfer}
Assume that $\sfF:\calC\to [{-}\infty,0]$ is geodesically convex on
$(\calC,D)$, and that it is $p$-homogeneous for some $p\geq 1/2$, i.e.\
\[
\sfF([x,r]) = r^p \sfF([x,1]) \text{ for all } [x,r]\in \calC .
\]
Moreover, assume $\sfd(x,y)<\pi$ for all $x,y\in \calX$. 

Then, $ \E(x):=\sfF([x,1])\in [{-}\infty,0]$ is geodesically convex on $(\calX,\sfd)$.
\end{proposition}
\begin{proof} We use the fact that all geodesics $x:[0,1]\to \calX$
  connecting $x$ and $x_1$ are given by the geodesics $z:[0,1]
  \mapsto [\ol x(t),r(t)] \in \calC$ by a simple reparametrization,
  see \cite[Thm.\,2.7]{LasMie19GPCA}, namely, setting
  $\delta=d(x,x_1)\in {]0,\pi[}$ we have  
\[
x(t)=\ol x(\beta_\delta(t)) \text{ with } \beta_\delta(t):= \frac{\sin(t
  \delta)}{\sin(t\delta) + \sin((1{-}t)\delta)}, 
\]
where $\beta_\delta(0)=0$ and $\beta_\delta(1)=1$.
Moreover, $r(t) =1- t(1{-}t) D([x,1],[x_1,1])^2 $ with $D([x,1],[x_1,1])
=\left(2(1{-}\cos \delta)\right)^{1/2} $  can be rewritten
via
\[
r(\beta_\delta(t))=r_\delta(t): = \frac{\sin(\delta)} 
{\sin(t\delta) + \sin((1{-}t)\delta)} \in [1/2,1], 
\]
where $r(0)=r(1)=1$.
With this we obtain 
\begin{align*}
\E(x(t))&=\sfF([x(t),1])&&(\text{definition of } \E)\\
 &= \sfF([\ol x(\beta_\delta(t)),1]) &&\text{(reparametrization)}\\
 &= \frac1{r_\delta(t)^p}  \sfF([\ol x(\beta_\delta(t)),r_\delta(t)])
   && (\text{$p$-homogeneity of } \sfF)\\
&  \leq  \frac1{r_\delta(t)^p}\:\left( \left(1{-}\beta_\delta(t)\right)\sfF([\ol x(0),r(0)]) 
     + \beta_\delta(t) \sfF([\ol x(1),r(1)]) \right)
   && (\text{geodesic cvx of $\sfF$ on }\calC)\\
&  =  \frac{1{-}\beta_\delta(t)}{r_\delta(t)^p}\,\E(x) 
     + \frac{\beta_\delta(t)}{r_\delta(t)^p}\,\E(x_1) 
   && (\text{definition of } \E)
\end{align*}
Because of $\E(x_j)\leq 0$ it suffices to show the two estimates 
\begin{equation}
  \label{eq:TwoEstim}
  \frac{1{-}\beta_\delta(t)}{r_\delta(t)^p} \geq  1-t \quad \text{and} \quad 
 \frac{\beta_\delta(t)}{r_\delta(t)^p} \geq  t \qquad \text{for all
 }t\in [0,1] \text{ and } \delta\in [0,\pi].
\end{equation}
For the
second estimate and the case $p\geq 1$ we use the explicit form and obtain  
\[
\frac{\beta_\delta(t)}{r_\delta(t)^p}= \frac1{r_\delta(t)^{p-1}}
\:\frac{\sin(t\delta)}{\sin\delta} \geq  t \quad \text{for all }t\in [0,1],
\]
where we used $r_\beta(t)\leq 1$ and that $t\mapsto \sin(t\delta)$ is
concave on $[0,1]$ because of $\delta \in {]0,\pi[}$. Thus, the
result certainly holds for $p\geq 1$. 

However, it also holds for $p\in [1/2,1]$ by the following arguments.  Define
the function
\[
Q_p(t,\delta)=\frac{\sin(t\delta)}{t \,(\sin\delta)^p}
\left(\sin(t\delta)+\sin((1{-}t)\delta)\right)^{p-1}.   
\]
It suffices to show $Q_p(t,\delta)\geq 1$ for $t\in {]0,1[}$ and $\delta \in
{]0,\pi[}$. Clearly, we have $Q_p(1,\delta)=1$ and we find 
\[
\pl_t Q_p(1,\delta)=-1 +  p \frac{\delta \cos\delta}{\sin \delta} +
(1{-}p)\,\frac\delta{\sin\delta} .
\]
For $p\in [1/2,1]$ one can show that $\pl_t
Q_p(t,\delta) \leq 0$ which implies $Q_p(t,\delta) \geq Q_p(1,\delta)=1$ which
is the desired second estimate in \eqref{eq:TwoEstim}. 

To see why $p\geq 1/2$ is necessary, we use $\lim_{\delta\to \pi^-}\left( \sin
\delta\: \pl_t 
Q_p(1,\delta)\right) = \pi(1{-}2p)$. Thus, for $p<1/2$ we have $\pl_t
Q_p(t,\delta)>0$ for $\delta \approx \pi$, which implies $Q_p(t,\delta)<1$.

The first estimate in \eqref{eq:TwoEstim} follows similarly, namely by
changing $t$ to $1{-}t$. This  proves the result. 
\end{proof}

We now consider $(\calM(\YX),\HK)$ as the cone over
$(\calP(\YX),\SHK)$ for some convex and compact $\YX\subset\R^d$.  We first
observe that \cite[Thm.\,3.4]{LasMie19GPCA} guarantees $\SHK(\nu_0,\nu_1) \leq
\pi/2$  that the condition $\sfd(x,y) <\pi$ is automatically satisfied.    

From \cite{LiMiSa23FPGG} we know  that the functionals 
\[
\E_q(\mu)=\int_\YX \varrho(x)^q \:\rmd x \quad \text{ for } \mu =\varrho
\,\rmd x
\]
are geodesically 0-convex on $(\calM(\YX),\HK)$ whenever $q > 1$. Moreover, in
the sense of cones we have $\mu= r^2\varrho \Ld $ giving $p$-homogeneity
with $p =2q$, namely
\[
\E_q(\mu)=\E_q(r^2 \varrho \Ld )=\sfF_q([\varrho  \Ld ,r]) 
= r^{2q} \sfF_q([\varrho \Ld ,1])= r^{2q} \E_q(\varrho \Ld ) . 
\]
However, the above result is not applicable because of $\sfF_q(\mu)\geq 0$. 

Note also that the special case $q=1$ leads to the mass functional
$\E_\rmM(\mu)=\E_1(\mu)=\int_\YX \mu(\d x)=\mu(\YX)$ which is 
geodesically 2-convex for $\HK$ (as well as geodesically 2-concave). 
However, its spherical restriction is
obviously constant, hence it is geodesically $0$-convex and $0$-concave. 
This means that we have a drop in the convexity, namely 
\[
0= \Lambda_{\SHK} \lneqq \Lambda_\HK = 2.
\]

However, the above result can be applied in the case of functionals of the form
\begin{align}
  \label{eq:E-q.leq.1} 
\E_q(\mu)= - \int_\YX \varrho(x)^q\: \d x \quad \text{for } q\in (0,1) \text{
  and } \mu= \varrho  \Ld  + \mu^\perp ,
\end{align}
where $ \mu^\perp $ is singular with respect to $  \Ld $. 
This leads to the \medskip

\noindent
\begin{proof}[Proof of Proposition \ref{pr:q.leq.1}] 
For $\nu_0,\nu_1 \in \calP(X)$ we first observe $\HK^2(\nu_0,\nu_1) \leq
\nu_0(X)+\nu_1(X)$ which implies $\SHK(\nu_0,\nu_1)= 2
\arcsin(\frac12\HK(\nu_0,\nu_1)) \leq 2 \arcsin(\frac12 \sqrt2) = \pi/2 < \pi$. 
 
It is shown in \cite{LiMiSa23FPGG}  that $\E_q$ is
geodesically 0-convex under the following conditions:
\[
q \in [1/3, 1/2] \text{ and } d=1 \qquad \text{or} \qquad 
q=1/2 \text{ and } d=2.
\]
Moreover, we obviously have $\E_q(\mu)\leq 0$ and $\E_q(r^2\varrho \Ld )=\sfF_q([\varrho \Ld ,r])
= r^{2q} \sfF_q([\varrho \Ld ,1])= r^{2q} \E_q(\varrho \Ld )$. Using $q\geq 1/3$ we have
$p$-homogeneity with $p=2q\geq 2/3$. Hence, all assumptions of Proposition
\ref{pr:Transfer} are satisfied, and the geodesic convexity of $\E_q$
restricted to $(\calP(\YX),\SHK)$ follows. 
\end{proof}

\paragraph*{Acknowledgments.} The authors are grateful to Giuseppe Savar\'e for
fruitful and stimulating discussions and for sharing the preprint
\cite{Sava11?GFDS}.  V. Laschos was supported by DFG, through the project EXC-2046/1 within Germany´s Excellence Strategy – The Berlin Mathematics
Research Center MATH+ (project ID: 390685689). A. Mielke was
partially supported by DFG through the project Mie\,459/9-1 within the Priority
Program SPP\,2256 \emph{Variational Methods for Predicting Complex Phenomena in 
Engineering Structures and Materials} (project ID: 441470105).

\footnotesize


\newcommand{\etalchar}[1]{$^{#1}$}
\def\cprime{$'$}
\providecommand{\bysame}{\leavevmode\hbox to3em{\hrulefill}\thinspace}

\end{document}